%% file: main.tex
\documentclass[a4paper,10pt]{article}
\usepackage[british]{babel}
\usepackage{microtype}
\usepackage{amsmath,amssymb,amsfonts,amsthm}
\usepackage{mathtools}
\usepackage{algorithm}
\usepackage{algpseudocode}
\usepackage{bbm}
\usepackage{mathrsfs}
\usepackage[multiple,stable]{footmisc}
\usepackage{enumitem}
\usepackage{csquotes}
\usepackage{multirow}
\usepackage{booktabs}
\usepackage{float}
\usepackage{cancel}

\allowdisplaybreaks[4]
\usepackage[dvipsnames]{xcolor}
\usepackage[colorlinks=true, urlcolor=Thistle,linkcolor=BlueViolet,citecolor=Thistle]{hyperref}
\usepackage{geometry}
\usepackage{listings}
\usepackage[bf]{caption}
\usepackage{graphicx, subfig}
\usepackage{graphics}
\usepackage{caption}
\usepackage{tikz}
\usepackage{pgfplots}
\usepgfplotslibrary{patchplots}
\usetikzlibrary{patterns, positioning, arrows}
\pgfplotsset{compat=1.15}
\usepackage{tikz-3dplot}
\usepackage{tensor}
\usepackage{mdframed}
\usepackage[citestyle=authoryear,bibstyle=authortitle]{biblatex}
\newcommand*{\Var}{\operatorname{Var}}

\newcommand{\beq}{\begin{equation}}
\newcommand{\eeq}{\end{equation}}
\newcommand{\bes}{\begin{equation*}}
\newcommand{\ees}{\end{equation*}}
\newcommand{\bea}{\begin{equation}\begin{aligned}}	
\newcommand{\eea}{\end{aligned}\end{equation}}	
\newcommand{\beas}{\begin{equation*}\begin{aligned}}  
\newcommand{\eeas}{\end{aligned}\end{equation*}}
\newcommand{\indep}{\mathrel{\perp\!\!\!\perp}}
\newmdtheoremenv{theo}{Theorem}

\definecolor{identifiercolor}{rgb}{.4,.6,.56}
\definecolor{stringcolor}{gray}{0.5}
\definecolor{thistle}{rgb}{0.85, 0.75, 0.85}
\definecolor{inactivecolor}{rgb}{0.15,0.15,0.5}
\colorlet{Violet}{black}

\providecommand{\id}{\leavevmode\hbox{\small$\mathrm{1}$\kern-3.8pt\normalsize$\mathrm{1}$}}
\renewcommand{\thefootnote}{\fnsymbol{footnote}}
\newtheorem{Definition}{Definition}

\newtheorem{Theorem}{Theorem}
\newtheorem{Lemma}{Lemma}
\newtheorem{Corollary}{Corollary}
\newtheorem{Proposition}{Proposition}

\newtheorem{Example}{Example}
\newtheorem{Assumption}{Assumption}
\numberwithin{Theorem}{section}
\numberwithin{Definition}{section}
\numberwithin{Lemma}{section}
\numberwithin{Algorithm}{section}
\numberwithin{Proposition}{section}
\numberwithin{Corollary}{section}
\numberwithin{equation}{section}
\numberwithin{Assumption}{section}
\numberwithin{Example}{section}

\let\oldproof\proof
\renewcommand{\proof}{\oldproof}

\algnewcommand\Input{\item[\textbf{Input:}]}
\algnewcommand\Output{\item[\textbf{Output:}]}

\begin{document}

\vspace{1cm}

\title{ {\Large \bf Learning sufficient low-dimensional structures through conditional optimal transport} }

\vspace{2cm}

\author{Kaiqiang Alan Zeng$^{1,\hyperlink{corresponding-email}{\textcolor{blue}{*}}}$ and Efstathia Bura$^{1,\hyperlink{corresponding-email}{\textcolor{blue}{\dagger}}}$}
\date{}
\maketitle

\begin{center}
    {\small $^{1}$\textit{Institute of Statistics and Mathematical Methods in Economics,\\ Faculty of Mathematics and Geoinformation,\\ Technische Universität Wien, Vienna, Austria}}
\end{center}
\footnotetext[1]{\hypertarget{corresponding-email}{\url{kaiqiang.zeng@tuwien.ac.at}}}
\footnotetext[2]{\hypertarget{corresponding-email}{Corresponding author: \url{efstathia.bura@tuwien.ac.at}}}

\rule[-8pt]{420pt}{0.05em}

\begin{abstract}
\noindent

Sufficient dimension reduction seeks a low-dimensional covariate representation that preserves the conditional law of a response. We introduce SDR-COT, which represents that law by conditional optimal transport from an independent reference response. On separable Hilbert spaces, sufficiency forces the response component of the optimal triangular map to factor through the reduction. For quadratic cost, the induced interpolation has a Borel current-state velocity on every truncated time interval, without global injectivity of the terminal map, and this velocity has the same factorisation. These results motivate a conditional-flow-matching criterion. For linear reductions, we prove consistency using a suitably tuned relaxed empirical coupling. Euclidean responses are treated through slicewise Caffarelli bounds; Hilbert-valued responses are treated through Gaussian Sobolev regularity, interpolation compression and uniqueness of a Gaussian continuity equation. Numerical studies with Euclidean and functional data show competitive performance, especially  when sufficient information is not solely contained in the conditional mean.

\end{abstract}

\rule[-8pt]{420pt}{0.05em}

%\tableofcontents

%\thispagestyle{empty}
%\setcounter{page}{0}
%\newpage

\renewcommand{\thefootnote}{\arabic{footnote}}
%--------------------------------------------------------------------------------
%----------------------------Single Agent---------------------------------------
%--------------------------------------------------------------------------------
\section{Introduction}
\label{sec:introduction}

Sufficient dimension reduction (SDR) replaces a covariate $X$ by a lower-dimensional representation without losing information about a response $Y$. In its classical linear form, one seeks a matrix $B$ with few columns such that
\[
    Y\indep X\mid B^{T}X,
\]
where $\indep$ denotes stochastic independence. The column space of a minimal such matrix is the central subspace, and its dimension is the structural dimension. Beginning with sliced inverse regression \parencite{Li1991}, a large literature has developed methods for estimating this subspace; see, for example, \textcite{Cook2007,Li2018}. The defining conditional independence is stronger than a statement about the conditional mean: a sufficient reduction must preserve the entire conditional law of $Y$ given $X$.

Classical SDR methods approach this distributional target indirectly. Inverse-regression procedures, including sliced inverse regression, sliced average variance estimation, directional regression and contour regression, are computationally attractive but typically rely on linearity, coverage or related conditions on the covariate distribution \parencite{Li1991,CookWeisberg1991,Lietal2005,LiWang2007}. Forward-regression methods such as principal Hessian directions, minimum average variance estimation and conditional variance estimation relax some of these requirements, but use selected regression moments or nonparametric smoothing \parencite{Li1992,Xiaetal2002,FertlBura2022}. Such objectives can be highly effective, yet they may overlook sufficient information expressed through heteroscedasticity, tail behaviour, mixtures or multimodality rather than through the particular moments being modelled.

Nonlinear and non-Euclidean SDR methods broaden this picture. Kernel methods extend inverse-regression and support-vector ideas beyond linear projections \parencite{Fukumizuetal2004,li2011principal,lee2013general}, while neural-network methods increase the flexibility of the representation class \parencite{Huang2020,chen2024deep}. Functional SDR treats covariates as Hilbert-space-valued random elements \parencite{ferre2003functional,hsing2009rkhs,lee2022functional}, and nonlinear functional SDR combines basis truncation, kernels, or other finite-dimensional approximations \parencite{li2017nonlinear}. These extensions replace the central subspace by a central class or central $\sigma$-algebra, as formalised by \textcite{lee2013general}. They also make the statistical problem harder: the representation may be nonlinear, the covariate space may be infinite-dimensional, and kernel-based estimators may require spectral decompositions of large Gram matrices.

A recent line of work targets the conditional distribution more directly using generative modelling. \textcite{xu2025conditional} propose GenSDR, which learns nonlinear sufficient representations through conditional stochastic interpolation and proves exhaustiveness guarantees for the central $\sigma$-algebra in a Euclidean framework. \textcite{dong2026flowsdr} propose FlowSDR, a likelihood-based method that jointly learns a linear projection and a conditional density parameterised by normalising flows. These methods are close in spirit to the present paper because they learn reductions by modelling the conditional law itself, rather than only inverse moments or conditional means. The present work takes a different route: it uses conditional optimal transport (COT) as a geometric representation of how the regular conditional response law varies with the covariate.

The connection is already visible in scalar quantile regression. If $U\sim \mathcal{U}([0,1])$, the conditional quantile function $u\mapsto Q_{Y\mid X}(u\mid x)$ transports the reference law of $U$ to the conditional law $\mathbb{P}_{Y\mid X}(\cdot\mid x)$ and is the monotone optimal transport map on the line. For multivariate responses, \textcite{Carlieretal2016} formulate conditional vector quantiles through optimal transport maps; in the scalar case their construction reduces to classical conditional quantiles. This viewpoint connects naturally with quantile-based SDR\@. For example, \textcite{KongXia2014} show that, when $Y\indep X\mid B^{T}X$, gradients of conditional quantile functions lie in the column space of $B$, and they recover the central subspace by aggregating conditional-quantile information across quantile levels. The present paper extends this idea beyond scalar quantiles: conditional transport provides a global description of distributional change. In consequence, an SDR structure should appear as low-dimensional dependence in the transport construction rather than only through local quantile-gradient information. This perspective is suitable for multivariate and Hilbert-valued responses.

We work in a general setting in which both the covariate and the response may take values in separable Hilbert spaces. When distinguishing source and target responses, we write the observed response as $Y_1$. Let $\eta$ denote the law of $X$, let $\rho$ be a reference measure on the response space, and let $Y_0\sim\rho$ be an artificial source response independent of $(X,Y_1)$. The source and target joint laws are then
\[
    \mu_0=\eta\otimes\rho,\qquad \mu_1=\mathbb{P}_{(X,Y_1)},
\]
respectively. Conditional optimal transport seeks a triangular map $T^\star(x,y)=(x,T_{\mathcal{Y}}^\star(x,y))$ that transports $\mu_0$ to $\mu_1$ while preserving the covariate coordinate \parencite{baptista2020conditional,Hosseinietal2025}. If a Borel representation $R$ is sufficient, so that $Y_1\indep X\mid R(X)$, then the conditional law of $Y_1$ given $X$ is, up to null sets, determined by $R(X)$. The corresponding COT problem can therefore be compared with a reduced transport problem indexed by the sufficient representation. The main theoretical question is whether this reduced dependence is inherited by the optimal triangular map and by the Monge-induced velocity used in the dynamic formulation.

Our contribution is fourfold. The first contribution is a Benamou--Brenier type optimality statement for dynamic COT. Existing dynamic COT theory provides conditional Wasserstein geometry and Benamou--Brenier formulations \parencite{Kerriganetal2024,chemseddine2025conditional}; however, \textcite{chemseddine2025conditional} work in a Euclidean setting, whereas \textcite{Kerriganetal2024} treat Hilbert spaces but assume injectivity of the triangular Monge map when constructing a Monge-induced velocity. In the quadratic product-source setting, we prove that the Monge interpolation admits a Borel current-state velocity representative on every truncated interval $[0,1-\tau]$ without assuming global injectivity of the terminal triangular map.

The second contribution is the SDR-COT factorisation theorem. Under a mild regularity condition on the reference response law and the conditional independence condition $Y_1\indep X\mid R(X)$, we prove that the response component of the optimal triangular COT map satisfies
\[
    T_{\mathcal{Y}}^\star(x,y)=G_{\mathcal{Y}}^\star(R(x),y)
\]
up to null sets. That is, the response part of the optimal triangular COT map factors through the reduction and depends on the covariate only through its reduced form. We further show that the same dependence carries over from the static to the dynamic setting: the optimal triangular velocity field generating the Monge interpolation also depends on the covariate only through the sufficient reduction. Thus, the conditional independence assumption underlying SDR becomes a structural property shared by both the static transport map and the time-varying velocity field.

The third contribution is methodological. The factorisation result motivates SDR-COT, a conditional flow-matching training objective in which the response velocity is parameterised by time, a learned reduction of the covariate, and the current response value, rather than by the raw covariate. We develop the population objective, its exhaustiveness property, and both Euclidean and functional algorithms within a unified estimation framework. Because conditional flow matching reduces to a regression problem once endpoint pairs are supplied, the resulting procedure scales efficiently to large samples.

The fourth contribution is statistical, establishing the consistency of linear SDR in both Euclidean and separable Hilbert settings. Because our Euclidean proof relies fundamentally on Caffarelli's contraction theorem \parencite{kolesnikov2011mass}, a result tied strictly to finite dimensions, handling functional responses necessitates a distinct technical toolkit. For Euclidean responses, Caffarelli's theorem provides crucial Lipschitz bounds on the current-state velocity, allowing us to identify the central subspace via a zero-loss condition. Conversely, for Hilbert-valued responses, we bypass the absence of this result by leveraging Cameron–Martin-valued Sobolev velocities and the uniqueness of the Gaussian continuity equation. Under standard assumptions, this alternative route guarantees zero-loss identification and consistency for the function-on-function case, notably without requiring continuity assumptions on the ambient response space.

For general nonlinear representations, our population guarantee remains an exhaustiveness result: every representation whose generated $\sigma$-algebra contains the central $\sigma$-algebra attains the full-information value. As with other exhaustiveness results \parencite{xu2025conditional}, this property alone does not select a minimal representation. In contrast, the linear theory achieves identification by fixing the output dimension at the structural dimension and imposing explicit regularity and correct-specification conditions. %\textcolor{blue}{Finally, numerical experiments provide a proof of concept across Euclidean, functional-covariate, and function-on-function settings. An analysis of the Capital Bikeshare dataset further compares SDR-COT with functional PCA (FPCA) and weak conditional-moment methods, utilising functional temperature as the covariate and functional demand as the response.}

We evaluate SDR-COT in synthetic nonlinear Euclidean, linear Euclidean, functional-covariate, and function-on-function settings, and in a real-data analysis of Capital Bikeshare demand. Across these studies, SDR-COT is competitive with established SDR benchmarks and is especially advantageous when the information distinguishing covariate values is carried by distributional features beyond the conditional mean, the typical target of classical moment-based SDR theory. 
In the functional Bikeshare analysis, SDR-COT's test error is smaller than both functional PCA (FPCA) and weak-conditional-moment-based competitors' \parencite{li2022weakfunctional} across most  reduction dimensions.

The rest of the paper is organised as follows. Section~\ref{sec:preliminaries} introduces the probability and optimal-transport background used throughout. Section~\ref{sec:restricted_dynamic_cot} establishes the dynamic COT construction. Section~\ref{sec:sdr_cot} reviews the measure-theoretic SDR formulation and proves the COT map and velocity factorisation results. Section~\ref{sec:sdr_cot_estimation} develops the population flow-matching criterion, the relaxed empirical coupling, and the Euclidean and functional algorithms. Section~\ref{sec:linear_sdr_cot_statistics} develops the statistical theory for linear SDR-COT in Euclidean and functional settings. Section~\ref{sec:numerics} presents the simulation studies and Bikeshare data analysis; these experiments are intended as a proof of concept rather than a fully optimised benchmarking study. Section~\ref{sec:conclusion} summarises the findings, relates the framework to optimal-transport barycentre problems, and outlines directions for future work. Appendix~\ref{app:proof_sec_pre} contains deferred proofs. Appendix~\ref{app:gaussian_analysis_background} provides a self-contained introduction to Gaussian analysis on Hilbert spaces, which is utilised in the consistency proof for the functional case.

\section{Background and setup}
\label{sec:preliminaries}

This section introduces the probability and optimal-transport notation, definitions, and previously established results used throughout the article. The restricted dynamic result needed for SDR-COT is stated separately in Section~\ref{sec:restricted_dynamic_cot}.

\subsection{Notation}
\label{sec:notation}
Throughout the paper, we primarily work with separable Hilbert spaces and occasionally with separable Banach spaces. Polish spaces appear when an argument needs only the standard Borel machinery of probability theory, such as regular conditional measures, disintegration, kernels, and measurable selection. Every separable Hilbert space is Polish, so these statements apply to the Hilbert-valued setting used in the main results. Some static optimal-transport definitions are meaningful on general Polish spaces, but the results used below are stated in the Hilbert setting. For a Polish space $\mathcal{X}$, we write $\mathcal{B}(\mathcal{X})$ for its Borel $\sigma$-algebra and $\mathcal{P}(\mathcal{X})$ for the set of Borel probability measures on $\mathcal{X}$, equipped with the topology of weak convergence unless stated otherwise. If $\mathcal{X}$ is a normed space and $p\in[1,\infty)$, then $\mathcal{P}^{p}(\mathcal{X})$ denotes the subset of measures with finite $p$-th moment. Given product spaces, $\pi^i$ denotes the $i$th coordinate projection, and $\pi^{i,j}$ denotes the projection onto the $i$th and $j$th coordinates. For a measurable map $T:\mathcal{X}\to\mathcal{Y}$ and a measure $\mu$ on $\mathcal{X}$, $T_{\#}\mu$ denotes the pushforward measure:
\[
    (T_{\#}\mu)(A)\coloneqq\mu(T^{-1}(A)),\qquad A\in\mathcal{B}(\mathcal{Y}).
\]
For a measure $\mu$ and a Banach space $\mathcal{E}$, we write $L^{p}(\mu;\mathcal{E})$ for the corresponding Bochner $L^{p}$ space of $\mathcal{E}$-valued, $p$-integrable functions. When $\mathcal{E}$ is Euclidean, this notation denotes the usual vector-valued $L^p$ space.

Throughout, a superscript $\star$ denotes an optimal object, whereas $A^*$ denotes the adjoint of a linear operator $A$; notation such as $\mathcal{X}^*$ retains its standard meaning as a continuous dual space. This distinction is used consistently for maps, couplings, curves, velocities, and objective values.

Expectations and conditional expectations of Hilbert-valued random elements are understood throughout in the Bochner sense. In particular, if $\mathcal{E}$ is a separable Hilbert space and $Z\in L^1(\mathbb{P};\mathcal{E})$, then
\[
    \mathbb{E}[Z]\coloneqq\int_{\Omega}Z(\omega)\mathbb{P}(d\omega)\in\mathcal{E}.
\]
Equivalently, $\mathbb{E}[Z]$ is the Riesz representative characterised by $\langle\mathbb{E}[Z],h\rangle_{\mathcal{E}}=\mathbb{E}[\langle Z,h\rangle_{\mathcal{E}}]$ for every $h\in\mathcal{E}$. If $Z\in L^2(\mathbb{P};\mathcal{E})$ and $(u\otimes v)h\coloneqq\langle v,h\rangle_{\mathcal{E}}u$, its covariance operator is the Bochner integral
\[
    \Gamma_{Z}\coloneqq\mathbb{E}\bigl[(Z-\mathbb{E}[Z])\otimes(Z-\mathbb{E}[Z])\bigr]\in\mathcal{S}_{1}(\mathcal{E}),
\]
where $\mathcal{S}_{1}(\mathcal{E})$ is the space of trace-class operators on $\mathcal{E}$. Equivalently, $\Gamma_Z$ is the unique positive self-adjoint operator satisfying
\[
    \langle\Gamma_{Z}h,k\rangle_{\mathcal{E}}=\mathbb{E}\bigl[\langle Z-\mathbb{E}[Z],h\rangle_{\mathcal{E}}\langle Z-\mathbb{E}[Z],k\rangle_{\mathcal{E}}\bigr],\qquad h,k\in\mathcal{E},
\]
which is its Riesz-representation characterisation. Thus the Bochner and Riesz constructions agree under the stated moment assumptions; see \parencite[Section 7.2]{hsing2015theoretical}.

If $\mathcal{H}\subseteq\mathcal{F}$ is a sub-$\sigma$-algebra of a probability space $(\Omega,\mathcal{F},\mathbb{P})$, we write
\[
    \overline{\mathcal{H}}^{\mathbb{P}}\coloneqq\left\{ A\in\mathcal{F}:A\mathbin{\triangle}B\subseteq N\text{ for some }B\in\mathcal{H},\ N\in\mathcal{F}\text{ with }\mathbb{P}(N)=0\right\} 
\]
for its $\mathbb{P}$-completion within the ambient $\sigma$-algebra $\mathcal{F}$. We use the analogous notation for completions on a state space, such as $\overline{\sigma(S)}^{\,\eta}$.

\subsection{Regular conditional measures}\label{sec:regularcondmeas}
Throughout the paper, conditional distributions are understood as regular conditional measures. We only use the Polish-space case, where regular conditional measures exist and are unique up to marginal null sets; equivalently, one may view them as disintegrations of probability measures \parencite[Theorem 2.4]{Ambrosioetal2021}. For a comprehensive discussion, see \parencite[Section 10.4]{bogachev2007measure1}.

\begin{Definition}[Regular conditional measures]
\label{def:regular_cond}
Let $\mathcal{X}$ and $\mathcal{Y}$ be Polish spaces. For a set $C\in\mathcal{B}(\mathcal{X}\times \mathcal{Y})$, define the slices $C_{x}\coloneqq\{y\in\mathcal{Y}:(x,y)\in C\}$. For a probability measure $\mu\in\mathcal{P}(\mathcal{X}\times \mathcal{Y})$ with first marginal $\pi^{1}_{\#}\mu=\eta$, there exists a system of regular conditional measures $\{\mu(\cdot\mid x)\}_{x\in\mathcal{X}}\subset\mathcal{P}(\mathcal{Y})$, uniquely determined $\eta$-a.e., such that:
\begin{enumerate}[label=\roman*)]
    \item for every $x\in\mathcal{X}$, the map $A\mapsto\mu(A\mid x)$ is a probability measure on $(\mathcal{Y},\mathcal{B}(\mathcal{Y}))$;
    \item for every $A\in\mathcal{B}(\mathcal{Y})$, the map $x\mapsto \mu(A\mid x)$ is $\mathcal{B}(\mathcal{X})$-measurable;
    \item one has
    \[
       \mu(C)=\int_{\mathcal{X}}\mu(C_{x}\mid x)\eta(dx),\qquad\forall\;C\in\mathcal{B}(\mathcal{X}\times\mathcal{Y}).
    \]
\end{enumerate}
\end{Definition}
We call the kernel
\[
   \mu(\cdot\mid\cdot):\mathcal{B}(\mathcal{Y})\times\mathcal{X}\to[0,1]
\]
a regular conditional measure; it is equivalently a Markov kernel from $\mathcal{X}$ to $\mathcal{Y}$ \parencite[Section 2.3]{wald2025flow}. We use this terminology to distinguish kernels such as $\mu(\cdot\mid x)$ from conditional expectations, whose versions are defined only up to almost-sure equality.

We now specialise this notation to probability measures induced by random elements. Let $(\Omega,\mathcal{F},\mathbb{P})$ be a probability space, let $X:\Omega\to\mathcal{X}$ and $Y:\Omega\to\mathcal{Y}$ be random elements taking values in Polish spaces $\mathcal{X}$ and $\mathcal{Y}$. We write
\[
   \mathbb{P}_{X}:=X_{\#}\mathbb{P},\qquad\mathbb{P}_{Y}:=Y_{\#}\mathbb{P},\qquad\mathbb{P}_{(X,Y)}:=(X,Y)_{\#}\mathbb{P}
\]
for the laws of $X$, $Y$, and $(X,Y)$, respectively. The notation $\mathbb{P}_{Y\mid X}$ denotes the regular conditional measure of $Y$ given $X$, defined by disintegrating the joint law $\mathbb{P}_{(X,Y)}$ with respect to its first marginal $\mathbb{P}_X$. Thus
\[
    \mathbb{P}(X\in B,Y\in A)=\int_{B}\mathbb{P}_{Y\mid X}(A\mid x)\mathbb{P}_{X}(dx),\qquad\forall\;A\in\mathcal{B}(\mathcal{Y}),B\in\mathcal{B}(\mathcal{X}).
\]
It is important not to interpret $\mathbb{P}_{Y\mid X}(\cdot\mid x)$ as the distribution of a random element ``$Y$ given $X=x$'' constructed by conditioning on the event $\{X=x\}$ inside the original probability space. Indeed, if $\mathbb{P}_X(\{x\})=0$, then the event $\{X=x\}$ has probability zero, and the elementary conditional-probability formula does not apply. 

The next lemma states the precise relation between regular conditional measures and ordinary conditional expectations.

\begin{Lemma}
\label{lem:MK_condP}
For every $A\in\mathcal{B}(\mathcal{Y})$, the random element $\omega\mapsto \mathbb{P}_{Y\mid X}(A\mid X(\omega))$ is a version of $\mathbb{P}(Y\in A\mid\sigma(X))$. Equivalently,
\[
    \mathbb{P}_{Y\mid X}(A\mid X(\cdot))=\mathbb{E}\bigl[\mathbbm{1}_{\{Y\in A\}}\mid\sigma(X)\bigr]\qquad\mathbb{P}\text{-a.s.}
\]
\end{Lemma}

\begin{proof}
See Appendix~\ref{app:deferred_proofs}.
\end{proof}

\subsection{A brief summary of optimal transport}
\label{sec:intro_ot}
This section presents the optimal transport facts used later in the paper. The Monge and Kantorovich formulations can be stated on Polish spaces with a lower semicontinuous cost. The dynamic Benamou--Brenier formulation, however, uses a continuity equation and a velocity field, so we state it on separable Hilbert spaces. For a broader treatment of optimal transport, see \textcite{Villani2009,ambrosio2005gradient}.

Let $\mathcal{X}$ be a separable Hilbert space. Given $\mu,\nu\in\mathcal{P}^{p}(\mathcal{X})$, and $p\in [1,\infty)$, the Monge problem is defined as
\begin{equation}
   \inf_{T_{\#}\mu=\nu}\int_{\mathcal{X}}\Vert x-T(x)\Vert^{p}\mu(dx).
\label{eq:ot_monge}
\end{equation}
It admits a convex relaxation known as the Kantorovich problem
\begin{equation}
   W^{p}_{p}(\mu,\nu)\coloneqq\inf_{\gamma\in\Gamma(\mu,\nu)}\int_{\mathcal{X}\times\mathcal{X}}\Vert x-\widetilde{x}\Vert^{p}\gamma(dx,d\widetilde{x}),
\label{eq:ot_kantorovich}
\end{equation}
where 
\[
    \Gamma(\mu,\nu)\coloneqq\{\gamma\in\mathcal{P}^{p}(\mathcal{X}\times\mathcal{X}):\pi^{1}_{\#}\gamma=\mu,\;\pi^{2}_{\#}\gamma=\nu\}
\]
denotes the set of couplings of $\mu$ and $\nu$. We refer to the Monge and Kantorovich problems as static formulations of optimal transport.

The infimum of the Kantorovich problem \eqref{eq:ot_kantorovich} is always attained, and $W_p$ defines a metric on $\mathcal{P}^{p}(\mathcal{X})$ known as the $p$-Wasserstein distance.\footnote{The special case $p=1$ is also called the earth mover's distance or Kantorovich-Rubinstein metric.} Moreover, $W_p$ characterises weak convergence together with convergence of the $p$-th moments \parencite[Theorem 6.9]{Villani2009}. The space $(\mathcal{P}^{p}(\mathcal{X}),W_p)$ is called the $p$-Wasserstein space. The minimisers of \eqref{eq:ot_kantorovich} are called optimal couplings, and we denote the set of minimisers by $\Gamma^\star(\mu,\nu)$. If the minimiser is unique, then we write it as $\gamma^\star$. 

If $\mathcal{X}$ is Euclidean and $p=2$, Brenier's theorem \parencite{Brenier1987} states that absolute continuity of the source measure with respect to Lebesgue measure yields a unique optimal coupling induced by a transport map. For every $p\in(1,\infty)$, the same existence and uniqueness conclusion holds for the strictly convex $p$-power cost; this is also covered by Proposition~\ref{prop:sol_ot} below. Thus
\[
    W_{p}(\mu,\nu)\coloneqq\left(\min_{T_{\#}\mu=\nu}\int_{\mathcal{X}}\Vert x-T(x)\Vert^{p}\mu(dx)\right)^{\frac{1}{p}}.
\]
In other words, the Monge problem \eqref{eq:ot_monge} and the Kantorovich problem \eqref{eq:ot_kantorovich} are equivalent, in the sense that the optimal coupling is uniquely given by $\gamma^{\star}=(\mathrm{id}, T^{\star})_{\#}\mu$, where $T^{\star}$ is the unique optimal transport map solving \eqref{eq:ot_monge}. 

When $\mathcal{X}$ is an infinite-dimensional separable Hilbert space, the Lebesgue measure is not available. Existence and uniqueness of Monge solutions can nevertheless be obtained under a replacement for absolute continuity. We first recall the notions of Gaussian null sets and regular measures \parencite[Definition 6.2.1]{ambrosio2005gradient}.

\begin{Definition}[Gaussian null sets and regular measures]
\label{def:gaussian_null}
Let $\mathcal{X}$ be a separable Banach space with dual $\mathcal{X}^*$. A measure $\lambda\in\mathcal{P}(\mathcal{X})$ is Gaussian if \(\ell_{\#}\lambda\) is a possibly degenerate one-dimensional Gaussian measure for every \(\ell\in\mathcal{X}^*\). It is nondegenerate if \(\ell_{\#}\lambda\) has strictly positive variance for every nonzero \(\ell\in\mathcal{X}^*\).

We say $B\in\mathcal{B}(\mathcal{X})$ is a Gaussian null set if $\lambda(B)=0$ for every nondegenerate Gaussian measure $\lambda\in\mathcal{P}(\mathcal{X})$. Furthermore, we say that a measure $\mu$ is regular if $\mu(B)=0$ for every Gaussian null set $B$. The class of regular measures in $\mathcal{X}$ is denoted by $\mathcal{P}_{r}(\mathcal{X})$. We also write $\mathcal{P}^{p}_{r}(\mathcal{X})\coloneqq\mathcal{P}_{r}(\mathcal{X})\cap\mathcal{P}^{p}(\mathcal{X})$.
\end{Definition}
The class $\mathcal{P}^{p}_{r}(\mathcal{X})$ is non-empty. Indeed, by Fernique's theorem \parencite[Section 2.2.1]{da2014stochastic}, every Gaussian measure on a separable Banach space has finite moments of all orders; hence every nondegenerate Gaussian measure belongs to $\mathcal{P}^{p}_{r}(\mathcal{X})$ for every $p\geq 1$. In Euclidean spaces, Gaussian null sets coincide with Lebesgue null sets. Consequently, a probability measure on $\mathbb{R}^{d}$ is regular in the above sense if and only if it is absolutely continuous with respect to Lebesgue measure. This notion of regularity allows us to state the following existence and uniqueness result for static optimal transport. 

\begin{Proposition}[{\textcite[Theorem 6.2.10]{ambrosio2005gradient}}]
\label{prop:sol_ot}
Let $\mathcal{X}$ be a separable Hilbert space, $\mu \in \mathcal{P}^{p}_{r}(\mathcal{X})$ and $\nu \in \mathcal{P}^{p}(\mathcal{X})$. Then the Kantorovich problem with cost 
\[
    c(x,\widetilde{x})=\Vert x-\widetilde{x}\Vert^{p}_{\mathcal{X}},\qquad p\in(1,\infty)
\]
has a unique solution $\gamma^{\star}=(\mathrm{id}, T^{\star})_{\#}\mu$, where  $T^{\star}:\mathcal{X}\rightarrow\mathcal{X}$ is a Borel map in $L^{p}(\mu;\mathcal{X})$ solving the corresponding Monge problem, and $\mathrm{id}$ denotes the identity mapping. 

If, in addition, $\nu\in\mathcal{P}^{p}_{r}(\mathcal{X})$, then $T^{\star}$ is $\mu$-essentially injective. That is, there exists a Borel set $\Omega\subset \mathcal{X}$ with $\mu(\Omega)=1$ such that the restriction $T^{\star}\mid_{\Omega}:\Omega\to T^{\star}(\Omega)$ is injective.
\end{Proposition}

We now introduce the dynamic formulation of OT, also known as the Benamou--Brenier formulation \parencite{benamou2000computational}. Let $\mathcal{X}$ be a separable Hilbert space. Given $\mu_0,\mu_1\in \mathcal{P}^{p}(\mathcal{X})$ and $p\in (1,\infty)$, the dynamic formulation is
\begin{equation}
   \inf_{(\mu_{t},v_{t})} \int^{1}_{0}\Vert v_{t}\Vert^{p}_{L^{p}(\mu_{t};\mathcal{X})}\,dt,
\label{eq:ot_dynamic}    
\end{equation}
where the minimum is taken over all pairs $(\mu_t,v_t)$ satisfying the following conditions: 
\begin{enumerate}[label=\roman*)]
    \item the curve $t\mapsto \mu_t$ is a narrowly continuous map from $[0,1]$ into $\mathcal{P}^{p}(\mathcal{X})$ that satisfies the boundary conditions $\mu_t\big|_{t=0}=\mu_0$ and $\mu_t\big|_{t=1}=\mu_1$;\footnote{The curve $t\mapsto \mu_t$ is called a narrowly continuous curve if $\mu_{t}\rightharpoonup\mu_{t^{\prime}}$ for $t\to t^{\prime}$.}
    \item the velocity field $v_t:\mathcal{X}\to \mathcal{X}$ is Borel and satisfies $v_{t}\in L^{p}(\mu_{t};\mathcal{X})$, $\int^{1}_{0}\|v_{t}\|^{p}_{L^{p}(\mu_{t};\mathcal{X})}dt<\infty$;
    \item the pair $(\mu_t,v_t)$ fulfils the continuity equation 
    \begin{equation}
       \partial_{t}\mu_{t}+\nabla\cdot(v_{t}\mu_{t})=0.
    \label{eq:continuity_eq}
    \end{equation}
\end{enumerate}
The distributional formulation uses the following standard class of test functions \parencite[Definition 5.1.11]{ambrosio2005gradient}.

\begin{Definition}[Smooth cylindrical test functions]
\label{def:smooth_cylindrical_test_functions}
Let \(\mathcal{E}\) be a separable Hilbert space and let \(T>0\). We denote by \(\operatorname{Cyl}_{c}^{\infty}((0,T)\times\mathcal{E})\) the class of functions of the form
\[
    \varphi(t,z)=\psi\bigl(t,\langle z,e_1\rangle_{\mathcal{E}},\ldots,
    \langle z,e_d\rangle_{\mathcal{E}}\bigr),
\]
where \(d\in\mathbb{N}\), \(e_1,\ldots,e_d\in\mathcal{E}\) are orthonormal, and \(\psi\in C_c^\infty((0,T)\times\mathbb{R}^d)\).
\end{Definition}

The continuity equation \eqref{eq:continuity_eq} is understood in the distributional sense:
\begin{equation}
    \int^{1}_{0}\int_{\mathcal{X}}(\partial_{t}\varphi(t,x)+v_{t}(x)\cdot\nabla_{x}\varphi(t,x))\mu_{t}(dx)dt=0,\qquad\forall\;\varphi\in\operatorname{Cyl}_{c}^{\infty}((0,1)\times\mathcal{X}),
\label{eq:continuity_eq_weak}
\end{equation}

It can be shown that the minimal value of \eqref{eq:ot_dynamic} equals $W^{p}_{p}(\mu_{0},\mu_{1})$, and that an optimal curve $\mu_{t}^{\star}$ is absolutely continuous in $(\mathcal{P}^{p}(\mathcal{X}),W_p)$.\footnote{Meaning that there exists $m\in L^{1}(dt|_{(0,1)};\mathbb{R})$ such that $W_{p}(\mu_{s},\mu_{t})\leq\int^{t}_{s}m(u)\,du$ for $s,t\in(0,1)$.} Its metric derivative\footnote{The metric derivative of a curve $\mu_t$ is defined as 
\[
   \vert\mu^{\prime}\vert(t)\coloneqq\lim_{s\to t}\frac{W_{p}(\mu_{s},\mu_{t})}{\vert s-t\vert},\qquad s,t\in(0,1).
\]
} is given by 
\[
    \vert(\mu^{\star})^{\prime}\vert(t)=W_{p}(\mu_{0},\mu_{1}),\qquad\forall\;t\in(0,1),
\]
i.e., $\mu_{t}^{\star}$ is a constant-speed geodesic in $(\mathcal{P}^{p}(\mathcal{X}),W_p)$ connecting $\mu_0$ and $\mu_1$.\footnote{Recall that a geodesic connecting $\mu_0$ and $\mu_1$ satisfies $W_{p}(\mu_{s}^{\star},\mu_{t}^{\star})=|s-t|W_{p}(\mu_{0},\mu_{1})$ for $s,t\in[0,1]$.} We refer to Section 7 and Theorem 8.3.1 of \textcite{ambrosio2005gradient}, and to \textcite{lisini2007characterization}, for detailed discussions of dynamic OT and geodesics in Wasserstein spaces. 

To summarise, for $p\in(1,\infty)$, the Kantorovich problem \eqref{eq:ot_kantorovich} and the dynamic problem \eqref{eq:ot_dynamic} are equivalent. Under the additional regularity condition that the source measure assigns zero mass to Gaussian null sets, the Monge problem \eqref{eq:ot_monge} is also equivalent to them. If $T^\star$ is the optimal Monge map from $\mu_0$ to $\mu_1$, then the induced coupling and displacement interpolation are
\[
   \gamma^{\star}=(\mathrm{id},T^{\star})_{\#}\mu_{0},\qquad \mu^{\star}_{t}=(T^{\star}_{t})_{\#}\mu_{0},
\]
where $T^{\star}_{t}\coloneqq(1-t)\mathrm{id}+tT^{\star}$. When $T_t^\star$ admits a measurable inverse on a full-measure image, the corresponding velocity field is represented by
\[
    v_t^\star=(T^\star-\mathrm{id})\circ(T_t^\star)^{-1}.
\]
The curve $\mu_t^\star$ is called the displacement, or McCann, interpolation between $\mu_0$ and $\mu_1$.

\subsection{Conditional optimal transport}
\label{sec:intro_cot}
This subsection reviews the conditional optimal transport (COT) framework used later in the SDR derivations. The map-factorisation result developed in Section~\ref{sec:sdr_cot} is static, whereas the algorithmic objective is motivated by dynamic COT.

Static COT, including conditional Monge and conditional Kantorovich formulations, was introduced in \textcite{Carlieretal2016} as a multivariate extension of conditional quantiles. \textcite{Hosseinietal2025} develop a general separable-Hilbert-space treatment of static COT\@. Dynamic COT and the conditional Wasserstein space were developed in parallel by \textcite{chemseddine2025conditional} in the Euclidean setting and by \textcite{Kerriganetal2024} in a Hilbert-space setting.

Connecting a static Monge map to an Eulerian dynamic velocity typically requires an inverse for the interpolated transport maps. In finite-dimensional settings, such invertibility can be obtained under strong regularity assumptions through Caffarelli-type contraction arguments; see, for example, the proof of \textcite[Proposition 17]{chemseddine2025conditional} and \textcite{kolesnikov2011mass}. In the Hilbert-space setting there is no directly analogous regularity theorem, and \textcite{Kerriganetal2024} assume injectivity of the triangular Monge map. Section~\ref{sec:restricted_dynamic_cot} addresses the specific gap needed here for the quadratic cost and a product source law.

As in Section~\ref{sec:intro_ot}, we first introduce the conditional Monge and Kantorovich problems, then the dynamic formulation. For static COT, we use $\mu$ and $\nu$ for source and target measures. We reserve $\mu_0$ and $\mu_1$ for endpoint measures of a time-indexed curve $(\mu_t)_{t\in[0,1]}$, as in displacement interpolation and Benamou--Brenier formulations.

Given a fixed $\eta\in\mathcal{P}(\mathcal{X})$, we define the space of measures sharing a common $\mathcal{X}$-marginal as
\[
   \mathcal{P}^{\eta}(\mathcal{X}\times\mathcal{Y})\coloneqq\{\mu\in\mathcal{P}(\mathcal{X}\times\mathcal{Y}):\pi^{1}_{\#}\mu=\eta\}.
\]
Let $\mu,\nu\in\mathcal{P}^{\eta}(\mathcal{X}\times\mathcal{Y})$, and let $\mu(\cdot\mid x)$ and $\nu(\cdot\mid x)$ denote regular conditional measures with respect to the common $\mathcal{X}$-marginal $\eta$. In static COT, the goal is to find a measurable map $T_{\mathcal{Y}}:\mathcal{X}\times\mathcal{Y}\rightarrow \mathcal{Y}$ such that
\begin{equation}
    T_{\mathcal{Y}}(x,\cdot)_{\#}\mu(\cdot\mid x)=\nu(\cdot\mid x),\qquad\text{for }\eta\text{-a.e. }x\in\mathcal{X}.
\label{eq:slice_pushforward}
\end{equation}
In practice, one rarely has direct access to the conditional measures; instead, one observes samples from the joint law. \textcite{baptista2020conditional} therefore use a triangular map $T:\mathcal{X}\times\mathcal{Y}\rightarrow \mathcal{X}\times\mathcal{Y}$ of the form $T(x,y)=(x,T_{\mathcal{Y}}(x,y))$. If $T_{\#}\mu=\nu$, then the slicewise condition \eqref{eq:slice_pushforward} also holds \parencite{baptista2020conditional,Hosseinietal2025}. In the applications below, the source is a product measure $\mu=\eta \otimes \rho$, where $\rho$ is a simple reference law on the response space.

The conditional Monge problem is
\begin{equation}
\begin{aligned}
    \inf_{T_{\#}\mu=\nu}&\;\int_{\mathcal{X}\times\mathcal{Y}}\Vert(x,y)-T(x,y)\Vert^{p}\mu(dx,dy)\\\textrm{s.t.}&\;T(x,y)=(x,T_{\mathcal{Y}}(x,y)).
\end{aligned}
\label{eq:cond_Monge}
\end{equation}
The corresponding conditional Kantorovich problem is
\begin{equation}
    \inf_{\gamma\in\Gamma_{\eta}(\mu,\nu)}\int_{(\mathcal{X}\times\mathcal{Y})^{2}}\Vert(x,y)-(\widetilde{x},\widetilde{y})\Vert^{p}\gamma(dx,dy,d\widetilde{x},d\widetilde{y}),
\label{eq:cond_Kantorovich}  
\end{equation}
where 
\[
    \Gamma_{\eta}(\mu,\nu)\coloneqq\left\{ \gamma\in\mathcal{P}((\mathcal{X}\times\mathcal{Y})^{2}):\pi^{1,2}_{\#}\gamma=\mu,\pi^{3,4}_{\#}\gamma=\nu,\pi^{1,3}_{\#}\gamma=(\mathrm{id}\times\mathrm{id})_{\#}\eta\right\}.
\]
This is the set of admissible triangular couplings. A triangular coupling $\gamma\in\Gamma_{\eta}(\mu,\nu)$ couples $\mu$ and $\nu$ while forcing the two $\mathcal{X}$-coordinates to agree $\gamma$-a.s. As in the ordinary optimal transport case, the conditional Kantorovich problem leads to the conditional Wasserstein distance \parencite{chemseddine2025conditional,Kerriganetal2024}.

\begin{Definition}[Conditional Wasserstein space and distance]
\label{def:cond_Wp}
Let $\mathcal{X}$ and $\mathcal{Y}$ be separable Hilbert spaces. Given a fixed marginal $\eta\in\mathcal{P}(\mathcal{X})$, for $p\in(1,\infty)$, let $\mathcal{P}^{p,\eta}(\mathcal{X}\times\mathcal{Y})\coloneqq\mathcal{P}^{p}(\mathcal{X}\times\mathcal{Y})\cap\mathcal{P}^{\eta}(\mathcal{X}\times\mathcal{Y})$. For $\mu,\nu\in \mathcal{P}^{p,\eta}(\mathcal{X}\times\mathcal{Y})$, define
\begin{equation}
    W_{p}^{\eta}(\mu,\nu)\coloneqq \left(\int_{\mathcal{X}}W_{p}^{p}\bigl(\mu(\cdot\mid x),\nu(\cdot\mid x)\bigr)\eta(dx)\right)^{\frac{1}{p}},
\label{eq:cond_Wp}
\end{equation}
where $W_{p}$ is the $p$-Wasserstein distance on $\mathcal{P}^{p}(\mathcal{Y})$. We call $W_{p}^{\eta}$ the conditional $p$-Wasserstein distance, and $(\mathcal{P}^{p,\eta}(\mathcal{X}\times\mathcal{Y}),W_{p}^{\eta})$ the conditional Wasserstein space.
\end{Definition}

The conditional $p$-Wasserstein distance equals the minimal conditional Kantorovich cost. For separable Hilbert spaces, the existence and uniqueness results for the conditional Monge and Kantorovich problems are established under suitable assumptions by \textcite{Hosseinietal2025}. We restate the special case of their results where the source is chosen to be a product measure.

\begin{Proposition}[{\textcite[Proposition 3.8]{Hosseinietal2025}}]
\label{prop:condMonge}
Let $\mathcal{X}$ and $\mathcal{Y}$ be separable Hilbert spaces, and take $p\in(1,\infty)$. Let $\mu,\nu\in \mathcal{P}^{p,\eta}(\mathcal{X}\times\mathcal{Y})$, where $\mu=\eta\otimes\rho$, $\rho\in\mathcal{P}^{p}_{r}(\mathcal{Y})$.

Then there exists a jointly measurable triangular map $T^{\star}(x,y)=(x,T^{\star}_{\mathcal{Y}}(x,y))$ such that for $\eta$-a.e.\ $x\in\mathcal{X}$, the map $T^{\star}_{\mathcal{Y}}(x,\cdot)$ is the $\rho$-essentially unique optimal transport map from $\rho$ to $\nu(\cdot\mid x)$ for the cost $c(y,\widetilde{y})=\Vert y-\widetilde{y}\Vert_{\mathcal{Y}}^{p}$. In particular, $T^{\star}$ is the unique solution to the conditional Monge problem \eqref{eq:cond_Monge}, up to $\mu$-a.e.\ equality. Moreover, $\gamma^{\star}\coloneqq(\mathrm{id},T^{\star})_{\#}\mu$ is the unique solution to the conditional Kantorovich problem \eqref{eq:cond_Kantorovich}.
\end{Proposition}

As in OT, the conditional Kantorovich problem \eqref{eq:cond_Kantorovich} is equivalent to the dynamic COT 
\begin{equation}
\begin{aligned}
    \inf_{(\mu_{t},v_{t})}\;&\int^{1}_{0}\Vert v_{t}\Vert^{p}_{L^{p}(\mu_{t};\mathcal{X}\times\mathcal{Y})}dt\\\textrm{s.t.}\;&\partial_{t}\mu_{t}+\nabla\cdot(v_{t}\mu_{t})=0\\&v_{t}(x,y)=(0,v^{\mathcal{Y}}_{t}(x,y))\\&\mu_{t}\big|_{t=0}=\mu_{0},\quad\mu_{t}\big|_{t=1}=\mu_{1}.
\end{aligned}
\label{eq:cond_BB}
\end{equation}
Here the infimum is over narrowly continuous curves $t\mapsto\mu_t$ in $\mathcal{P}^{p,\eta}(\mathcal{X}\times\mathcal{Y})$ and jointly Borel velocity fields with finite action that satisfy the continuity equation in the weak cylindrical sense. A velocity field satisfying $v_{t}(x,y)=(0,v^{\mathcal{Y}}_{t}(x,y))$ is called a triangular velocity field, and it ensures that admissible curves remain in $\mathcal{P}^{p,\eta}(\mathcal{X}\times\mathcal{Y})$.

\section{Monge dynamics for conditional optimal transport}
\label{sec:restricted_dynamic_cot}

Dynamic COT and conditional Wasserstein geometry are established frameworks \parencite{chemseddine2025conditional,Kerriganetal2024}. The result in this section is more specific: it supplies the Monge-induced current-state velocity needed for SDR-COT without assuming global injectivity of the terminal triangular map. In the quadratic product-source setting, cyclic monotonicity instead yields injectivity of each interpolating map on a common full-measure set for every $t<1$. This permits a jointly Borel velocity representative and, on each truncated interval, a Benamou--Brenier optimality statement.

The next theorem gives the dynamic form of the conditional Monge construction needed later. In the quadratic product-source setting, it builds the displacement interpolation and a Borel velocity representative on the full-measure images reached before time $t=1$.

\begin{Theorem} 
\label{thm:disp_interp}
Let $\mathcal{X}$ and $\mathcal{Y}$ be separable Hilbert spaces, and let $\eta\in \mathcal{P}^{2}(\mathcal{X})$. Assume that $\mu_{0},\mu_{1}\in\mathcal{P}^{2,\eta}(\mathcal{X}\times\mathcal{Y})$, where $\mu_0=\eta\otimes \rho$ with $\rho\in\mathcal{P}^{2}_{r}(\mathcal{Y})$.  

Then there exist a Borel map $T_{\mathcal{Y}}^{\star}:\mathcal{X}\times\mathcal{Y}\to\mathcal{Y}$ and a Markov kernel $x\mapsto\gamma_x^{\star}\in\mathcal{P}(\mathcal{Y}\times\mathcal{Y})$ such that, for $\eta$-a.e.\ $x\in\mathcal{X}$,
\[
    \gamma_x^{\star}=(\mathrm{id},T^{\star}_{\mathcal{Y}}(x,\cdot))_{\#}\rho\in\Gamma^{\star}(\rho,\mu_{1}(\cdot\mid x)),
\]
where $T_{\mathcal{Y}}^\star(x,\cdot)$ is the $\rho$-essentially unique optimal transport map from $\rho$ to $\mu_{1}(\cdot\mid x)$ for the quadratic cost $c(y,\widetilde y)=\|y-\widetilde y\|^2$. Define
\[
    T^{\star}(x,y)\coloneqq(x,T^{\star}_{\mathcal{Y}}(x,y)),\qquad\gamma^{\star}\coloneqq(\mathrm{id},T^{\star})_{\#}\mu_{0},
\]
and for $t\in[0,1]$:
\[
    T^{\star}_{t}(x,y)\coloneqq\bigl(x,(1-t)y+tT^{\star}_{\mathcal{Y}}(x,y)\bigr),\qquad\mu^{\star}_{t}\coloneqq(T^{\star}_{t})_{\#}\mu_{0}.
\]
Then:
\begin{enumerate}[label=\roman*)]
    \item The coupling $\gamma^\star\in\Gamma_\eta(\mu_0,\mu_1)$ is the unique conditional Kantorovich optimiser in~\eqref{eq:cond_Kantorovich} with $p=2$, and
the triangular map \(T^\star\) solves the conditional Monge problem~\eqref{eq:cond_Monge}, uniquely
\(\mu_0\)-a.e. 
    \item 
    Let \(G\subset\mathcal{X}\) be  a Borel set, with \(\eta(G)=1\), such that, for every $x\in G$,
    \[
        \gamma_x^{\star}=(\mathrm{id},T^{\star}_{\mathcal{Y}}(x,\cdot))_{\#}\rho\in\Gamma^{\star}(\rho,\mu_{1}(\cdot\mid x)).
    \]
    Define
    \[
    \begin{aligned}
        S&\coloneqq\{(x,y_{0},y_{1})\in\mathcal{X}\times\mathcal{Y}\times\mathcal{Y}:(y_{0},y_{1})\in\operatorname{supp}(\gamma_x^{\star})\},\\
        O&\coloneqq\{(x,y)\in G\times\mathcal{Y}:(x,y,T^{\star}_{\mathcal{Y}}(x,y))\in S\}.
    \end{aligned}
    \]
    Then $S\in\mathcal{B}(\mathcal{X}\times\mathcal{Y}\times\mathcal{Y})$, $O\in\mathcal{B}(\mathcal{X}\times\mathcal{Y})$ and $\mu_0(O)=1$.
    \item For every $t\in[0,1)$, the map $T_t^\star$ is injective on $O$, the set $O_t\coloneqq T_t^\star(O)$ is Borel, and the inverse map $(T_t^\star)^{-1}:O_t\to O$ is Borel. Moreover, \(\mu_t^\star(O_t)=1\).
    \item For every $t\in[0,1)$, define
    \[
        v^{\star}_{t}(x,y)\coloneqq\begin{cases}
(T^{\star}-\mathrm{id})\circ(T^{\star}_{t})^{-1}(x,y), & (x,y)\in O_{t},\\
0, & (x,y)\notin O_{t}. 
\end{cases}
    \]
    Then $v_t^\star$ is Borel. Moreover, the map $(t,x,y)\mapsto v_t^\star(x,y)$ is Borel on $[0,1)\times\mathcal{X}\times\mathcal{Y}$, and
    \[
        v^{\star}_{t}(T^{\star}_{t}(x,y))=(0,T^{\star}_{\mathcal{Y}}(x,y)-y)\qquad\forall\;(x,y)\in O.
    \]
    \item The curve $t\mapsto \mu_t^\star$, $t\in[0,1]$, is absolutely continuous in $(\mathcal{P}^{2,\eta}(\mathcal{X}\times\mathcal{Y}),W^{\eta}_{2})$. 
\end{enumerate}
\end{Theorem}

\begin{proof}[Proof of Theorem~\ref{thm:disp_interp}]
Since $\mathcal{X}$ and $\mathcal{Y}$ are Polish, the disintegration theorem gives a regular conditional measure $x\mapsto\mu_1(\cdot\mid x)\in\mathcal{P}(\mathcal{Y})$ of the second coordinate under $\mu_1$. Let
\[
    A\coloneqq \{x\in\mathcal{X}:\mu_1(\cdot\mid x)\in\mathcal{P}^{2}(\mathcal{Y})\}.
\]
By assumption, $\eta(A)=1$. Moreover, since $x\mapsto\mu_1(\cdot\mid x)$ is measurable and $\mathcal{P}^{2}(\mathcal{Y})$ is a Borel subset of $\mathcal{P}(\mathcal{Y})$, the set $A$ is Borel.

Fix an arbitrary probability measure $\overline\nu\in\mathcal{P}^{2}(\mathcal{Y})$ and define a modified kernel
\[
    \overline{\mu}_{1}(\cdot\mid x)\coloneqq\begin{cases}
\mu_{1}(\cdot\mid x), & x\in A,\\
\overline{\nu}, & x\notin A.
\end{cases}
\]
Then $x\mapsto\overline\mu_1(\cdot\mid x)$ is Borel measurable as a map into $\mathcal{P}(\mathcal{Y})$, takes values in $\mathcal{P}^{2}(\mathcal{Y})$, and agrees with $x\mapsto\mu_1(\cdot\mid x)$ for $\eta$-a.e.\ $x$. Moreover, the second-moment map
\[
    x\mapsto\int_{\mathcal{Y}}\|y\|^{2}\overline{\mu}_{1}(dy\mid x)
\]
is measurable; hence $x\mapsto\overline\mu_1(\cdot\mid x)$ is Borel as a map into $(\mathcal{P}^{2}(\mathcal{Y}),W_2)$. The measure induced by the kernel $\overline\mu_1(\cdot\mid x)$ and the marginal $\eta$ equals $\mu_1$, because the two kernels agree for $\eta$-a.e.\ $x$. Since $\rho\in\mathcal{P}^{2}_{r}(\mathcal{Y})$, Proposition~\ref{prop:condMonge}, applied with source $\mu_0=\eta\otimes\rho$ and target $\mu_1$, yields a Borel triangular map $T^\star(x,y)=(x,T_{\mathcal{Y}}^\star(x,y))$ such that, for $\eta$-a.e.\ $x$, the map $T_{\mathcal{Y}}^\star(x,\cdot)$ is the $\rho$-essentially unique optimal transport map from $\rho$ to $\mu_1(\cdot\mid x)$ for the cost $\Vert y-\widetilde y\Vert^2$. Moreover, $\gamma^\star\coloneqq(\mathrm{id},T^{\star})_{\#}\mu_0$ is the unique solution of the conditional Kantorovich problem.

By measurable selection of optimal couplings \parencite[Corollary 5.22]{Villani2009}; see also \textcite{fontbona2010measurability}, we may choose a Borel kernel
\[
    x\mapsto \gamma_x^{\star}\in \Gamma^\star(\rho,\overline\mu_1(\cdot\mid x)),\qquad x\in\mathcal{X}.
\]
For each $x\in\mathcal{X}$, also define
\[
    \kappa_x\coloneqq\bigl(\mathrm{id},T^{\star}_{\mathcal{Y}}(x,\cdot)\bigr)_{\#}\rho.
\]
We first check that $x\mapsto\kappa_x$ is a measurable kernel. Let $\psi:\mathcal{Y}\times\mathcal{Y}\to\mathbb{R}$ be bounded and Borel. Then
\[
    x\mapsto\int_{\mathcal{Y}\times\mathcal{Y}}\psi(y_{0},y_{1})\kappa_x(dy_{0},dy_{1})=\int_{\mathcal{Y}}\psi\bigl(y,T^{\star}_{\mathcal{Y}}(x,y)\bigr)\rho(dy)
\]
is measurable, because $(x,y)\mapsto\psi(y,T_{\mathcal{Y}}^\star(x,y))$ is Borel. Hence $x\mapsto\kappa_x$ is a measurable kernel. By uniqueness of the optimal coupling in each conditional problem, for $\eta$-a.e.\ $x$,
\[
    \gamma_x^{\star}=\kappa_x=\bigl(\mathrm{id},T^{\star}_{\mathcal{Y}}(x,\cdot)\bigr)_{\#}\rho\in\Gamma^{\star}(\rho,\mu_{1}(\cdot\mid x)).
\]
Since both $x\mapsto\gamma_x^{\star}$ and $x\mapsto\kappa_x$ are Borel maps into the standard Borel space $\mathcal{P}(\mathcal{Y}\times\mathcal{Y})$, the set
\[
    G\coloneqq A\cap\{x\in\mathcal{X}:\gamma_x^{\star}=\kappa_x\}
\]
is Borel and satisfies $\eta(G)=1$.
This proves the asserted existence of $T_{\mathcal{Y}}^\star$ and $\gamma_x^\star$.

The assertion in (i) follows directly from Proposition~\ref{prop:condMonge}, since the modified kernel $\overline\mu_1(\cdot\mid x)$ agrees with $\mu_1(\cdot\mid x)$ for $\eta$-a.e.\ $x$ and therefore induces the same measure $\mu_1$.

We now prove (ii). Since $x\mapsto \gamma_x^\star$ is the chosen measurable version of the family of conditional optimal couplings, it is a measurable kernel from $\mathcal{X}$ to $\mathcal{Y}\times\mathcal{Y}$. Therefore, by Lemma~\ref{lem:Borel_graph} applied with state space $\mathcal{Y}\times\mathcal{Y}$,
\[
   S=\{(x,y_{0},y_{1})\in \mathcal{X}\times\mathcal{Y}\times\mathcal{Y}:(y_{0},y_{1})\in \mathrm{supp}(\gamma_x^\star)\}\in\mathcal{B}(\mathcal{X}\times\mathcal{Y}\times\mathcal{Y}).
\]
Because $G$ is Borel and $T^{\star}_{\mathcal{Y}}$ is Borel, the map $(x,y)\mapsto (x,y,T^{\star}_{\mathcal{Y}}(x,y))$ is Borel, and hence
\[
   O=\{(x,y)\in G\times\mathcal{Y}:(x,y,T^{\star}_{\mathcal{Y}}(x,y))\in S\}\in\mathcal{B}(\mathcal{X}\times\mathcal{Y}).
\]
We next show that $\mu_{0}(O)=1$. For every $x\in G$, the support of $\gamma_x^\star$ is closed, $\gamma_x^\star(\mathrm{supp}(\gamma_x^\star))=1$, and $\gamma_x^\star=(\mathrm{id},T^{\star}_{\mathcal{Y}}(x,\cdot))_{\#}\rho$. Hence
\[
    \rho\bigl(\{y\in\mathcal{Y}:(y,T^{\star}_{\mathcal{Y}}(x,y))\in\mathrm{supp}(\gamma_x^\star)\}\bigr)=1,\qquad x\in G.
\]
Therefore,
\[
    \mu_{0}(O)=\int_{\mathcal{X}}\rho(O_{x})\,\eta(dx)=\eta(G)=1, \qquad O_{x}\coloneqq\{y:(x,y)\in O\}.
\]
This proves (ii).

We now prove (iii). Fix $t\in[0,1)$. If two points of $O$ have different $\mathcal{X}$-coordinates, then their images under $T_t^\star$ have different $\mathcal{X}$-coordinates. It remains to consider $(x,y),(x,\widetilde y)\in O$ with $y\neq \widetilde y$. Then
\[
    (y,T^{\star}_{\mathcal{Y}}(x,y)),\ (\widetilde y,T^{\star}_{\mathcal{Y}}(x,\widetilde y))\in \mathrm{supp}(\gamma_x^\star).
\]
Since $x\in G$, $\gamma_x^\star$ is an optimal coupling for the quadratic cost, and its support is cyclically monotone \parencite[Section 6.2.3]{ambrosio2005gradient}. In particular,
\[
    \langle T^{\star}_{\mathcal{Y}}(x,y)-T^{\star}_{\mathcal{Y}}(x,\widetilde{y}),y-\widetilde{y}\rangle_{\mathcal{Y}}\ge0.
\]
Hence
\[
\begin{aligned}
    &\left\langle (1-t)(y-\widetilde{y})+t\bigl(T^{\star}_{\mathcal{Y}}(x,y)-T^{\star}_{\mathcal{Y}}(x,\widetilde{y})\bigr),y-\widetilde{y}\right\rangle _{\mathcal{Y}}\\
    &\qquad=(1-t)\|y-\widetilde{y}\|^{2}_{\mathcal{Y}}
    +t\langle T^{\star}_{\mathcal{Y}}(x,y)-T^{\star}_{\mathcal{Y}}(x,\widetilde{y}),y-\widetilde{y}\rangle_{\mathcal{Y}}\\
    &\qquad\ge(1-t)\|y-\widetilde{y}\|^{2}_{\mathcal{Y}}>0.
\end{aligned}
\]
Therefore $(1-t)y+tT^{\star}_{\mathcal{Y}}(x,y)\neq(1-t)\widetilde{y}+tT^{\star}_{\mathcal{Y}}(x,\widetilde{y})$, so $T^{\star}_{t}$ is injective on $O$. Note that $O$ is a Borel subset of the Polish space $\mathcal{X}\times\mathcal{Y}$, and $T^{\star}_{t}:O\to\mathcal{X}\times\mathcal{Y}$ is Borel and injective. By the Lusin-Souslin theorem \parencite[Corollary 15.2]{kechris2012classical}, the image $O_{t}=T^{\star}_{t}(O)$ is Borel, and the inverse map $(T_{t}^{\star})^{-1}:O_{t}\to O$ is Borel.
Since $\mu_t^\star=(T_t^\star)_\#\mu_0$ and $\mu_0(O)=1$, it also follows that $\mu_t^\star(O_t)=1$.

For (iv), since $T^{\star}$ and $(T_{t}^{\star})^{-1}$ are Borel on their domains, $v_{t}^{\star}$ is Borel. To prove joint measurability in $(t,x,y)$, define the space-time interpolation map
\[
    \mathcal{T}:[0,1)\times O\to [0,1)\times\mathcal{X}\times\mathcal{Y},\qquad \mathcal{T}(t,z)\coloneqq(t,T_t^\star(z)),
\]
where $z\coloneqq (x,y)$. The map $\mathcal{T}$ is Borel and injective, because each $T_t^\star$ is injective on $O$. By the Lusin-Souslin theorem, the image
\[
    \mathcal{O}\coloneqq \mathcal{T}([0,1)\times O)=\{(t,z):z\in O_t\}
\]
is Borel and the inverse $\mathcal{T}^{-1}:\mathcal{O}\to[0,1)\times O$ is Borel. Let $\pi^O:[0,1)\times O\to O$ denote the projection onto the second coordinate. Hence
\[
    (t,z)\mapsto
    \begin{cases}
        (T^\star-\mathrm{id})(\pi^O(\mathcal{T}^{-1}(t,z))), & (t,z)\in\,\mathcal{O},\\
        0, & (t,z)\notin\mathcal{O},
    \end{cases}
\]
is Borel. This is precisely $(t,z)\mapsto v_t^\star(z)$. Moreover, for every $(x,y)\in O$,
\[
    v^{\star}_{t}(T^{\star}_{t}(x,y))=(T^{\star}-\mathrm{id})(x,y)=(0,T^{\star}_{\mathcal{Y}}(x,y)-y).
\]

Finally, for (v), note that \(\mu_t^\star\in\mathcal{P}^{2,\eta}\).
Write
\[
    \mu_t^\star(\cdot\mid x)\coloneqq\bigl((1-t)\mathrm{id}+tT_{\mathcal{Y}}^\star(x,\cdot)\bigr)_\#\rho.
\]
Then $\mu_t^\star(\cdot\mid x)$ is a regular conditional measure for $\mu_t^\star$ with respect to $\eta$. Let $0\leq s<t\leq 1$. The coupling induced by
\[
   y\mapsto((1-s)y+sT^{\star}_{\mathcal{Y}}(x,y),(1-t)y+tT^{\star}_{\mathcal{Y}}(x,y))
\]
is an admissible coupling between $\mu_s^\star(\cdot\mid x)$ and $\mu_t^\star(\cdot\mid x)$. Hence
\[
\begin{aligned}
    (W^{\eta}_{2}(\mu^{\star}_{s},\mu^{\star}_{t}))^{2}=&\int_{\mathcal{X}}W^{2}_{2}\bigl(\mu^{\star}_{s}(\cdot\mid x),\mu^{\star}_{t}(\cdot\mid x)\bigr)\eta(dx)\\\leq&\int_{\mathcal{X}}\int_{\mathcal{Y}}\Vert\bigl((1-t)y+tT^{\star}_{\mathcal{Y}}(x,y)\bigr)-\bigl((1-s)y+sT^{\star}_{\mathcal{Y}}(x,y)\bigr)\Vert^{2}\rho(dy)\eta(dx)\\=&|t-s|^{2}\int_{\mathcal{X}}\int_{\mathcal{Y}}\Vert T^{\star}_{\mathcal{Y}}(x,y)-y\Vert^{2}\rho(dy)\eta(dx)\\=&|t-s|^{2}(W^{\eta}_{2}(\mu_{0},\mu_{1}))^{2}.
\end{aligned}
\]
The final equality 
is valid because \(T_{\mathcal{Y}}^\star(x,\cdot)\) is the conditional optimal transport map for \(\eta\)-a.e. \(x\). The last integral is finite because $\rho\in\mathcal{P}^{2}(\mathcal{Y})$ and $\mu_1\in\mathcal{P}^{2,\eta}(\mathcal{X}\times\mathcal{Y})$. Therefore $t\mapsto\mu_t^\star$ is absolutely continuous in the conditional Wasserstein space.
\end{proof}

Theorem~\ref{thm:disp_interp} turns the slicewise quadratic conditional Monge map into a displacement interpolation and a Borel current-state velocity on the full-measure images $O_t$, for every $t<1$. This is similar to the conditional McCann-interpolant result of \textcite{Kerriganetal2024}. However, rather than imposing injectivity of the triangular Monge map as an a priori assumption, we establish the required injectivity of $T_t^\star$ on a set of full $\mu_0$-measure. The following corollary gives the corresponding truncated Benamou--Brenier type optimality statement.

\begin{Corollary}
\label{cor:weak_cond_BB}
Assume the hypotheses of Theorem~\ref{thm:disp_interp}. Let $v_t^\star$ be the velocity field from Theorem~\ref{thm:disp_interp}, defined for $t\in[0,1)$. For every $\tau\in(0,1)$, set $T_\tau\coloneqq1-\tau$. Then the restricted pair $(\mu_t^\star,v_t^\star)_{t\in[0,T_\tau]}$ solves the dynamic COT \eqref{eq:cond_BB} on the time interval $[0,T_\tau]$ with endpoint measures $\mu_0$ and $\mu_{T_\tau}^\star$.  
\end{Corollary}

\begin{proof}
Fix $\tau\in(0,1)$ and set $T_\tau\coloneqq1-\tau$. The measure-velocity pair $(\mu^\star_t, v^\star_t)_{t\in[0,T_\tau]}$ is triangular and has endpoints $\mu_0$ and $\mu_{T_\tau}^\star$. To show that it satisfies the continuity equation on $(0,T_\tau)$, let $\varphi\in\operatorname{Cyl}_{c}^{\infty}((0,T_\tau)\times(\mathcal{X}\times\mathcal{Y}))$. Since $\mu_0(O)=1$, the chain rule and Theorem~\ref{thm:disp_interp} give, for $\mu_{0}$-a.e.\ $(x,y)$,
\[
\begin{aligned}
    \frac{d}{dt}\varphi(t,T^{\star}_{t}(x,y))=&\partial_{t}\varphi(t,T^{\star}_{t}(x,y))+\langle\nabla_{x,y}\varphi(t,T^{\star}_{t}(x,y)),\frac{d}{dt}T^{\star}_{t}(x,y)\rangle_{\mathcal{X}\times\mathcal{Y}}\\=&\partial_{t}\varphi(t,T^{\star}_{t}(x,y))+\langle\nabla_{x,y}\varphi(t,T^{\star}_{t}(x,y)),v^{\star}_{t}(T^{\star}_{t}(x,y))\rangle_{\mathcal{X}\times\mathcal{Y}}.
\end{aligned}
\]
Integrating over $(0,T_\tau)\times\mathcal{X}\times\mathcal{Y}$ with respect to $dt\,\mu_0(dx,dy)$, we obtain
\[
\begin{aligned}
    &\int^{T_\tau}_{0}\int_{\mathcal{X}\times\mathcal{Y}}\frac{d}{dt}\varphi(t,T^{\star}_{t}(x,y))\mu_{0}(dx,dy)dt\\
    &\qquad=\int^{T_\tau}_{0}\int_{\mathcal{X}\times\mathcal{Y}}
    \bigl[\partial_{t}\varphi(t,T^{\star}_{t}(x,y))\\
    &\qquad\quad+\langle\nabla_{x,y}\varphi(t,T^{\star}_{t}(x,y)),v^{\star}_{t}(T^{\star}_{t}(x,y))\rangle\bigr]\mu_{0}(dx,dy)dt.
\end{aligned}
\]
Since $\mu^{\star}_{t}=(T^{\star}_{t})_{\#}\mu_{0}$, the right-hand side equals
\[
    \int^{T_\tau}_{0}\int_{\mathcal{X}\times\mathcal{Y}}[\partial_{t}\varphi(t,x,y)+\langle\nabla_{x,y}\varphi(t,x,y),v^{\star}_{t}(x,y)\rangle]\mu^{\star}_{t}(dx,dy)dt.
\]
By the fundamental theorem of calculus, $\varphi$ vanishes outside its compact time support in $(0,T_\tau)$,
\[
    \int^{T_\tau}_{0}\frac{d}{dt}\varphi(t,T^{\star}_{t}(x,y))dt=0,
\]
consequently,
\[
    \int^{T_\tau}_{0}\int_{\mathcal{X}\times\mathcal{Y}}\frac{d}{dt}\varphi(t,T^{\star}_{t}(x,y))\mu_{0}(dx,dy)dt=0.
\]
Therefore,
\[
    \int^{T_\tau}_{0}\int_{\mathcal{X}\times\mathcal{Y}}[\partial_{t}\varphi(t,x,y)+\langle\nabla_{x,y}\varphi(t,x,y),v^{\star}_{t}(x,y)\rangle]\mu^{\star}_{t}(dx,dy)dt=0,
\]
which verifies that $(\mu^{\star}_t, v^{\star}_t)$ satisfies the continuity equation, and hence it is admissible.

For the constructed pair, the action attains the conditional Wasserstein cost on the interval $[0,T_\tau]$. Indeed, since $\mu^{\star}_t=(T^{\star}_t)_{\#}\mu_0$ and $\mu_0(O)=1$, Theorem~\ref{thm:disp_interp} gives, for a.e.\ $t\in(0,T_\tau)$,
\[
\begin{aligned}
    \Vert v^{\star}_{t}\Vert^{2}_{L^{2}(\mu^{\star}_{t};\mathcal{X}\times\mathcal{Y})}&=\int_{\mathcal{X}\times\mathcal{Y}}\Vert v^{\star}_{t}(x,y)\Vert^{2}_{\mathcal{X}\times\mathcal{Y}}\mu^{\star}_{t}(dx,dy)=\int_{\mathcal{X}\times\mathcal{Y}}\Vert v^{\star}_{t}(T^{\star}_{t}(x,y))\Vert^{2}_{\mathcal{X}\times\mathcal{Y}}\mu_{0}(dx,dy)\\&=\int_{\mathcal{X}}\int_{\mathcal{Y}}\Vert T_{\mathcal{Y}}^{\star}(x,y)-y\Vert^{2}_{\mathcal{Y}}\rho(dy)\eta(dx)=\int_{\mathcal{X}}W^{2}_{2}\bigl(\rho,\mu_1(\cdot\mid x)\bigr)\eta(dx)=\left(W^{\eta}_{2}(\mu_{0},\mu_{1})\right)^{2}.
\end{aligned}
\]
Consequently,
\[
    \int^{T_\tau}_{0}\Vert v^{\star}_{t}\Vert^{2}_{L^{2}(\mu^{\star}_{t};\mathcal{X}\times\mathcal{Y})}dt=T_\tau\left(W^{\eta}_{2}(\mu_{0},\mu_{1})\right)^{2}.
\]
Moreover, for $\eta$-a.e.\ $x$, the curve
\[
    \mu_t^\star(\cdot\mid x)=\bigl((1-t)\mathrm{id}+tT_{\mathcal{Y}}^\star(x,\cdot)\bigr)_{\#}\rho
\]
is the displacement interpolation generated by the optimal coupling $(\mathrm{id},T_{\mathcal{Y}}^\star(x,\cdot))_{\#}\rho$ between $\rho$ and $\mu_1(\cdot\mid x)$. Hence it is a constant-speed geodesic in $(\mathcal{P}^{2}(\mathcal{Y}),W_2)$, and
\[
    W_2\bigl(\rho,\mu_{T_\tau}^{\star}(\cdot\mid x)\bigr)=T_\tau W_2\bigl(\rho,\mu_1(\cdot\mid x)\bigr)
    \qquad\text{for }\eta\text{-a.e. }x.
\]
Integrating the squared identity over $\eta$ gives
\[
    \left(W_2^\eta(\mu_0,\mu_{T_\tau}^\star)\right)^2
    =T_\tau^2\int_{\mathcal{X}}W_2^2\bigl(\rho,\mu_1(\cdot\mid x)\bigr)\eta(dx)
    =T_\tau^2\left(W_2^\eta(\mu_0,\mu_1)\right)^2.
\]
Thus the constructed action equals
\[
    \int^{T_\tau}_{0}\Vert v^{\star}_{t}\Vert^{2}_{L^{2}(\mu^{\star}_{t};\mathcal{X}\times\mathcal{Y})}dt
    =\frac{\left(W^{\eta}_{2}(\mu_{0},\mu_{T_\tau}^\star)\right)^2}{T_\tau}.
\]

To show optimality, let $(\widetilde{\mu}_{t},\widetilde{v}_{t})_{t\in[0,T_\tau]}$ be any admissible measure-velocity pair with endpoints $\mu_0$ and $\mu_{T_\tau}^\star$. If its action is infinite, there is nothing to prove. Otherwise, finite action implies $t\mapsto\Vert\widetilde{v}_{t}\Vert_{L^{2}(\widetilde{\mu}_{t};\mathcal{X}\times\mathcal{Y})}$ belongs to $L^{1}(dt|_{(0,T_\tau)};\mathbb{R})$, and the weak continuity equation gives a narrowly continuous representative of $t\mapsto\widetilde\mu_t$. By \textcite[Theorem 4]{Kerriganetal2024} on $(0,T_\tau)$, 
\[
    \vert \widetilde{\mu}^{\prime}\vert(t)\le\Vert\widetilde{v}_{t}\Vert_{L^{2}(\widetilde{\mu}_{t};\mathcal{X}\times\mathcal{Y})}\qquad\text{for a.e. }t\in(0,T_{\tau}).
\]
Therefore, 
\[
    W^{\eta}_{2}(\mu_{0},\mu^{\star}_{T_{\tau}})\leq\int^{T_{\tau}}_{0}\vert\widetilde{\mu}^{\prime}\vert(t)dt\le\int^{T_{\tau}}_{0}\Vert\widetilde{v}_{t}\Vert_{L^{2}(\widetilde{\mu}_{t};\mathcal{X}\times\mathcal{Y})}dt,
\]
and Cauchy-Schwarz gives
\[
    \frac{\left(W^{\eta}_{2}(\mu_{0},\mu_{T_\tau}^\star)\right)^{2}}{T_\tau}\leq\int^{T_\tau}_{0}\Vert\widetilde{v}_{t}\Vert^{2}_{L^{2}(\widetilde{\mu}_{t};\mathcal{X}\times\mathcal{Y})}dt.
\] 
Hence no admissible pair has smaller action than the constructed pair on $[0,T_\tau]$. 
\end{proof}

Corollary~\ref{cor:weak_cond_BB} should be read as a truncated,
measure-theoretic Benamou--Brenier statement along the conditional Monge
transport paths. The truncation to \(T_\tau<1\) ensures that the interpolating
maps are injective on full-measure sets, so that the velocity field admits a
Borel current-state representative. This identifies an optimal velocity field
on its natural \(dt\,\mu_t^\star\)-a.e.\ equivalence class, which is the level of
precision needed for the SDR factorisation below.

\section{Sufficient dimension reduction and COT factorisation}
\label{sec:sdr_cot}

This section first recalls the measure-theoretic formulation of SDR needed for the transport arguments. It then shows that the optimal COT map and its Monge-induced velocity factor through any finite-dimensional representation generating the relevant SDR $\sigma$-algebra. Section~\ref{sec:sdr_cot_estimation} turns these factorisations into population and empirical estimation procedures.

\subsection{General theory of sufficient dimension reduction}
\label{sec:SDRintro}
\textcite{lee2013general}'s measure-theoretic formulation of sufficient dimension reduction introduces the central $\sigma$-algebra, a rigorous object that connects to COT. We adapt this framework below, defining the central $\sigma$-algebra via the forward conditional measure $\mathbb{P}_{Y\mid X}$ and relating it to Lee's original inverse conditional law formulation.

Throughout, for a sub-$\sigma$-algebra $\mathcal{G}\subseteq\sigma(X)$, the notation 
\[
   Y \indep X \mid \mathcal{G}
\]
means that the random elements $X$ and $Y$ are conditionally independent given $\mathcal{G}$. Since $\mathcal{G}\subseteq\sigma(X)$, this condition is equivalent to
\[
    \mathbb{E}[g(Y)\mid\sigma(X)]
    =
    \mathbb{E}[g(Y)\mid\mathcal{G}]
    \qquad\mathbb{P}\text{-a.s.}
\]
for every bounded Borel function $g:\mathcal{Y}\to\mathbb{R}$. In this case, $\mathcal{G}$ is called an SDR $\sigma$-algebra for $Y$ versus $X$. The next proposition states the conditional independence condition in terms of the regular conditional law $\mathbb{P}_{Y\mid X}$.

\begin{Proposition}
\label{prop:sdr_MK}
Assume that \(\mathcal{Y}\) is Polish. Let $\mathbb{P}_{Y\mid X}$ be the regular conditional measure of $Y$ given $X$, and let $\mathcal{G}\subseteq\sigma(X)$. The following are equivalent:
\begin{enumerate}[label=\roman*)]
    \item $Y \indep X \mid \mathcal{G}$;
    \item The $\mathcal{P}(\mathcal{Y})$-valued random element $\omega\mapsto \mathbb{P}_{Y\mid X}(\cdot\mid X(\omega))$ admits a $\mathcal{G}$-measurable version. That is, there exists a \(\mathcal{G}\)-measurable map $\widetilde Z:\Omega\to\mathcal{P}(\mathcal{Y})$ such that $\widetilde Z=Z
\; \mathbb{P}\text{-a.s.}$, where $Z(\omega)=\mathbb{P}_{Y\mid X}(\cdot\mid X(\omega))$.
\end{enumerate}
\end{Proposition}

\begin{proof}
See Appendix~\ref{app:deferred_proofs}.
\end{proof}

An SDR $\sigma$-algebra need not be unique. Indeed, the full $\sigma$-algebra $\sigma(X)$ is always an SDR $\sigma$-algebra, since conditioning on $\sigma(X)$ trivially preserves all information in $X$. The object of interest is therefore the smallest sub-$\sigma$-algebra of $\sigma(X)$ that still retains all information about $Y$ contained in $X$. This motivates the following forward-law definition.

\begin{Definition}[Central $\sigma$-algebra]
\label{def:central_algebra}
Let $\mathbb{P}_{Y\mid X}$ be the regular conditional measure of $Y$ given $X$. The central $\sigma$-algebra for $Y$ versus $X$ is defined by 
\[
    \mathcal{G}_{Y\mid X}\coloneqq\sigma(\omega\mapsto\mathbb{P}_{Y\mid X}(\cdot\mid X(\omega))).
\]
Let $\mathcal{C}\subseteq \mathcal{B}(\mathcal{Y})$ be a countable determining class. Then the central $\sigma$-algebra can be written equivalently as
\[
    \mathcal{G}_{Y\mid X}=\sigma(\omega\mapsto\mathbb{P}_{Y\mid X}(C\mid X(\omega)):C\in\mathcal{C}).
\]
The notation is understood modulo $\mathbb{P}$-null sets: different versions of the regular conditional law generate the same completed $\sigma$-algebra.
\end{Definition}
The following result gives the corresponding minimality statement.

\begin{Proposition}
\label{prop:unique_central_algebra}
The central $\sigma$-algebra $\mathcal{G}_{Y\mid X}\subseteq\sigma(X)$ is an SDR $\sigma$-algebra for $Y$ versus $X$. Moreover, if $\mathcal{G}\subseteq\sigma(X)$ is any SDR $\sigma$-algebra, then
\[
    \mathcal{G}_{Y\mid X}\subseteq\overline{\mathcal{G}}^{\mathbb{P}},
\]
where $\overline{\mathcal{G}}^{\mathbb{P}}$ denotes the $\mathbb{P}$-completion of $\mathcal{G}$. Consequently, $\mathcal{G}_{Y\mid X}$ is the unique minimal SDR $\sigma$-algebra up to $\mathbb{P}$-null sets.
\end{Proposition}

\begin{proof}
The kernel $x\mapsto\mathbb{P}_{Y\mid X}(\cdot\mid x)$ is Borel as a map into $\mathcal{P}(\mathcal{Y})$. Consequently, its composition with $X$ is $\sigma(X)$-measurable, and therefore $\mathcal{G}_{Y\mid X}\subseteq\sigma(X)$. By definition, this composition is also $\mathcal{G}_{Y\mid X}$-measurable. Hence, by Proposition~\ref{prop:sdr_MK},  
\[
    Y\indep X\mid \mathcal{G}_{Y\mid X}.
\]
Thus $\mathcal{G}_{Y\mid X}$ is an SDR $\sigma$-algebra.

Now let $\mathcal{G}\subseteq\sigma(X)$ be any SDR $\sigma$-algebra.
By Proposition~\ref{prop:sdr_MK}, the random probability measure $\mathbb{P}_{Y\mid X}(\cdot\mid X)$ admits a $\mathcal{G}$-measurable version. Therefore
\[
   \sigma\bigl(\omega\to\mathbb{P}_{Y\mid X}(\cdot\mid X(\omega))\bigr)\subseteq\overline{\mathcal{G}}^{\mathbb{P}}.
\]
That is, $\mathcal{G}_{Y\mid X}\subseteq\overline{\mathcal{G}}^{\mathbb{P}}$. Thus, after identifying sub-$\sigma$-algebras with the same $\mathbb{P}$-completion, $\mathcal{G}_{Y\mid X}$ is contained in every SDR $\sigma$-algebra. It is therefore the unique smallest, and hence the unique minimal, SDR $\sigma$-algebra in this sense.
\end{proof}

\paragraph{Relation to the inverse conditional law formulation.}
\textcite{lee2013general} define the central $\sigma$-algebra as the unique minimal SDR $\sigma$-algebra, when such an object exists. Their existence theorem is formulated through the inverse regular conditional distributions $\{\mathbb{P}_{X\mid Y}(\cdot\mid y):y\in\mathcal{Y}\}$ and assumes that this family is dominated by a common $\sigma$-finite measure.

Their proof follows the classical minimal-sufficiency argument. By the Halmos-Savage theorem \parencite[Lemma 2.1]{shao1999mathematical}, the dominating measure may be chosen as a countable convex combination
\[
    \mathbb{Q}=\sum_{k\in\mathbb{N}}c_k\mathbb{P}_{X\mid Y}(\cdot\mid y_k)\quad
    \textrm{where}\quad  c_k>0,\quad \sum_{k\in\mathbb{N}}c_k=1,
\]
so that $\mathbb{P}_{X\mid Y}(\cdot\mid y)\ll\mathbb{Q}$ for every $y\in\mathcal{Y}$.
\textcite{lee2013general} write
\[
    p_y(x)\coloneqq\frac{d\mathbb{P}_{X\mid Y}(\cdot\mid y)}{d\mathbb{Q}}(x),
\]
to characterise SDR $\sigma$-algebras by the essential $\mathcal{G}$-measurability of $p_y$, for every $y$, modulo $\mathbb{Q}$.

The present formulation avoids this inverse-density construction because the transport problem is naturally expressed through the forward conditional law. By Proposition~\ref{prop:sdr_MK}, the SDR condition $Y\indep X\mid \mathcal{G}$ is equivalent to the existence of a $\mathcal{G}$-measurable version of the random conditional law $\mathbb{P}_{Y\mid X}(\cdot\mid X)$. Proposition~\ref{prop:unique_central_algebra} therefore shows that $\mathcal{G}_{Y\mid X}$ is sufficient and minimal up to $\mathbb{P}$-completion. When the domination assumption used by \textcite{lee2013general} also holds, the $\mathbb{P}$-completions of their central $\sigma$-algebra and the present $\mathcal{G}_{Y\mid X}$ coincide, since both are minimal among SDR
\(\sigma\)-algebras. 

Although sub-$\sigma$-algebras provide a clean way to state sufficiency and minimality, they are too abstract to estimate directly from data. For this reason, \textcite{lee2013general} pass from $\sigma$-algebras to estimable classes of functions. We recall the corresponding function-space terminology here. Define
\[
    L^{2}_{0}(\mathbb{P}_{X};\mathbb{R})\coloneqq\left\{ h:\mathcal{X}\rightarrow\mathbb{R}:h\textrm{ is Borel},\;\mathbb{E}[h^{2}(X)]<\infty,\;\mathbb{E}[h(X)]=0\right\}.
\]

\begin{Definition}[SDR class]
For a sub-$\sigma$-algebra $\mathcal{G} \subseteq \sigma(X)$, its SDR class is
\[
    \mathfrak{M}_{\mathcal{G}}\coloneqq\left\{ h\in L^{2}_{0}(\mathbb{P}_{X};\mathbb{R}):h(X)\text{ is }\mathcal{G}\text{-measurable}\right\}.
\]
If $\mathcal{G} = \mathcal{G}_{Y\mid X}$, then $\mathfrak{M}_{\mathcal{G}_{Y\mid X}}$ is called the central class.
\end{Definition}

Thus an SDR class is the function-space counterpart of an SDR $\sigma$-algebra: it records the centred square-integrable functions of $X$ whose information is already contained in $\mathcal{G}$. The central class is the corresponding object for the minimal SDR $\sigma$-algebra. In the framework of \textcite{lee2013general}, the central class is targeted by operator-based estimators such as generalised sliced inverse regression (GSIR) and generalised sliced average variance estimation (GSAVE). These procedures generalise classical sliced inverse regression and sliced average variance estimation via the kernel trick, and their kernel and spectral steps can be computationally demanding for large samples. In the present paper we do not estimate this function class directly. Instead, we assume that the relevant $\sigma$-algebra is generated, up to null sets, by a finite-dimensional Borel map. Accordingly, a Borel map $R:\mathcal{X}\to\mathbb{R}^d$ is called a sufficient representation for $\mathcal{G}$ if
\[
    \overline{\mathcal{G}}^{\mathbb{P}}=\overline{\sigma(R(X))}^{\mathbb{P}}.
\]
The terminology ``sufficient representation'' refers to this finite-generator assumption. It is the nonlinear analogue of the usual finite-dimensional sufficient predictor in linear SDR: when such reductions exist, it is natural to pass between $\sigma$-algebras, classes of functions, and finite-dimensional representations \parencite{lee2013general,Li2018}.

This is also where the present approach departs from the operator-based
estimators developed by \textcite{lee2013general}. Their GSIR and GSAVE
procedures target subspaces of the central class through spectral decompositions
of population operators and their sample analogues, in a framework closely
related to reproducing-kernel approaches to nonlinear SDR
\parencite{Fukumizuetal2004}. In contrast, we do not estimate the central class
through such operators. Instead, we assume that the relevant SDR
\(\sigma\)-algebra admits a finite-dimensional Borel generator and directly
parameterise this generating map. We next show how the optimal transport map
and the corresponding velocity field factor through such a representation.

\subsection{Factorisation of the COT map and velocity}

We now connect conditional optimal transport with the SDR framework. Our goal is to translate the conditional independence assumption underlying SDR into a functional measurability condition that can be exploited by optimal transport maps. The Doob--Dynkin lemma \parencite[Lemma 1.13]{kallenberg1997foundations} plays a central role in the argument; we restate it in Lemma~\ref{lem:doob_dynkin}.

\begin{Lemma}[Doob-Dynkin lemma]
\label{lem:doob_dynkin}
Let $(\Omega,\mathcal{F})$ be a measurable space, and let $\mathcal{S}$ and $\mathcal{E}$ be standard Borel spaces. Let $\xi:\Omega\to \mathcal{S}$ and $\zeta:\Omega\to \mathcal{E}$ be measurable maps. Then $\xi$ is $\sigma(\zeta)$-measurable if and only if there exists a Borel map $h:\mathcal{E}\to \mathcal{S}$ such that $\xi=h\circ\zeta$.
\end{Lemma}

We adopt the standard assumptions of sufficient dimension reduction, adapted for our optimal transport framework.

\begin{Assumption}
\label{ass:cond_indep}
Let $Y_0$, $Y_1$, and $X$ be random elements defined on a common probability space $(\Omega,\mathcal{F},\mathbb{P})$. The variables $Y_0$ and $Y_1$ take values in a separable Hilbert space $\mathcal{Y}$, and $X$ takes values in a separable Hilbert space $\mathcal{X}$.\footnote{
Examples of random elements in separable Hilbert spaces include random variables, random vectors, random matrices, multi-way arrays (informally referred to as tensors in the machine learning community), and functional data (square-integrable stochastic processes). 
} We assume $Y_0 \indep (X, Y_1)$ and $Y_1\indep X\mid\mathcal{G}$, where $\mathcal{G}\subseteq \sigma(X)$.
\end{Assumption}

The following assumption supplies the moment and source regularity conditions used to construct the conditional Monge map and the associated conditional interpolation. The regularity condition is imposed on the artificial source law $\rho$, not on the conditional target laws.
\begin{Assumption}
\label{ass:vanish_gauss_null}
Let $\rho\coloneqq\mathbb{P}_{Y_0}$ and $\eta\coloneqq\mathbb{P}_{X}$. We assume $\rho\in\mathcal{P}^{2}_{r}(\mathcal{Y})$, $\eta\in\mathcal{P}^{2}(\mathcal{X})$, and $\mu_1\coloneqq\mathbb{P}_{(X,Y_1)}\in\mathcal{P}^{2,\eta}(\mathcal{X}\times\mathcal{Y})$. We write $\mu_1(\cdot\mid x)\coloneqq\mathbb{P}_{Y_1\mid X}(\cdot\mid x)$ for the regular conditional measure of $Y_1$ given $X$; the finite second-moment condition implies $\mu_1(\cdot\mid x)\in\mathcal{P}^{2}(\mathcal{Y})$ for $\eta$-a.e.\ $x\in\mathcal{X}$. 
\end{Assumption}

We now connect the general SDR theory on the abstract probability space to the dynamic transport problem on the state space. The SDR $\sigma$-algebra $\mathcal{G}$ lives on $\Omega$, whereas the COT maps are Borel maps on $\mathcal{X}\times\mathcal{Y}$. The link between these two levels is obtained by pulling state-space kernels back through the random element $X$. 
The results below apply to any SDR $\sigma$-algebra $\mathcal{G}$; taking $\mathcal{G}=\mathcal{G}_{Y_1\mid X}$, the central $\sigma$-algebra for $Y_1$ versus $X$, gives the corresponding statements for the central $\sigma$-algebra.

The following assumption states that the relevant SDR $\sigma$-algebra admits a finite-dimensional sufficient representation. This is the form needed by the factorisation and by the neural parameterisation used later.

\begin{Assumption}
\label{ass:sigma_algebra}
There exist $d\geq 1$ and a bounded Borel map $R:\mathcal{X}\to \mathbb{R}^d$ such that $\overline{\mathcal{G}}^{\mathbb{P}}=\overline{\sigma(R(X))}^{\mathbb{P}}$.\footnote{We choose $R$ bounded. This entails no loss of generality at the level of generated $\sigma$-algebras, since composing any finite-dimensional generator coordinatewise with the injective bounded Borel map $r\mapsto\arctan(r)$ preserves the generated $\sigma$-algebra.}
\end{Assumption}

Assumption~\ref{ass:sigma_algebra} is a structural finite-dimensional reduction assumption, not an assumption that the whole central class is a finite-dimensional linear subspace of $L^2(\mathbb{P}_X;\mathbb{R})$. It is the $\sigma$-algebraic version of assuming that the regression relation is captured by a finite sufficient predictor, as is standard in SDR \parencite{lee2013general,Li2018}. The bounded representative ensures $R_{\#}\eta\in\mathcal{P}^{2}(\mathbb{R}^d)$ when we invoke the dynamic COT result on the reduced state space.

The integer $d$ should not be interpreted as an intrinsic dimension of a $\sigma$-algebra: a countably generated standard-Borel information structure can be encoded, up to null sets, by a single real-valued Borel variable. Thus the reduction dimension becomes statistically meaningful only after one restricts the candidate representations to a structured class, such as linear maps, smooth maps, or neural networks with a prescribed output dimension. In the theory below, $R$ is used to express the existence of a finite generator for the relevant SDR $\sigma$-algebra; in the algorithm, the chosen output dimension and network architecture define the estimable class of reductions.

The next lemma is the state-space bridge between conditional independence and conditional transport. It also makes explicit the passage from equality of completed $\sigma$-algebras on $\Omega$ to a Borel response kernel indexed by $R(x)$.

\begin{Lemma}[Conditional-law factorisation]
\label{lem:conditional_law_factorisation}
Under Assumptions~\ref{ass:cond_indep},~\ref{ass:vanish_gauss_null}, and~\ref{ass:sigma_algebra}, there exists a Borel map $\kappa:\mathbb{R}^d\to\mathcal{P}(\mathcal{Y})$ such that $\kappa(s)\in\mathcal{P}^{2}(\mathcal{Y})$ for every $s\in\mathbb{R}^d$ and
\[
    \mathbb{P}_{Y_1\mid X}(\cdot\mid X)=\kappa(R(X))
    \qquad\mathbb{P}\text{-a.s.}
\]
Equivalently,
\[
    \mu_1(\cdot\mid x)=\kappa(R(x))
    \qquad\text{for }\eta\text{-a.e. }x\in\mathcal{X}.
\]
\end{Lemma}

\begin{proof}
By Proposition~\ref{prop:sdr_MK}, there is a $\mathcal{G}$-measurable map $K^{\mathcal{G}}:\Omega\to\mathcal{P}(\mathcal{Y})$ satisfying
\[
    K^{\mathcal{G}}=\mathbb{P}_{Y_1\mid X}(\cdot\mid X)
    \qquad\mathbb{P}\text{-a.s.}
\]
Set $\mathcal{H}\coloneqq\sigma(R(X))$, and let $\{C_n\}_{n\in\mathbb{N}}$ be a countable determining class for $\mathcal{P}(\mathcal{Y})$. For each $n$, the scalar random variable $K^{\mathcal{G}}(C_n)$ is measurable with respect to $\mathcal{G}$ and hence with respect to the common completion
\[
    \overline{\mathcal{G}}^{\mathbb{P}}
    =\overline{\mathcal{H}}^{\mathbb{P}}.
\]
Every real-valued random variable measurable with respect to a completed $\sigma$-algebra has a version measurable with respect to the underlying $\sigma$-algebra. Lemma~\ref{lem:simultaneous_kernel_version}, applied to these countably many versions, therefore yields an $\mathcal{H}$-measurable map $\overline K:\Omega\to\mathcal{P}(\mathcal{Y})$ such that $\overline K=K^{\mathcal{G}}$ $\mathbb{P}$-a.s.

The Doob--Dynkin lemma now gives a Borel map $\kappa_0:\mathbb{R}^d\to\mathcal{P}(\mathcal{Y})$ such that
\[
    \overline K=\kappa_0(R(X)).
\]
Let $\zeta\coloneqq R_\#\eta$. The set
\[
    D_2\coloneqq\{s\in\mathbb{R}^d:\kappa_0(s)\in\mathcal{P}^{2}(\mathcal{Y})\}
\]
is Borel because $\mathcal{P}^{2}(\mathcal{Y})$ is a Borel subset of $\mathcal{P}(\mathcal{Y})$. Assumption~\ref{ass:vanish_gauss_null} and the almost-sure identities above imply $\zeta(D_2)=1$. Define
\[
    \kappa(s)\coloneqq
    \begin{cases}
        \kappa_0(s), & s\in D_2,\\
        \rho, & s\notin D_2.
    \end{cases}
\]
Then $\kappa$ is Borel as a map into $\mathcal{P}(\mathcal{Y})$, takes values in $\mathcal{P}^{2}(\mathcal{Y})$, and satisfies
\[
    \mathbb{P}_{Y_1\mid X}(\cdot\mid X)=\kappa(R(X))
    \qquad\mathbb{P}\text{-a.s.}
\]

Finally, both $x\mapsto\mathbb{P}_{Y_1\mid X}(\cdot\mid x)$ and $x\mapsto\kappa(R(x))$ are Borel maps into the Polish space $\mathcal{P}(\mathcal{Y})$. Hence their equality set is Borel, and the preceding display shows that its $\eta=\mathbb{P}_X$ measure is one. This proves the state-space identity.
\end{proof}

\begin{Proposition}
\label{prop:T_Gmeasurable}
Under Assumptions~\ref{ass:cond_indep},~\ref{ass:vanish_gauss_null}, and~\ref{ass:sigma_algebra}, let $T^\star(x,y)=(x,T_{\mathcal{Y}}^\star(x,y))$ be a Borel representative of the $\mu_0$-a.e.\ unique triangular Monge map pushing $\mu_{0}=\eta\otimes\rho$ to $\mu_{1}=\mathbb{P}_{(X,Y_1)}$ for the quadratic cost. Then there exists a Borel map $G_{\mathcal{Y}}^\star:\mathbb{R}^{d}\times\mathcal{Y}\to\mathcal{Y}$ such that
\[
    T_{\mathcal{Y}}^\star(x,y)=G_{\mathcal{Y}}^\star(R(x),y)
\]
for $\mu_0$-a.e.\ $(x,y)$.
\end{Proposition}

\begin{proof}
Let $\kappa$ be the response kernel from Lemma~\ref{lem:conditional_law_factorisation}, set $\zeta\coloneqq R_{\#}\eta$, and define the measure $\widetilde\nu_1$ on $\mathbb{R}^d\times\mathcal{Y}$ by
\[
    \widetilde\nu_1(ds,dy)\coloneqq\zeta(ds)\kappa(s)(dy).
\]
Since $R$ is bounded and $\mu_1(\cdot\mid x)=\kappa(R(x))$ for $\eta$-a.e.\ $x$, the measure $\widetilde\nu_1$ belongs to $\mathcal{P}^{2,\zeta}(\mathbb{R}^d\times\mathcal{Y})$. Applying Proposition~\ref{prop:condMonge} on the reduced state space with source law $\zeta\otimes\rho$ and target law $\widetilde\nu_1$ gives a Borel triangular map $(s,y)\mapsto(s,G_{\mathcal{Y}}^\star(s,y))$ such that, for $\zeta$-a.e.\ $s$, $G_{\mathcal{Y}}^\star(s,\cdot)$ is the $\rho$-a.e.\ unique optimal map from $\rho$ to $\kappa(s)$.

Proposition~\ref{prop:condMonge}, applied on the original state space with source $\mu_0=\eta\otimes\rho$ and target $\mu_1$, gives that $T_{\mathcal{Y}}^\star(x,\cdot)$ is the $\rho$-a.e.\ unique optimal map from $\rho$ to $\mathbb{P}_{Y_1\mid X}(\cdot\mid x)$ for $\eta$-a.e.\ $x$. Since $\mathbb{P}_{Y_1\mid X}(\cdot\mid x)=\kappa(R(x))$ for $\eta$-a.e.\ $x$, uniqueness gives
\[
    T_{\mathcal{Y}}^\star(x,y)=G_{\mathcal{Y}}^\star(R(x),y)
\]
for $\mu_0$-a.e.\ $(x,y)$.
\end{proof}

Because the optimal velocity in the Benamou--Brenier formulation is explicitly linked to the Monge map, this $\sigma$-algebraic dependence transfers to the dynamics. We state this dependence both along the source parameterisation, using the initial variable $y_0$, and as a function of the current state variable.

\begin{Proposition}
\label{prop:v_Gmeasurable}
Under Assumptions~\ref{ass:cond_indep},~\ref{ass:vanish_gauss_null}, and~\ref{ass:sigma_algebra}, let $T^\star(x,y)=(x,T_{\mathcal{Y}}^\star(x,y))$ be a Borel representative of the triangular Monge map from $\mu_{0}=\eta\otimes\rho$ to $\mu_{1}=\mathbb{P}_{(X,Y_1)}$, and let $G_{\mathcal{Y}}^\star$ be the reduced $\mathcal{Y}$-component map from Proposition~\ref{prop:T_Gmeasurable}. Set\footnote{\(T_t^\star(x,y)=(x,(1-t)y+tT_{\mathcal{Y}}^\star(x,y))\).}
\[
    T_{t}^\star\coloneqq (1-t)\mathrm{id}+tT^\star,\qquad \mu_{t}^\star\coloneqq (T_{t}^\star)_{\#}\mu_{0}.
\]
Then, for every $t\in(0,1)$, the derivative of the source-parameterised interpolation factors through $R$: for $\mu_0$-a.e.\ $(x,y_0)$,
\[
    \dot{T}_{t}^\star(x,y_0)
    =T^\star(x,y_0)-(x,y_0)
    =\bigl(0,G_{\mathcal{Y}}^\star(R(x),y_0)-y_0\bigr).
\]
Moreover, the Monge-induced Borel velocity representative in Theorem~\ref{thm:disp_interp} can be chosen to factor through the current state in the same SDR coordinates. That is, there exists a Borel map $u_{\mathcal{Y}}^\star:[0,1)\times\mathbb{R}^d\times\mathcal{Y}\to\mathcal{Y}$ such that
\[
    v_t^\star(x,y)=(0,u_{\mathcal{Y}}^\star(t,R(x),y))
\]
for $dt$-a.e.\ $t\in[0,1)$ and $\mu_t^\star$-a.e.\ $(x,y)\in\mathcal{X}\times\mathcal{Y}$.
\end{Proposition}

\begin{proof}
By Proposition~\ref{prop:T_Gmeasurable},
\[
    T_{\mathcal{Y}}^\star(x,y_0)=G_{\mathcal{Y}}^\star(R(x),y_0)
\]
for $\mu_0$-a.e.\ $(x,y_0)$, which proves the displayed factorisation of $\dot T_t^\star$.

For the current-state factorisation, let $\zeta\coloneqq R_{\#}\eta$ and define
\[
\nu_0\coloneqq\zeta\otimes\rho,\qquad
    G^\star(s,y_0)\coloneqq(s,G_{\mathcal{Y}}^\star(s,y_0)),\qquad
    \nu_1\coloneqq G^\star_{\#}\nu_0
\]
on $\mathbb{R}^d\times\mathcal{Y}$. By construction, the conditional law of $\nu_1$ given $s$ is $(G_{\mathcal{Y}}^\star(s,\cdot))_\#\rho$ for $\zeta$-a.e.\ $s$.
Since $R$ is bounded, $\zeta\in\mathcal{P}^{2}(\mathbb{R}^d)$. Moreover, $\nu_1\in\mathcal{P}^{2,\zeta}(\mathbb{R}^d\times\mathcal{Y})$, because $G_{\mathcal{Y}}^\star(R(x),y_0)=T_{\mathcal{Y}}^\star(x,y_0)$ for $\mu_0$-a.e.\ $(x,y_0)$ and $T^\star_{\#}\mu_0=\mu_1\in\mathcal{P}^{2,\eta}(\mathcal{X}\times\mathcal{Y})$. By the construction in Proposition~\ref{prop:T_Gmeasurable}, $G_{\mathcal{Y}}^\star(s,\cdot)$ is, for $\zeta$-a.e.\ $s$, the $\rho$-a.e.\ unique quadratic optimal map from $\rho$ to $(G_{\mathcal{Y}}^\star(s,\cdot))_{\#}\rho$. Hence $G^\star$ is the conditional Monge map from $\nu_0$ to $\nu_1$.

Apply Theorem~\ref{thm:disp_interp} to this reduced COT problem. Its Borel velocity representative can be written as $(0,u_{\mathcal{Y}}^\star)$ for a Borel map $u_{\mathcal{Y}}^\star:[0,1)\times\mathbb{R}^d\times\mathcal{Y}\to\mathcal{Y}$ satisfying
\[
    u_{\mathcal{Y}}^\star\bigl(t,s,(1-t)y_0+tG_{\mathcal{Y}}^\star(s,y_0)\bigr)
    =G_{\mathcal{Y}}^\star(s,y_0)-y_0
\]
for $dt\otimes\nu_{0}$-a.e.\ $(t,s,y_0)$. Pulling this identity back by $(x,y_0)\mapsto(R(x),y_0)$ and using the first display gives
\[
    (0,u_{\mathcal{Y}}^\star(t,R(x),T_{t,\mathcal{Y}}^\star(x,y_0)))
    =(0,T_{\mathcal{Y}}^\star(x,y_0)-y_0)
\]
for $dt\otimes\mu_{0}$-a.e.\ $(t,x,y_0)$, where $T_{t,\mathcal{Y}}^\star(x,y_0)=(1-t)y_0+tT_{\mathcal{Y}}^\star(x,y_0)$.

Define
\[
    \widetilde v_t(x,y)\coloneqq(0,u_{\mathcal{Y}}^\star(t,R(x),y)).
\]
The map $(t,x,y)\mapsto\widetilde v_t(x,y)$ is Borel and satisfies
\[
    \widetilde v_t(T_t^\star(x,y_0))=(0,T_{\mathcal{Y}}^\star(x,y_0)-y_0)
\]
for $dt\otimes\mu_0$-a.e.\ $(t,x,y_0)$. The inverse-map representative from Theorem~\ref{thm:disp_interp} satisfies the same identity along the same interpolation paths. Since $\mu_t^\star=(T_t^\star)_{\#}\mu_0$, the two representatives agree for $dt$-a.e.\ $t$ and $\mu_t^\star$-a.e.\ $(x,y)$. Replacing the representative from Theorem~\ref{thm:disp_interp} by $\widetilde v_t$ on this equivalence class gives the asserted current-state factorisation. Thus the velocity field admits a Borel representative that depends on the
current state only through \(R(x)\) and the current \(\mathcal{Y}\)-coordinate.
\end{proof}

Proposition~\ref{prop:v_Gmeasurable} shows that both the time derivative along
the interpolation and a current-state velocity representative depend on the
covariate only through the sufficient representation \(R(x)\). By Corollary~\ref{cor:weak_cond_BB}, this representative is an
optimal triangular velocity representative on every restricted interval
\([0,1-\tau]\), \(\tau\in(0,1)\).

\section{SDR-COT estimation via flow matching}
\label{sec:sdr_cot_estimation}

The preceding factorisation results show that, under the SDR condition, both the terminal response map and the current-state response velocity depend on the covariate only through a sufficient representation. This section turns that structural result into an estimable criterion. We first formulate the population conditional-flow-matching objective and its exhaustiveness property, and then describe the relaxed empirical coupling and the Euclidean and functional implementations used in the numerical experiments.

\subsection{Population flow-matching formulation}
\label{sec:cot_via_flow_matching}
To solve the dynamic COT problem, we use conditional flow matching (CFM), a simulation-free regression method for learning velocity fields from endpoint pairs \parencite{lipman2022flow}. The term ``conditional'' in CFM should not be confused with conditioning on the covariate $X$ in COT; rather, it refers to conditioning on a realised pair of endpoints. Given a coupling $\gamma$, let $(Z_0,Z_1)\sim\gamma$ and define the linear interpolation
\[
   Z_t \coloneqq (1-t)Z_0+tZ_1,\qquad t\in[0,1].
\]
The CFM objective is
\[
    \mathcal{L}_{\mathrm{CFM}}(u)
    =
    \int_0^1
    \mathbb{E}_{\gamma}\bigl[
        \Vert \dot Z_t-u_t(Z_t)\Vert^2
    \bigr]\,dt
    =
    \int_0^1
    \mathbb{E}_{\gamma}\bigl[
        \Vert Z_1-Z_0-u_t(Z_t)\Vert^2
    \bigr]\,dt.
\]
For almost every $t$, this is an ordinary squared-loss regression problem, whose pointwise minimiser is
\[
    u_t(z)
    =
    \mathbb{E}_{\gamma}\bigl[\dot Z_t\mid Z_t=z\bigr]
    =
    \mathbb{E}_{\gamma}\bigl[Z_1-Z_0\mid Z_t=z\bigr].
\]
This is the standard marginalisation argument in CFM: a velocity target specified conditionally on sampled endpoint pairs induces an Eulerian velocity field by averaging endpoint displacements over all pairs whose interpolation reaches the same state. If $\gamma$ is induced by a Monge map and the interpolation is injective on a full-measure set, this conditional mean coincides with the usual current-state velocity. The CFM framework is also used by \textcite{chemseddine2025conditional,Kerriganetal2024} to solve COT.

In what follows, we use only the deterministic COT interpolation. Some conditional-flow constructions introduce Gaussian conditional paths to extend the regression target to neighbourhoods of the interpolation paths \parencite{lipman2024flow}. Such smoothing is not used here, since it would replace the deterministic
Monge-induced velocity by a regularised conditional mean and would no longer
target the exact velocity identified in Proposition~\ref{prop:v_Gmeasurable}.

\paragraph{COT flow matching under the SDR condition.}
Let $\gamma^\star$ be the deterministic triangular optimal coupling from $\mu_0=\eta\otimes\rho$ to $\mu_1=\mathbb{P}_{(X,Y_1)}$, and define $\overline{\gamma}^{\star}\coloneqq\pi^{1,2,4}_{\#}\gamma^\star$. On the canonical coordinates of $\overline{\gamma}^{\star}$, set
\[
    V\coloneqq Y_1-Y_0,\qquad Y_t\coloneqq Y_0+tV=(1-t)Y_0+tY_1.
\]
Since $\gamma^\star$ is triangular, the interpolated state is $(X,Y_t)$ and the endpoint displacement is $(0,V)$. Under Assumptions~\ref{ass:cond_indep},~\ref{ass:vanish_gauss_null}, and~\ref{ass:sigma_algebra}, Proposition~\ref{prop:v_Gmeasurable} shows that the Monge-induced current-state velocity may be represented as
\[
    v_t^\star(x,y)=(0,u_{\mathcal{Y}}^\star(t,R(x),y))
\]
for $dt$-a.e.\ $t\in [0,1)$ and $\mu_t^\star$-a.e.\ $(x,y)$. This motivates the parameterisation
\[
    v^{\theta}(t,x,y)=
    (0,u^{\theta_u}_{\mathcal{Y}}(t,R^{\theta_R}(x),y)).
\]
Fix $\tau\in(0,1)$ and write $T_\tau\coloneqq1-\tau$. The population SDR-CFM loss is
\begin{equation}
    \mathcal{L}(\theta_{u},\theta_{R})=\int^{T_\tau}_{0}\mathbb{E}_{(X,Y_{0},Y_{1})\sim\overline{\gamma}^{\star}}\left[\left\Vert V-u^{\theta_{u}}_{\mathcal{Y}}\bigl(t,R^{\theta_{R}}(X),Y_{t}\bigr)\right\Vert ^{2}_{\mathcal{Y}}\right]dt.
\label{eq:sdr_cot_loss}
\end{equation}
The truncation is consistent with Corollary~\ref{cor:weak_cond_BB}. The loss \eqref{eq:sdr_cot_loss} relies on the optimal population-level coupling $\gamma^\star$, which must be approximated in practice (see Section~\ref{sec:empirical_algorithms}).

This parameterisation should be read as a constrained factorisation of the velocity field, not as direct pointwise estimation of the abstract Borel representative $R$ in Assumption~\ref{ass:sigma_algebra}. For a fixed output dimension and architecture, the optimisation searches over a representation class $\{R^{\theta_R}:\mathcal{X}\to\mathbb{R}^{d}\}$ and a velocity class $\{u^{\theta_u}_{\mathcal{Y}}\}$, and it tries to make the COT displacement predictable from $(R^{\theta_R}(X),Y_t)$. At the unrestricted population level made explicit below, the criterion depends on $R^{\theta_R}$ through the sub-$\sigma$-algebras $\sigma(R^{\theta_R}(X),Y_t)$, $dt$-a.e.\ in $t$, rather than through a unique pointwise representative of $R^{\theta_R}$. Representations that generate the same sub-$\sigma$-algebras up to null sets for $dt$-a.e.\ $t$ therefore have the same unrestricted population value; any stronger identifiability must come from additional restrictions such as output dimension, architecture, or regularisation.

For a fixed $\theta_R$ and almost every $t$, a version of the unrestricted population minimiser of \eqref{eq:sdr_cot_loss} is the Hilbert-space conditional expectation
\[
    u^{\star,\theta_R}_{\mathcal{Y}}(t,s,y)=\mathbb{E}[V \mid R^{\theta_{R}}(X)=s,Y_{t}=y].
\]
Substituting this minimiser back into the population objective gives\footnote{For any $\mathcal{Y}$-valued random element $Z$ and $\sigma$-algebra $\mathcal{F}$, the scalar conditional variance is defined as $\Var(Z \mid \mathcal{F}) \coloneqq \mathbb{E}[\Vert Z - \mathbb{E}[Z \mid \mathcal{F}]\Vert_{\mathcal{Y}}^2 \mid \mathcal{F}]$.}
\[
    \mathcal{L}^\star(\theta_R) = \int_{0}^{T_\tau} \mathbb{E} \left[ \Var \left( V \mid R^{\theta_R}(X), Y_t \right) \right] dt.
\]
Because $\mathcal{Y}$ is a Hilbert space, the Pythagorean theorem in $L^2(\mathbb{P}; \mathcal{Y})$ guarantees that the law of total variance holds. Thus, the unconditional scalar variance of $V$ can be decomposed as:
\[
    \Var(V) = \mathbb{E} \left[ \Var \left( V \mid R^{\theta_R}(X), Y_t \right) \right] + \Var \left( \mathbb{E} \left[ V \mid R^{\theta_R}(X), Y_t \right] \right).
\]
Since $\Var(V) = \mathbb{E}[\Vert V - \mathbb{E}[V]\Vert_{\mathcal{Y}}^2]$ is an intrinsic property of the optimal coupling $\gamma^\star$, minimising the expected conditional variance is mathematically equivalent to maximising the variance of the conditional expectation:
\[
    \sup_{\theta_R} \int_{0}^{T_\tau} \Var \left( \mathbb{E} \left[ V \mid R^{\theta_R}(X), Y_t \right] \right) dt.
\]
Proposition~\ref{prop:population_maximality} makes the population comparison precise: when $R$ generates the central $\sigma$-algebra, representations containing that $\sigma$-algebra attain the full-information value. It does not assert a converse or uniqueness of the learned representation.

The following standard Hilbert-space projection result records the comparison used below; its proof is given in Appendix~\ref{app:deferred_proofs}. For a general treatment of probability and conditional expectations in separable Hilbert spaces, see \parencite[Chapter 1]{da2014stochastic}.
\begin{Lemma}
\label{lem:ineq_condvar}
Let $Y\in L^{2}(\mathbb{P};\mathcal{Y})$, where $\mathcal{Y}$ is a separable Hilbert space. For any sub-$\sigma$-algebras $\mathcal{F}_{1}\subseteq \mathcal{F}_{2} \subseteq \mathcal{F}$, let $Z_{1}\coloneqq\mathbb{E}[Y\mid\mathcal{F}_{1}]$ and $Z_{2}\coloneqq\mathbb{E}[Y\mid\mathcal{F}_{2}]$. Then $\Var(Z_{1})\leq\Var(Z_{2})$, and equality holds if and only if $Z_2=Z_1$ a.s., equivalently if and only if $Z_2$ admits an $\mathcal{F}_1$-measurable version.
\end{Lemma}

\begin{proof}
See Appendix~\ref{app:deferred_proofs}.
\end{proof}

\begin{Proposition}
\label{prop:population_maximality}
Under Assumptions~\ref{ass:cond_indep},~\ref{ass:vanish_gauss_null}, and~\ref{ass:sigma_algebra}, let $R:\mathcal{X}\to\mathbb{R}^{d}$ be the bounded sufficient representation in Assumption~\ref{ass:sigma_algebra}. Let $\gamma^\star=(\mathrm{id},T^\star)_{\#}(\eta\otimes\rho)$ be the optimal triangular coupling from $\mu_0=\eta\otimes\rho$ to $\mu_1=\mathbb{P}_{(X,Y_1)}$, define $\overline{\gamma}^{\star}\coloneqq\pi^{1,2,4}_{\#}\gamma^{\star}$, and fix $\tau\in(0,1)$ with $T_\tau\coloneqq1-\tau$.
On the canonical probability space $(\mathcal{X}\times\mathcal{Y}\times\mathcal{Y},\overline{\gamma}^{\star})$, let $(X,Y_0,Y_1)$ denote the coordinate maps and set
\[
    V\coloneqq Y_1-Y_0.
\]
Then $V=T_{\mathcal{Y}}^\star(X,Y_0)-Y_0$ $\overline{\gamma}^{\star}$-a.s.
For $t\in[0,T_\tau]$, let
\[
    Y_t\coloneqq Y_0+tV.
\]
For any Borel candidate representation $S:\mathcal{X}\to E$ into a standard Borel space, define
\[
    J(S)\coloneqq \int_0^{T_\tau} \Var\!\left(\mathbb{E}[V\mid S(X),Y_t]\right)dt,
    \qquad
    J_X\coloneqq \int_0^{T_\tau} \Var\!\left(\mathbb{E}[V\mid X,Y_t]\right)dt.
\]
Then $J(S)\le J_X$. If
\[
    \sigma(R)\subseteq \overline{\sigma(S)}^{\,\eta},
\]
then $J(S)=J_X$, where \(\sigma(R)=R^{-1}(\mathcal{B}(\mathbb{R}^d))\) and \(\sigma(S)=S^{-1}(\mathcal{B}(E))\). %as sub-$\sigma$-algebras of $(\mathcal{X},\mathcal{B}(\mathcal{X}),\eta)$. 
\end{Proposition}

\begin{proof}
Because $\gamma^\star$ is triangular, $\overline{\gamma}^{\star}$ is the law of $(X,Y_0,T_{\mathcal{Y}}^\star(X,Y_0))$ under $\eta\otimes\rho$. Hence the $(X,Y_0)$-marginal of $\overline{\gamma}^{\star}$ is $\eta\otimes\rho$, $V=T_{\mathcal{Y}}^\star(X,Y_0)-Y_0$ $\overline{\gamma}^{\star}$-a.s., and $V\in L^2(\overline{\gamma}^{\star};\mathcal{Y})$. Standard disintegration on the product space obtained by adjoining the time variable gives conditional-mean versions for $V$ given $(t,S(X),Y_t)$ and $(t,X,Y_t)$; Jensen's inequality then makes the integrands in $J(S)$ and $J_X$ finite and measurable.

Fix $t\in[0,T_\tau]$. Since $S(X)$ is $\sigma(X)$-measurable, we have $\sigma(S(X),Y_t)\subseteq\sigma(X,Y_t)$. Lemma~\ref{lem:ineq_condvar}, applied with $\mathcal{F}_1=\sigma(S(X),Y_t)$ and $\mathcal{F}_2=\sigma(X,Y_t)$ gives
\[
    \Var\!\left(\mathbb{E}[V\mid S(X),Y_t]\right)
    \le
    \Var\!\left(\mathbb{E}[V\mid X,Y_t]\right).
\]
Integrating over $t$ proves $J(S)\le J_X$.

It remains to prove equality under the containment assumption. Proposition~\ref{prop:v_Gmeasurable}, pulled back from the current-state law $\mu_t^\star$ to the source law $\eta\otimes\rho$, gives
\[
    V=u_{\mathcal{Y}}^\star(t,R(X),Y_t)
\]
for $dt\otimes\overline\gamma^\star$-a.e.\ $(t,X,Y_0,Y_1)$. Thus, for almost every $t<T_\tau<1$, the displacement $V$ has a $\sigma(R(X),Y_t)$-measurable version.

Suppose now that $\sigma(R)\subseteq\overline{\sigma(S)}^{\,\eta}$. Since $\mathbb{R}^d$ and $E$ are standard Borel spaces, the completion-version argument followed by the Doob--Dynkin lemma gives a Borel map $h:E\to\mathbb{R}^d$ such that
\[
    R(x)=h(S(x))
    \qquad\text{for }\eta\text{-a.e. }x.
\]
Consequently,
\[
    V=u_{\mathcal{Y}}^\star(t,h(S(X)),Y_t)
\]
for $dt\otimes\overline\gamma^\star$-a.e.\ $(t,X,Y_0,Y_1)$. Hence, for almost every $t$,
\[
    \mathbb{E}[V\mid S(X),Y_t]
    =
    \mathbb{E}[V\mid X,Y_t]
    =V
    \qquad\overline\gamma^\star\text{-a.s.}
\]
Both integrands in $J(S)$ and $J_X$ therefore equal $\Var(V)$ for almost every $t$, which proves $J(S)=J_X$.

Finally,
\[
    \mathbb{E}\!\left[\Var(V\mid S(X),Y_t)\right]
    =
    \Var(V)-\Var\!\left(\mathbb{E}[V\mid S(X),Y_t]\right),
\]
so minimising the population regression loss is equivalent to maximising $J(S)$.
\end{proof}

Proposition~\ref{prop:population_maximality} therefore justifies the population SDR-CFM loss as an exhaustiveness-oriented procedure. When \(\mathcal G=\mathcal G_{Y_1\mid X}\), every Borel representation \(S\) whose generated \(\sigma\)-algebra contains \(\mathcal G_{Y_1\mid X}\), up to completion, retains all predictive information for the deterministic COT displacement target and attains the optimal value \(J_X\). This result does not, however, establish minimality: the population criterion cannot distinguish the central \(\sigma\)-algebra from a strictly richer representation. Our consistency theory consequently focuses on linear SDR with the true structural dimension \(d\) fixed, for which inclusion of sufficient \(d\)-dimensional subspaces implies equality and the target can be measured by the projection distance \(d_{\mathrm{pr}}\). Nonlinear consistency would require an additional mechanism enforcing minimality, an appropriate topology for generated \(\sigma\)-algebras, and complexity control for the fitted representation class. Establishing such a theory is a separate problem that we leave for future work.

\subsection{Empirical estimation and algorithms}
\label{sec:empirical_algorithms}

Let $\mu_1$ be the joint law of $(X,Y_1)$ on $\mathcal{X}\times\mathcal{Y}$. Given i.i.d.\ observations $\{(X_i,Y_{1,i})\}_{i=1}^n\sim\mu_1$, we construct the source sample by retaining the observed covariates and drawing responses independently from a reference law $\rho\in\mathcal{P}^{2}_{r}(\mathcal{Y})$. Thus $Y_{0,i}\stackrel{\mathrm{i.i.d.}}{\sim}\rho$ independently of the target sample, and the population source law is $\mu_0=\eta\otimes\rho$, where $\eta$ is the marginal law of $X$. The empirical endpoint measures are
\[
    \widehat{\mu}_{0}=n^{-1}\sum^{n}_{i=1}\delta_{(X_{i},Y_{0,i})},\qquad\widehat{\mu}_{1}=n^{-1}\sum^{n}_{i=1}\delta_{(X_{i},Y_{1,i})}.
\]

The SDR-CFM loss in Section~\ref{sec:cot_via_flow_matching} requires endpoint pairs from the triangular COT coupling. Since this coupling is not observed, we approximate it by an anisotropic ordinary optimal transport problem with cost
\[
    c_{\varepsilon}(x,y,\widetilde{x},\widetilde{y}):=\Vert x-\widetilde{x}\Vert^{2}_{\mathcal{X}}+\varepsilon\Vert y-\widetilde{y}\Vert^{2}_{\mathcal{Y}}.
\]
For small $\varepsilon>0$, covariate mismatch dominates the transport cost, so the relaxed coupling favours pairs with nearly identical covariates. Such anisotropic relaxations are commonly used to approximate conditional transport \parencite{baptista2020conditional,Hosseinietal2025,chemseddine2025conditional}. The relaxed population optimum need not be unique at a fixed positive \(\varepsilon\). Nevertheless, under uniqueness of the triangular COT coupling, Proposition~\ref{prop:relaxed_empirical_coupling_consistency} gives a deterministic sequence $\varepsilon_n\downarrow0$ along which any measurable choices of corresponding empirical relaxed optima converge to $\gamma^\star$ in $W_2$.

At the sample level, we solve the empirical transport problem between $\widehat{\mu}_0$ and $\widehat{\mu}_1$ with cost $c_\varepsilon$, obtaining $\widehat{\gamma}_{\varepsilon}^\star$. For large samples this step may be replaced by minibatch OT approximations \parencite{tong2023improving}. Samples from $\widehat{\gamma}_{\varepsilon}^\star$ are then used to train the reduction and velocity maps in the empirical version of \eqref{eq:sdr_cot_loss}. Interpolation times are drawn from $\mathcal{U}([0,1-\tau])$, with $\tau>0$, matching the restricted-time dynamic formulation.

The first algorithm treats the Euclidean case $\mathcal{X}=\mathbb{R}^{d_x}$ and $\mathcal{Y}=\mathbb{R}^{d_y}$.

\begin{algorithm}[H]
\caption{Euclidean covariate, Euclidean response}
\label{alg:euclidean_sdr_cot}
\begin{algorithmic}[1]
\Input Target data $\{(X_i,Y_{1,i})\}_{i=1}^n$, source response law $\rho$, target dimension $d$, cost weight $\varepsilon>0$, time truncation $\tau>0$, minibatch size $M$.
\State Draw $Y_{0,i}\stackrel{\mathrm{i.i.d.}}{\sim}\rho$ and form $Z_{0,i}=(X_i,Y_{0,i})$, $Z_{1,i}=(X_i,Y_{1,i})$.
\State Estimate a relaxed empirical COT coupling $\widehat{\gamma}_{\varepsilon}^\star$ with marginals
\[
    \widehat{\mu}_0=n^{-1}\sum_{i=1}^n\delta_{Z_{0,i}},\qquad \widehat{\mu}_1=n^{-1}\sum_{i=1}^n\delta_{Z_{1,i}},
\]
using the cost $c_\varepsilon$.
\State Initialise a reduction network $R_{\theta}:\mathbb{R}^{d_x}\to\mathbb{R}^{d}$ and a triangular velocity network $u^{\phi}_{\mathcal{Y}}:[0,1]\times\mathbb{R}^{d}\times\mathbb{R}^{d_{y}}\to\mathbb{R}^{d_{y}}$.
\For{each stochastic optimisation step}
    \State Draw $M$ endpoint pairs $((X_{0,m},Y_{0,m}),(X_{1,m},Y_{1,m}))$ from $\widehat{\gamma}_{\varepsilon}^\star$ and $t_m\sim\mathcal{U}([0,1-\tau])$, $m=1,\ldots,M$.
    \State Set
    \[
        Y_{t,m}\coloneqq (1-t_m)Y_{0,m}+t_mY_{1,m},\qquad V_{\mathcal{Y},m}\coloneqq Y_{1,m}-Y_{0,m}.
    \]
    \State Update $(\theta,\phi)$ by descending the minibatch loss
    \[
        M^{-1}\sum_{m=1}^{M}\left\Vert V_{\mathcal{Y},m}-u^{\phi}_{\mathcal{Y}}\left(t_{m},R_{\theta}(X_{0,m}),Y_{t,m}\right)\right\Vert ^{2}_{\mathcal{Y}}.
    \]
\EndFor
\Output The fitted reduction map $R_{\widehat{\theta}}$.
\end{algorithmic}
\end{algorithm}

If the sufficient representation is assumed to be linear, the reduction network is replaced by
\[
    R_{B}(x)=B^{T}x,\qquad B\in\mathrm{St}(d_{x},d)\coloneqq\{B\in\mathbb{R}^{d_{x}\times d}:B^{T}B=I_{d}\}.
\]
The optimisation over $\theta$ is then replaced by Riemannian optimisation over the Stiefel manifold, while the velocity parameters $\phi$ remain Euclidean.

For functional covariates with Euclidean responses, let $\mathcal{X}$ be a separable Hilbert space, for example $L^2(ds|_{[0,1]};\mathbb{R})$, and let $\mathcal{Y}=\mathbb{R}^{d_y}$. Given an orthonormal basis $\{e_k\}_{k=1}^{\infty}$, define the projection $P_Kx=\sum_{k=1}^{K}\langle x,e_k\rangle e_k$ and the coefficient map
\[
    \xi_K(x)=\bigl(\langle x,e_1\rangle,\ldots,\langle x,e_K\rangle\bigr)^T\in\mathbb{R}^K.
\]
The functional problem is approximated by applying Algorithm~\ref{alg:euclidean_sdr_cot} to $\xi_K(X_i)$. If the curves are observed on grid points $s_1,\ldots,s_J$, the coefficients may be estimated by quadrature:
\[
    \widehat{\xi}_{ik}=\sum^{J}_{j=1}w_{j}X_{i}(s_{j})e_{k}(s_{j}),\qquad k=1,\ldots,K,
\]
where $w_1,\ldots,w_J$ are integration weights.\footnote{For equally spaced grids, one may use Riemann-sum or trapezium weights.}

The source law $\rho$ remains a law on the Euclidean response space. In the projected problem, the empirical endpoint measures are
\[
    \widehat{\mu}^{K}_{0}=n^{-1}\sum^{n}_{i=1}\delta_{(\widehat{\xi}_{i},Y_{0,i})},\qquad\widehat{\mu}^{K}_{1}=n^{-1}\sum^{n}_{i=1}\delta_{(\widehat{\xi}_{i},Y_{1,i})},
\]
where $\widehat{\xi}_{i}=(\widehat{\xi}_{i1},\ldots,\widehat{\xi}_{iK})^{T}$ and $Y_{0,i}\sim\rho$. 

\begin{algorithm}[H]
\caption{Functional covariate, Euclidean response}
\label{alg:functional_sdr_cot_projection}
\begin{algorithmic}[1]
\Input Functional target data $\{(X_i,Y_{1,i})\}_{i=1}^n$, where $X_i=(X_i(s_1),\ldots,X_i(s_J))$, basis functions $\{e_k\}_{k=1}^K$, quadrature weights $\{w_j\}_{j=1}^J$, source response law $\rho$ on $\mathbb{R}^{d_y}$, target dimension $d$, cost weight $\varepsilon>0$, time truncation $\tau>0$.
\State Estimate the first $K$ basis coefficients of each functional covariate:
\[
    \widehat{\xi}_{ik}=\sum^{J}_{j=1}w_{j}X_{i}(s_{j})e_{k}(s_{j}),\qquad k=1,\ldots,K,
\]
and set $\widehat{\xi}_{i}=(\widehat{\xi}_{i1},\ldots,\widehat{\xi}_{iK})^{T}\in\mathbb{R}^{K}$.
\State Run Algorithm~\ref{alg:euclidean_sdr_cot} on the projected samples $(\widehat{\xi}_i,Y_{1,i})$.
\Output The fitted projected functional representation
\[
    \widehat{R}_{K}(X)=R_{\widehat{\theta}}(\widehat{\xi}_{K}(X)).
\]
\end{algorithmic}
\end{algorithm}

This procedure estimates the best representation available after the chosen finite-dimensional projection. Its error therefore combines the Euclidean SDR-COT estimation error with the projection error from replacing $X$ by $P_KX$.

When both the covariate and the response are functional, we still use a finite-dimensional projection for the covariate, but the response velocity is treated as a function-valued object. Let $\mathcal{X}=L^2([0,1])$ and $\mathcal{Y}=L^2(D;\mathbb{R}^{d_y})$, with the response observed on grid points $r_1,\ldots,r_L\in D$ and quadrature weights $v_1,\ldots,v_L$. The source response law $\rho$ is chosen as a nondegenerate centred Gaussian process law on $\mathcal{Y}$ and is sampled on the response grid. The relaxed empirical COT cost combines a Euclidean cost on the projected covariate coefficients and a quadrature approximation of the response Hilbert norm.

The conditional-flow-matching loss remains the deterministic endpoint-regression loss, but the velocity now maps a current response function to another response function, as in functional flow matching \parencite{KerriganMiglioriniSmyth2023FFM}. This velocity can be parameterised by a neural operator. In the experiments below we use a Fourier neural operator (FNO), which represents operator-valued maps through spectral convolutions and pointwise nonlinearities \parencite{Kovachkietal2021NeuralOperator,LiEtal2020FNO}.  

\begin{algorithm}[H]
\caption{Functional covariate, functional response}
\label{alg:functional_functional_sdr_cot}
\begin{algorithmic}[1]
\Input Functional target data $\{(X_i,Y_{1,i})\}_{i=1}^n$, with $X_i$ observed on $s_1,\ldots,s_J$ and $Y_{1,i}$ observed on $r_1,\ldots,r_L$, basis functions $\{e_k\}_{k=1}^K$, quadrature weights $\{w_j\}_{j=1}^J$ and $\{v_\ell\}_{\ell=1}^L$, Gaussian process source law $\rho$ on $\mathcal{Y}$, target dimension $d$, cost weight $\varepsilon>0$, time truncation $\tau>0$, minibatch size $M$.
\State Estimate the first $K$ covariate coefficients
\[
    \widehat{\xi}_{ik}=\sum_{j=1}^{J}w_jX_i(s_j)e_k(s_j),\qquad k=1,\ldots,K,
\]
and set $\widehat{\xi}_i=(\widehat{\xi}_{i1},\ldots,\widehat{\xi}_{iK})^T$.
\State Draw $Y_{0,i}\stackrel{\mathrm{i.i.d.}}{\sim}\rho$ on the response grid and form $Z_{0,i}=(\widehat{\xi}_i,Y_{0,i})$, $Z_{1,i}=(\widehat{\xi}_i,Y_{1,i})$.
\State Estimate a relaxed empirical COT coupling $\widehat{\gamma}_{\varepsilon}^{\star}$ between $\widehat\mu_0=n^{-1}\sum_i\delta_{Z_{0,i}}$ and $\widehat\mu_1=n^{-1}\sum_i\delta_{Z_{1,i}}$ using the cost
\[
    \|\xi-\widetilde\xi\|_{\mathbb{R}^K}^2
    +\varepsilon\sum_{\ell=1}^{L}v_\ell\|y(r_\ell)-\widetilde y(r_\ell)\|_{\mathbb{R}^{d_y}}^2.
\]
\State Initialise a linear reduction $R_B(X)=B^T\widehat{\xi}_K(X)$, $B\in\mathrm{St}(K,d)$, and an FNO velocity $U_\phi:[0,1]\times\mathbb{R}^d\times\mathcal{Y}\to\mathcal{Y}$.
\For{each stochastic optimisation step}
    \State Draw $M$ endpoint pairs $((\widehat{\xi}_{0,m},Y_{0,m}),(\widehat{\xi}_{1,m},Y_{1,m}))$ from $\widehat{\gamma}_{\varepsilon}^{\star}$ and $t_m\sim\mathcal{U}([0,1-\tau])$, $m=1,\ldots,M$.
    \State Set $Y_{t,m}=(1-t_m)Y_{0,m}+t_mY_{1,m}$ and $V_{\mathcal{Y},m}=Y_{1,m}-Y_{0,m}$.
    \State Update $(B,\phi)$ by descending
    \[
        M^{-1}\sum_{m=1}^{M}\sum_{\ell=1}^{L}v_\ell
        \left\|V_{\mathcal{Y},m}(r_\ell)
        -U_\phi\bigl(t_m,B^T\widehat{\xi}_{0,m},Y_{t,m}\bigr)(r_\ell)
        \right\|_{\mathbb{R}^{d_y}}^2.
    \]
\EndFor
\Output The fitted projected functional representation $\widehat R_K(X)=\widehat B^T\widehat{\xi}_K(X)$.
\end{algorithmic}
\end{algorithm}

\input{statistical_properties.tex}

\section{Numerical experiments}
\label{sec:numerics}

This section evaluates the algorithms in Section~\ref{sec:empirical_algorithms} through simulation studies and a function-on-function analysis of Capital Bikeshare data. In particular, the linear Euclidean and functional implementations correspond to the estimators studied in Section~\ref{sec:linear_sdr_cot_statistics}, subject to the empirical-criterion and optimisation conditions stated there. The experiments are intended as proof-of-concept studies: the aim is to assess whether the SDR-COT objective recovers known sufficient predictors and gives competitive out-of-sample prediction, rather than to optimise running time or provide an exhaustive benchmark.

\input{simulation_studies.tex}
\input{real_data_analysis.tex}

\section{Conclusion}
\label{sec:conclusion}

We have formulated sufficient dimension reduction through conditional optimal transport. A sufficient representation determines the regular conditional response law, and this dependence is inherited by both the response component of the triangular Monge map and its Monge-induced velocity. In the quadratic product-source setting, the current-state velocity exists as a Borel representative on every truncated time interval without assuming global injectivity of the terminal map. The SDR conditional independence condition is therefore expressed as a common structural restriction on the static transport and its dynamics.

%This structure leads to a conditional-flow-matching estimator. For general representations, the population criterion is exhaustiveness-oriented: every representation containing the central $\sigma$-algebra attains the full-information value. For correct-dimensional linear reductions, the statistical theory converts this no-loss property into identification and consistency. At a fixed positive relaxation parameter, the population transport optimum need not be unique; nevertheless, the entire optimiser set converges to the unique triangular COT coupling as the parameter vanishes, and a sufficiently slow deterministic tuning sequence gives empirical $W_2$-consistency. For Euclidean responses, slicewise Caffarelli bounds and a uniform empirical-criterion condition yield central-subspace consistency and a projection-error bound. For Hilbert-valued responses, the corresponding fixed- and growing-projection results use Cameron--Martin Sobolev regularity, Gaussian divergence control, $L^r$ interpolation compression, and explicit empirical-process and sieve conditions, without ambient response continuity.

This structure leads to a conditional-flow-matching estimator. For general representations, the population criterion is exhaustiveness-oriented: every representation containing the central $\sigma$-algebra attains the full-information value. For linear reductions, the statistical theory converts this zero-loss property into identification and consistency. Although the population transport optimum need not be unique at a fixed positive relaxation parameter, the entire optimiser set converges to the unique triangular COT coupling as the parameter vanishes, and a sufficiently slow deterministic tuning sequence gives empirical $W_2$-consistency. For Euclidean responses, slicewise Caffarelli bounds and uniform convergence of the empirical criterion yield central-subspace consistency. For Hilbert-valued responses, the corresponding fixed- and growing-projection results instead use Cameron--Martin Sobolev regularity, Gaussian divergence control, $L^r$ interpolation compression, and explicit empirical-process and sieve conditions, without requiring ambient response continuity.

The restriction of this consistency theory to correctly specified linear reductions is substantive. Proposition~\ref{prop:population_maximality} shows that the population criterion is constant across all representations containing the central $\sigma$-algebra and therefore cannot by itself select the minimal one. In the linear theory, fixing the true structural dimension turns containment of sufficient subspaces into equality, while the Euclidean and fixed-projection arguments use compact finite-dimensional Stiefel parameter spaces and the projection distance $d_{\mathrm{pr}}$; these devices have no direct analogue for arbitrary Borel or neural-network representations. Indeed, injective transformations can change the parametrisation without changing the generated $\sigma$-algebra, and a fixed output dimension does not enforce informational minimality outside the linear setting. A nonlinear consistency theory would consequently require a minimality-inducing restriction or penalty, and an appropriate topology for generated $\sigma$-algebras. These requirements do not follow from the transport regularity arguments developed here and constitute a separate problem, rather than an incremental extension. We view this as a distinct future research direction opened up by this framework.

The numerical results complement the theory. SDR-COT is competitive in the nonlinear and linear Euclidean experiments, and its strongest relative performance occurs in models where the relevant information is carried by conditional scale or by several sufficient indices. It also performs well with functional covariates and responses. In the Capital Bikeshare analysis, its test errors are comparable to those of FPCA and weak conditional-moment methods across the reduction dimensions considered, with the smallest SDR-COT test MSE at dimension three.

The framework also clarifies a connection with optimal transport barycentre problems. The Monge optimal transport barycentre problem of \textcite{Lipnicketal2025} transports a family of conditional measures to a barycentre that removes a chosen source of variation. SDR-COT operates in the reverse direction: it transports a fixed reference response law to the regular conditional laws of the response and seeks a low-dimensional index for this family. In this respect, the invariant feature extraction method of \textcite{BounosGroismanSuedTabak2026} is closely related in spirit, although its invariance objective is distinct from the recovery of a central SDR representation.

%\textcolor{magenta}{The linear-reduction consistency theory rests on a coordinate-free metric available only because subspaces form a compact Grassmannian; extending it to general representations would require a new metric on the $\sigma$-algebra lattice and new complexity-control arguments, which we view as a distinct research direction opened up by this framework rather than a limitation of it.}

Several further questions remain open. The diagonal tuning argument gives no explicit sample-size rate, and coupling convergence alone does not imply uniform convergence for the fitted velocity classes. Rates and distributional limits for the estimated central subspace will require quantitative COT stability, verifiable empirical-process and margin conditions, and complexity control for the velocity model. Functional extensions must additionally account for basis approximation, interpolation compression, and infinite-dimensional regularity. Scalable COT solvers, stable minibatch couplings, and entropic or Schrödinger-bridge analogues of the map and velocity factorisations are further natural directions.

\section*{Acknowledgements}
The authors used ChatGPT and Claude to assist in checking the internal consistency of the mathematical proofs and in reviewing and improving the accompanying implementation code. All AI-generated suggestions were independently evaluated by the authors. The authors take full responsibility for the correctness of the proofs, the validity of the results, and the accuracy and reproducibility of the code.

\printbibliography

\newpage
\appendix

\section{Auxiliary lemmas and deferred proofs}
\label{app:proof_sec_pre}

\subsection{Auxiliary measurability lemmas}

\begin{Lemma}
\label{lem:Borel_graph}
Let $\mathcal{X}$ and $\mathcal{Y}$ be Polish and $K:\mathcal{X}\times\mathcal{B}(\mathcal{Y})\rightarrow[0,1]$ be a Markov kernel. Let
\[
    \mathrm{supp}(K(x,\cdot))\coloneqq\{y\in\mathcal{Y}:K(x,U)>0\textrm{ for every open }U\ni y\}.
\]
Then the graph
\[
    \mathrm{Graph}(\mathrm{supp}(K))\coloneqq\{(x,y)\in\mathcal{X}\times\mathcal{Y}:y\in\mathrm{supp}(K(x,\cdot))\}
\]
is a Borel subset of $\mathcal{X}\times\mathcal{Y}$.
\end{Lemma}

\begin{proof}
Let $(U_n)_{n\in\mathbb{N}}$ be a countable basis of open sets of $\mathcal{Y}$. Then $y\notin \mathrm{supp}(K(x,\cdot))$ if and only if there exists $n\in\mathbb{N}$ such that $y\in U_n$ and $K(x,U_n)=0$. Hence
\[
    \mathrm{Graph}(\mathrm{supp}(K))^{c}=\bigcup_{n\in\mathbb{N}}(\{x:K(x,U_{n})=0\}\times U_{n}).
\]
Since $x\mapsto K(x,U_n)$ is measurable for every open $U_n$, each set $\{x:K(x,U_{n})=0\}$ is measurable. Therefore the right-hand side is Borel, so the graph itself is Borel. 
\end{proof}

\begin{Lemma}
\label{lem:simultaneous_kernel_version}
Let $\mathcal{Y}$ be Polish, and let $Z:\Omega\to \mathcal{P}(\mathcal{Y})$ be a random probability measure. Let $\mathcal{G}\subseteq\mathcal{F}$ be a sub-$\sigma$-algebra. Suppose that there exists a countable determining class $\mathcal{C}=\{C_n\}_{n\in\mathbb{N}} \subseteq\mathcal{B}(\mathcal{Y})$ such that, for each $n\in\mathbb{N}$, the random element $\omega\mapsto Z(\omega)(C_n)$ admits a $\mathcal{G}$-measurable version. Then $Z$ admits a $\mathcal{G}$-measurable version as a $\mathcal{P}(\mathcal{Y})$-valued random element.
\end{Lemma}

\begin{proof}
For each $n\in\mathbb{N}$, choose a $\mathcal{G}$-measurable $Z_n:\Omega\to[0,1]$ such that
\[
    Z_{n}(\omega)=Z(\omega)(C_{n})\qquad\mathbb{P}\text{-a.s.}
\]
Since $\mathcal{C}$ is countable, there exists a single null set $N\in\mathcal{F}$ such that, for every $\omega\notin N$ and every $n\in\mathbb{N}$,
\[
    Z_n(\omega)=Z(\omega)(C_n).
\]
Define
\[
    W:\Omega\to[0,1]^{\mathbb{N}},\qquad W(\omega):=(Z_{n}(\omega))_{n\in\mathbb{N}}.
\]
Because each $Z_n$ is $\mathcal{G}$-measurable, $W$ is $\mathcal{G}$-measurable. Let
\[
    \Phi:\mathcal{P}(\mathcal{Y})\to[0,1]^{\mathbb{N}},\qquad\Phi(\nu):=(\nu(C_{n}))_{n\in\mathbb{N}}.
\]
The map $\Phi$ is Borel measurable. Since $\mathcal{C}$ is determining, $\Phi$ is injective. Moreover, $\mathcal{P}(\mathcal{Y})$ is a standard Borel space because $\mathcal{Y}$ is Polish. Hence, by the Lusin-Souslin theorem \parencite[Corollary 15.2]{kechris2012classical}, $\Phi(\mathcal{P}(\mathcal{Y}))$ is a Borel subset of $[0,1]^{\mathbb{N}}$, and the inverse $\Phi^{-1}:\Phi(\mathcal{P}(\mathcal{Y}))\to\mathcal{P}(\mathcal{Y})$ is Borel measurable.

For every $\omega\notin N$,
\[
    W(\omega)=\bigl(Z(\omega)(C_{n})\bigr)_{n\in\mathbb{N}}=\Phi(Z(\omega)).
\]
Therefore $W(\omega)\in\Phi(\mathcal{P}(\mathcal{Y}))$ for every $\omega\notin N$. Set
\[
    D\coloneqq W^{-1}\bigl(\Phi(\mathcal{P}(\mathcal{Y}))\bigr).
\]
The set $D$ belongs to $\mathcal{G}$, and $D^c\subseteq N$ implies $\mathbb{P}(D)=1$. Choose some fixed $\nu_0\in\mathcal{P}(\mathcal{Y})$, and define
\[
\widetilde{W}(\omega)\coloneqq \begin{cases}
W(\omega), & \omega\in D,\\
\Phi(\nu_{0}), & \text{otherwise}.
\end{cases}
\]
Since $\Phi(\mathcal{P}(\mathcal{Y}))$ is Borel and $W$ is $\mathcal{G}$-measurable, the map $\widetilde{W}$ is $\mathcal{G}$-measurable. Moreover, $\widetilde{W}(\omega)\in\Phi(\mathcal{P}(\mathcal{Y}))$, for all $\omega\in\Omega$.

Now define $\widetilde{Z}(\omega)\coloneqq\Phi^{-1}(\widetilde{W}(\omega))$, then the map $\widetilde{Z}:\Omega\to\mathcal{P}(\mathcal{Y})$ is $\mathcal{G}$-measurable. Finally, for $\omega\notin N$,
\[
    \widetilde{W}(\omega)=W(\omega)=\Phi(Z(\omega)) \implies \widetilde{Z}(\omega)=\Phi^{-1}(\Phi(Z(\omega)))=Z(\omega).
\]
Thus $\widetilde{Z}=Z$ $\mathbb{P}$-a.s., so $\widetilde{Z}$ is a $\mathcal{G}$-measurable version of $Z$.
\end{proof}

\subsection{Deferred standard proofs}
\label{app:deferred_proofs}

\begin{proof}[Proof of Lemma~\ref{lem:MK_condP}]
The map $x\mapsto\mathbb{P}_{Y\mid X}(A\mid x)$ is $\mathcal{B}(\mathcal{X})$-measurable, so $\mathbb{P}_{Y\mid X}(A\mid X)$ is $\sigma(X)$-measurable. For any $B_{\Omega}\in\sigma(X)$, there exists $B\in\mathcal{B}(\mathcal{X})$ such that $B_{\Omega}=X^{-1}(B)$. Hence, by Definition~\ref{def:regular_cond},
\[
\begin{aligned}
    \int_{B_{\Omega}}\mathbb{P}_{Y\mid X}(A\mid X(\omega))\mathbb{P}(d\omega)
    &=\int_{B}\mathbb{P}_{Y\mid X}(A\mid x)\mathbb{P}_{X}(dx)\\
    &=\mathbb{P}(X\in B,Y\in A)\\
    &=\mathbb{P}(B_{\Omega}\cap\{Y\in A\})
    =\int_{B_{\Omega}}\mathbbm{1}_{\{Y\in A\}}(\omega)\mathbb{P}(d\omega).
\end{aligned}
\]
Thus $\mathbb{P}_{Y\mid X}(A\mid X(\cdot))$ satisfies the defining property of $\mathbb{E}[\mathbbm{1}_{\{Y\in A\}}\mid\sigma(X)]$ and is therefore a version of $\mathbb{P}(Y\in A\mid\sigma(X))$.
\end{proof}

\begin{proof}[Proof of Proposition~\ref{prop:sdr_MK}]
Let
\[
   Z(\omega)\coloneqq\mathbb{P}_{Y\mid X}(\cdot\mid X(\omega))\in\mathcal{P}(\mathcal{Y}).
\]
Since $\mathcal{Y}$ is Polish, there exists a countable determining class $\mathcal{C}=\{C_{n}\}_{n\in\mathbb{N}}\subseteq\mathcal{B}(\mathcal{Y})$.

For (i) $\implies$ (ii), assume that $Y \indep X \mid \mathcal{G}$. Then, for every $A\in\mathcal{B}(\mathcal{Y})$,
\[
    \mathbb{P}_{Y\mid X}(A\mid X)=\mathbb{P}(Y\in A\mid\sigma(X))=\mathbb{P}(Y\in A\mid\mathcal{G})\qquad\mathbb{P}\text{-a.s.}
\]
In particular, for every $n\in\mathbb{N}$, the random element $Z(\cdot)(C_n)=\mathbb{P}_{Y\mid X}(C_n\mid X(\cdot))$ admits the $\mathcal{G}$-measurable version $\mathbb{P}(Y\in C_{n}\mid\mathcal{G})$. Lemma~\ref{lem:simultaneous_kernel_version} therefore shows that $Z$ admits a $\mathcal{G}$-measurable version as a $\mathcal{P}(\mathcal{Y})$-valued random element.

For (ii) $\implies$ (i), let $\widetilde{Z}:\Omega\to\mathcal{P}(\mathcal{Y})$ be a $\mathcal{G}$-measurable version of $Z$. Then, for every $A\in\mathcal{B}(\mathcal{Y})$, the random element $\widetilde{Z}(A)\coloneqq\widetilde{Z}(\cdot)(A)$ is $\mathcal{G}$-measurable and satisfies
\[
    \widetilde{Z}(A)=\mathbb{P}_{Y\mid X}(A\mid X)=\mathbb{P}(Y\in A\mid\sigma(X))\qquad\mathbb{P}\text{-a.s.},
\]
where the second equality follows from Lemma~\ref{lem:MK_condP}. The tower property and the $\mathcal{G}$-measurability of $\widetilde{Z}(A)$ now give
\[
\begin{aligned}
    \mathbb{P}(Y\in A\mid\mathcal{G})
    &=\mathbb{E}[\mathbbm{1}_{\{Y\in A\}}\mid\mathcal{G}]\\
    &=\mathbb{E}[\mathbb{E}[\mathbbm{1}_{\{Y\in A\}}\mid\sigma(X)]\mid\mathcal{G}]\\
    &=\mathbb{E}[\widetilde{Z}(A)\mid\mathcal{G}]
    =\widetilde{Z}(A)
    =\mathbb{P}(Y\in A\mid\sigma(X))\qquad\mathbb{P}\text{-a.s.}
\end{aligned}
\]
Let $\mathcal{H}$ be the class of bounded Borel functions $g:\mathcal{Y}\to\mathbb{R}$ for which
\[
    \mathbb{E}[g(Y)\mid\sigma(X)]=\mathbb{E}[g(Y)\mid\mathcal{G}]\qquad\mathbb{P}\text{-a.s.}
\]
The preceding argument shows that $\mathbbm{1}_A\in\mathcal{H}$ for every $A\in\mathcal{B}(\mathcal{Y})$. The class $\mathcal{H}$ is a vector space and is closed under bounded pointwise monotone limits. The bounded monotone class theorem therefore implies that $\mathcal{H}$ contains every bounded Borel function \parencite{billingsley1995probability,kallenberg1997foundations}. Since $\mathcal{G}\subseteq\sigma(X)$, this is equivalent to $Y \indep X \mid \mathcal{G}$.
\end{proof}

\begin{proof}[Proof of Lemma~\ref{lem:ineq_condvar}]
Because $\mathcal{Y}$ is separable, the Bochner space $L^{2}(\mathbb{P};\mathcal{Y})$ is a Hilbert space with inner product
\[
    \langle U,V\rangle_{L^{2}(\mathbb{P};\mathcal{Y})}
    \coloneqq\mathbb{E}[\langle U,V\rangle_{\mathcal{Y}}].
\]
Let $\mu\coloneqq\mathbb{E}[Y]$. The tower property gives $\mathbb{E}[Z_{1}]=\mathbb{E}[Z_{2}]=\mu$ and
\[
    \mathbb{E}[Z_{2}-Z_{1}\mid\mathcal{F}_{1}]
    =\mathbb{E}[Z_{2}\mid\mathcal{F}_{1}]-Z_{1}=0.
\]
Thus $Z_{2}-Z_{1}$ is orthogonal to the closed subspace of $\mathcal{F}_{1}$-measurable elements of $L^{2}(\mathbb{P};\mathcal{Y})$. In particular, it is orthogonal to $Z_{1}-\mu$. Since
\[
    Z_{2}-\mu=(Z_{1}-\mu)+(Z_{2}-Z_{1}),
\]
the Pythagorean theorem yields
\begin{align}\label{eq:pythag}
    \mathbb{E}[\Vert Z_{2}-\mu\Vert_{\mathcal{Y}}^{2}]
    &=\mathbb{E}[\Vert Z_{1}-\mu\Vert_{\mathcal{Y}}^{2}]
      +\mathbb{E}[\Vert Z_{2}-Z_{1}\Vert_{\mathcal{Y}}^{2}].
\end{align}
Equivalently,
\[
    \Var(Z_{2})=\Var(Z_{1})+\mathbb{E}[\Vert Z_{2}-Z_{1}\Vert_{\mathcal{Y}}^{2}].
\]
The inequality and equality condition follow immediately.
\end{proof}

\input{gaussian_analysis_background.tex}

\end{document}

%% file: statistical_properties.tex
\section{Statistical properties of linear SDR-COT}
\label{sec:linear_sdr_cot_statistics}

Proposition~\ref{prop:population_maximality} guarantees exhaustiveness for general representations but does not identify the minimal representation, since every richer representation has the same population value. We therefore restrict attention to linear subspaces of the correct structural dimension. The argument has three logically separate parts. First, the relaxed empirical COT coupling consistently approximates the population coupling, without invoking a velocity model. Second, for Euclidean responses, a slicewise Lipschitz condition turns zero population loss into identification of the central subspace; Caffarelli's contraction theorem provides a concrete way to verify this condition for the true velocity under a Gaussian source and strongly log-concave conditional target laws. Third, for Hilbert-valued responses, we do not impose ambient Lipschitz continuity. Zero-loss identification instead follows from uniqueness of a Gaussian continuity equation under Cameron--Martin Sobolev regularity and interpolation compression.

Although the coupling result in Section \ref{sec:relaxed_coupling_consistency} provides a basis for estimation, coupling convergence alone does not guarantee the uniform convergence of a fitted flow-matching criterion. Consequently, the consistency theorems impose this uniform convergence as a separate requirement. These theorems specifically address the full-sample relaxed coupling. Practical implementations introduce further approximations through minibatch transport, Monte Carlo integration over time, and stochastic optimisation, which are represented by the empirical-criterion error and the optimisation remainder used below.

\subsection{Linear set-up and identification}
\label{sec:linear_statistical_setup}

Let \(\mathcal{X}\) and \(\mathcal{Y}\) be separable Hilbert spaces. For a linear isometry \(B:\mathbb{R}^d\to\mathcal{X}\), write \(B^*:\mathcal{X}\to\mathbb{R}^d\) for its adjoint and \(P_B=BB^*\) for the orthogonal projection onto its range, \(\operatorname{ran}(B)\). Define the Stiefel manifold as
\[
    \mathrm{St}(\mathcal{X},d)\coloneqq\{B:\mathbb{R}^{d}\to\mathcal{X}:B^*B=I_{d}\}.
\]
For \(B,\widetilde B\in\mathrm{St}(\mathcal{X},d)\), the difference of the two rank-\(d\) projections is a Hilbert--Schmidt operator, and we define the projection distance
\begin{equation}
    d_{\mathrm{pr}}(B,\widetilde{B})\coloneqq\frac{\Vert P_{B}-P_{\widetilde{B}}\Vert_{\mathrm{HS}}}{\sqrt{2d}}.
\label{eq:projection_distance}
\end{equation}
When \(\mathcal{X}=\mathbb{R}^{d_x}\), we identify \(B\) with a semi-orthogonal \(d_x\times d\) matrix, so that \(B^*=B^T\), \(\mathrm{St}(\mathcal{X},d)=\mathrm{St}(d_x,d)\), and the Hilbert--Schmidt norm in \eqref{eq:projection_distance} reduces to the Frobenius norm used in Section~\ref{sec:numerics}.

Let \(\gamma^\star\) be the unique optimal triangular coupling with marginals \(\mu_0=\eta\otimes\rho\) and \(\mu_1=\mathbb{P}_{(X,Y_1)}\).\footnote{In particular, the regularity of \(\rho\) in Assumption~\ref{ass:vanish_gauss_null} makes the optimal triangular coupling \(\gamma^\star\) unique.} We view \(\gamma^\star\) as a probability measure on \((\mathcal{X}\times\mathcal{Y})^2\), with canonical coordinates \((X_0,Y_0,X_1,Y_1)\). Since it is concentrated on \(\{X_0=X_1\}\), write \(X=X_0=X_1\) and set
\[
    V\coloneqq Y_1-Y_0,
    \qquad
    Y_t\coloneqq (1-t)Y_0+tY_1=Y_0+tV,
    \qquad 0\leq t\leq T_\tau,
\]
where \(T_\tau\coloneqq1-\tau\) for a fixed \(\tau\in(0,1)\). Let \(\Theta\) be a compact metric space and let
\[
    \mathcal{U}\coloneqq\{u^{\theta}_{\mathcal{Y}}:[0,T_{\tau}]\times\mathbb{R}^{d}\times\mathcal{Y}\to\mathcal{Y}:\theta\in\Theta\}
\]
be a candidate class of reduced velocities. In this linear setting, the loss in \eqref{eq:sdr_cot_loss} reduces to
\begin{equation}
    Q(B,\theta)\coloneqq\int^{T_{\tau}}_{0}\int\Vert y_{1}-y_{0}-u^{\theta}_{\mathcal{Y}}\bigl(t,B^*x_{0},(1-t)y_{0}+ty_{1}\bigr)\Vert^{2}_{\mathcal{Y}}\,\gamma^{\star}(dx_{0},dy_{0},dx_{1},dy_{1})\,dt.
\label{eq:population_linear_risk}
\end{equation}

The following assumption specifies the statistical target. 

\begin{Assumption} 
\label{ass:linear_statistical_target}
Assumptions~\ref{ass:cond_indep} and~\ref{ass:vanish_gauss_null} hold with \(\mathcal{G}=\sigma(B_0^*X)\) for some \(B_0\in\mathrm{St}(\mathcal{X},d)\). The subspace \(\mathcal{S}_0\coloneqq\operatorname{ran}(B_0)\) is the unique \(d\)-dimensional linear central subspace for \(Y_1\) versus \(X\): it is contained in every linear dimension-reduction subspace and satisfies
\[
    Y_1\indep X\mid B_0^*X.
\]
\end{Assumption}

Assumption~\ref{ass:sigma_algebra} is compatible with Assumption~\ref{ass:linear_statistical_target}. Indeed, applying Proposition~\ref{prop:v_Gmeasurable} to the bounded one-to-one transformation obtained by applying \(\arctan\) coordinatewise to \(B_0^*X\), and then composing with its Borel inverse, gives a factorisation through \(B_0^*X\). This factorisation is only Borel; the additional regularity used for Euclidean estimation is introduced in Subsection~\ref{sec:euclidean_linear_consistency}, and is not imposed on the Hilbert-space construction in Subsection \ref{sec:functional_linear_consistency}.

\subsection{Consistency of the relaxed empirical coupling}
\label{sec:relaxed_coupling_consistency}

We next establish the coupling convergence needed for estimation. The conclusion is qualitative because it selects the relaxation parameter by a diagonal argument. An explicit sample-size rate would require a quantitative stability bound for the conditional optimal coupling.

\begin{Proposition}[Consistency of relaxed empirical COT couplings]
\label{prop:relaxed_empirical_coupling_consistency}
Let \(\mu_0,\mu_1\in\mathcal{P}^{2,\eta}(\mathcal{X}\times\mathcal{Y})\) be probability measures satisfying the conditions of Proposition~\ref{prop:condMonge}, ensuring that the triangular quadratic COT problem from \(\mu_0\) to \(\mu_1\) has a unique optimal coupling \(\gamma^\star\). Let \(\widehat\mu_{0,n}\) and \(\widehat\mu_{1,n}\) be random probability measures such that\footnote{For the empirical measure \(\widehat\mu_n=n^{-1}\sum_{i=1}^n\delta_{Z_i}\) of i.i.d. observations with law \(\mu\in\mathcal P^2(\mathcal X\times\mathcal Y)\), Varadarajan's theorem \parencite{varadarajan1958sample} gives \(\widehat\mu_n\rightharpoonup\mu\) almost surely. The strong law of large numbers applied to \(\lVert Z_i\rVert^2\) gives convergence of the second moments. Proposition~7.1.5 of \textcite{ambrosio2005gradient}, which characterises \(W_2\)-convergence by narrow convergence together with convergence of the second moments, then yields \(W_2(\widehat\mu_n,\mu)\to0\) almost surely.}
\[
    W_2(\widehat\mu_{j,n},\mu_j)\xrightarrow{p}0,
    \qquad j=0,1.
\]
For each \(n\) and \(\varepsilon>0\), let \(\widehat\gamma_{n,\varepsilon}^\star\) be a measurable choice of optimal coupling between \(\widehat\mu_{0,n}\) and \(\widehat\mu_{1,n}\) for the cost
\[
    c_{\varepsilon}(x_0,y_0,x_1,y_1)
    \coloneqq\Vert x_0-x_1\Vert_{\mathcal{X}}^2
    +\varepsilon\Vert y_0-y_1\Vert_{\mathcal{Y}}^2.
\]
For \(\varepsilon>0\), let
\[
    \Gamma^{\star}_{\varepsilon}\coloneqq\operatorname*{arg\,min}_{\gamma\in\Gamma(\mu_{0},\mu_{1})}\int c_{\varepsilon}\,d\gamma
\]
be the set of population optimal couplings for the same cost. This set is non-empty and compact in \(W_2\), with
\begin{equation}
    \sup_{\gamma\in\Gamma^{\star}_{\varepsilon}}
    W_2(\gamma,\gamma^\star)\longrightarrow0
    \qquad\text{as }\varepsilon\downarrow0,
    \label{eq:population_relaxed_coupling_consistency}
\end{equation}
There is also a deterministic sequence \(\varepsilon_n\downarrow0\) such that
\begin{equation}
    W_2(\widehat\gamma_{n,\varepsilon_n}^\star,\gamma^\star)
    \xrightarrow{p}0.
    \label{eq:relaxed_coupling_w2_consistency}
\end{equation}

Suppose that \(\mu_0=\eta\otimes\rho\), where \(\eta=\mathbb{P}_{X}\) and \(\rho\in\mathcal{P}^{2}_{r}(\mathcal{Y})\), and that \(\mu_1=\mathbb{P}_{(X,Y_1)}\). Let \(E_K\subset\mathcal{X}\) be increasing finite-dimensional subspaces with orthogonal projections \(P_K\) satisfying \(P_Kx\to x\) for every \(x\in\mathcal{X}\). Define
\[
    \mu_j^K\coloneqq (P_K,\mathrm{id}_{\mathcal{Y}})_\#\mu_j,
    \qquad j=0,1,
\]
and embed measures on \((E_K\times\mathcal{Y})^2\) in \((\mathcal{X}\times\mathcal{Y})^2\)  through natural inclusion.\footnote{``Natural inclusion'' means treating an element of the finite-dimensional subspace \(E_K\subset\mathcal X\) as the same element in the full space \(\mathcal X\).} The triangular COT problem between \(\mu_0^K\) and \(\mu_1^K\) has a unique optimal coupling \(\gamma_K^\star\), and
\begin{equation}
    W_2(\gamma_K^\star,\gamma^\star)\longrightarrow0.
    \label{eq:projected_cot_coupling_consistency}
\end{equation}
Consequently, there are deterministic sequences \(K_n\uparrow\infty\) and \(\varepsilon_n\downarrow0\), both sufficiently slow, for which the embedded relaxed empirical optimal couplings based on \((P_{K_n}X_i,Y_{0,i})\) and \((P_{K_n}X_i,Y_{1,i})\) converge to \(\gamma^\star\) in \(W_2\) in probability.
\end{Proposition}

\begin{proof}
For fixed \(\varepsilon>0\), existence of an optimum follows from the direct method for optimal transport. The set \(\Gamma(\mu_0,\mu_1)\) is tight, and the second moment of the product norm is the same for every one of its elements because the marginals are fixed. It is therefore compact in \(W_2\). Since the \(c_\varepsilon\)-cost is lower semicontinuous, \(\Gamma^{\star}_{\varepsilon}\) is non-empty and compact in \(W_2\).

Set \(C^\star\coloneqq\int\Vert y_0-y_1\Vert_{\mathcal{Y}}^2\,d\gamma^\star\).
Because \(\gamma^\star\) is triangular, it is an admissible competitor with \(c_\varepsilon\)-cost \(\varepsilon C^\star\). Thus every \(\gamma_\varepsilon\in\Gamma^{\star}_{\varepsilon}\) satisfies
\begin{equation}
    \int\Vert x_0-x_1\Vert_{\mathcal{X}}^2\,d\gamma_\varepsilon
    +\varepsilon\int\Vert y_0-y_1\Vert_{\mathcal{Y}}^2\,d\gamma_\varepsilon
    \leq \varepsilon C^\star.
    \label{eq:relaxed_cost_comparison}
\end{equation}
Let \(\varepsilon_k\downarrow0\) and choose any \(\gamma_{k}\in\Gamma_{\varepsilon_{k}}\). Compactness of \(\Gamma(\mu_0,\mu_1)\) in \(W_2\) gives a convergent subsequence with limit \(\overline\gamma\). By \eqref{eq:relaxed_cost_comparison} and lower semicontinuity, \(\overline\gamma\) is triangular and its response cost is at most \(C^\star\). The uniqueness of the triangular optimum therefore gives \(\overline\gamma=\gamma^\star\). Since this conclusion holds for every sequence and every selection, the uniform convergence in \eqref{eq:population_relaxed_coupling_consistency} follows.

We next justify the passage from the random empirical marginals to the selected empirical optimiser. Fix \(\varepsilon>0\). We claim that
\begin{equation}
    \operatorname{dist}_{W_2}
    (\widehat\gamma_{n,\varepsilon}^\star,\Gamma^{\star}_{\varepsilon})
    \xrightarrow{p}0,
    \label{eq:fixed_epsilon_empirical_stability}
\end{equation}
where \(\operatorname{dist}_{W_2}(\gamma,\Gamma)
\coloneqq\inf_{\widetilde\gamma\in\Gamma}W_2(\gamma,\widetilde\gamma)\).
Consider an arbitrary subsequence. Convergence in probability of both empirical marginals permits a further subsequence, not relabelled, along which
\[
    W_{2}(\widehat{\mu}_{j,n},\mu_{j})\xrightarrow{\mathrm{a.s.}}0,
    \qquad j=0,1.
\]
Fix an outcome in this probability-one event. The two convergent sequences of marginals are tight, hence so are their couplings. To verify the cost condition in the stability theorem, fix \(\gamma\in\Gamma(\mu_0,\mu_1)\), couple each \(\mu_j\) optimally to \(\widehat\mu_{j,n}\), and glue these two couplings to \(\gamma\) over its endpoint coordinates. The resulting measures \(\gamma_n\in\Gamma(\widehat\mu_{0,n},\widehat\mu_{1,n})\) converge to \(\gamma\) in \(W_2\), and their \(c_\varepsilon\)-costs converge to that of \(\gamma\). In particular, the empirical optimal costs are finite and uniformly bounded. The stability theorem for optimal couplings with the fixed continuous cost \(c_\varepsilon\) therefore applies, and every weak accumulation point of \(\widehat\gamma_{n,\varepsilon}^\star\) belongs to \(\Gamma^{\star}_{\varepsilon}\); see \textcite[Theorem~5.20]{Villani2009}. Moreover,
\begin{align*}
    &\int\bigl(\Vert(x_0,y_0)\Vert^2+\Vert(x_1,y_1)\Vert^2\bigr)
    \,d\widehat\gamma_{n,\varepsilon}^\star\\
    &\qquad=\int\Vert z\Vert^2\,d\widehat\mu_{0,n}(z)
    +\int\Vert z\Vert^2\,d\widehat\mu_{1,n}(z)\\
    &\qquad\longrightarrow
    \int\Vert z\Vert^2\,d\mu_0(z)+\int\Vert z\Vert^2\,d\mu_1(z).
\end{align*}
Thus every weakly convergent further subsequence converges in \(W_2\) to an element of \(\Gamma^{\star}_{\varepsilon}\). Compactness and a contradiction argument show that the distance in \eqref{eq:fixed_epsilon_empirical_stability} converges to zero almost surely along the chosen further subsequence. The subsequence characterisation of convergence in probability proves \eqref{eq:fixed_epsilon_empirical_stability} for the original sequence. This argument applies to any measurable selection of an empirical optimum; neither the empirical nor the population optimum need be unique at fixed \(\varepsilon\).

Choose \(\varepsilon_k\downarrow0\) so that \(\sup_{\gamma\in\Gamma_{\varepsilon_k}}W_2(\gamma,\gamma^\star)\leq k^{-1}\). By \eqref{eq:fixed_epsilon_empirical_stability}, choose deterministic integers \(N_k\uparrow\infty\) such that, for all \(n\geq N_k\),
\[
    \mathbb{P}\left(\operatorname{dist}_{W_2}
    (\widehat{\gamma}^{\star}_{n,\varepsilon_{k}},\Gamma_{\varepsilon_k})>k^{-1}\right)\leq k^{-1}.
\]
Set \(\varepsilon_n=\varepsilon_k\) whenever \(N_k\leq n<N_{k+1}\). Compactness of \(\Gamma_{\varepsilon_k}\) and the triangle inequality then give
\[
    \mathbb{P}\left(W_{2}(\widehat{\gamma}^{\star}_{n,\varepsilon_{n}},\gamma^{\star})>2k^{-1}\right)\leq k^{-1},\qquad N_{k}\leq n<N_{k+1},
\]
which proves \eqref{eq:relaxed_coupling_w2_consistency}.

For the projection result, Proposition~\ref{prop:condMonge} and the regularity of \(\rho\) yield a unique conditional Monge coupling from \(\rho\) to the conditional law of \(Y_1\) given \(P_KX\), for \((P_K)_\#\eta\)-a.e.\ projected covariate. Thus \(\gamma_K^\star\) exists and is unique. Moreover, \(\mu_j^K\to\mu_j\) in \(W_2\) because \(\mathbb{E}[\Vert P_KX-X\Vert_{\mathcal{X}}^2]\to0\).
The pushforward of \(\gamma^\star\) obtained by replacing both covariate coordinates by \(P_KX\) is a triangular coupling of \(\mu_0^K\) and \(\mu_1^K\) with response cost \(C^\star\). Hence the response cost of \(\gamma_K^\star\) is at most \(C^\star\). The convergence of the projected marginals gives tightness and convergence of the endpoint second moments, so \(\{\gamma_K^\star\}_K\) is precompact in \(W_2\). Every accumulation point is a triangular coupling of \(\mu_0\) and \(\mu_1\) with response cost at most \(C^\star\), and therefore equals \(\gamma^\star\). It follows that the full sequence converges in \(W_2\).

It remains to make the final simultaneous choice explicit. Choose a strictly increasing sequence \(K_m\) such that \(W_2(\gamma_{K_m}^\star,\gamma^\star)\leq m^{-1}\).
Let \(\Gamma_{K,\varepsilon}\) denote the set of population relaxed optima between \(\mu_0^K\) and \(\mu_1^K\), embedded in the full endpoint space, and let \(\widehat\gamma_{n,K,\varepsilon}^\star\) be any measurable empirical optimum. Applying \eqref{eq:population_relaxed_coupling_consistency} at fixed \(K_m\), choose \(\varepsilon_m\downarrow0\) so that
\[
    \sup_{\gamma\in\Gamma_{K_m,\varepsilon_m}}
    W_2(\gamma,\gamma_{K_m}^\star)\leq m^{-1}.
\]
The projected empirical marginals converge to \(\mu_j^{K_m}\) in \(W_2\) because the map \((x,y)\mapsto(P_{K_m}x,y)\) is \(1\)-Lipschitz. Hence the fixed-parameter argument in \eqref{eq:fixed_epsilon_empirical_stability} allows deterministic integers \(N_m\uparrow\infty\) to be chosen so that, for every \(n\geq N_m\),
\[
    \mathbb{P}\left(\operatorname{dist}_{W_2}
    (\widehat{\gamma}^{\star}_{n,K_{m},\varepsilon_{m}},
    \Gamma_{K_m,\varepsilon_m})>m^{-1}\right)\leq m^{-1}.
\]
For \(N_m\leq n<N_{m+1}\), set \(K_n=K_m\) and \(\varepsilon_n=\varepsilon_m\). Compactness of \(\Gamma_{K_m,\varepsilon_m}\) and the triangle inequality give
\[
    \mathbb{P}\left(W_{2}(\widehat{\gamma}^{\star}_{n,K_{n},\varepsilon_{n}},\gamma^{\star})>3m^{-1}\right)\leq m^{-1},
\]
which proves the final assertion.
\end{proof}

\subsection{Euclidean central-subspace consistency}
\label{sec:euclidean_linear_consistency}

We now take \(\mathcal{X}=\mathbb{R}^{d_x}\) and \(\mathcal{Y}=\mathbb{R}^{d_y}\). The following condition gives a concrete way to verify regularity of the true Euclidean COT velocity. It is compatible with the Gaussian source used in the numerical procedure, does not require bounded support, and allows the strength of log-concavity to vary with the sufficient index.

\begin{Assumption}[Euclidean conditional transport regularity]
\label{ass:euclidean_conditional_transport_regularity}
The source response law is \(\rho=\mathcal{N}(m_\rho,\Sigma_\rho)\), where \(\Sigma_\rho\) is positive definite. Set
\[
    \beta_\rho\coloneqq\Vert\Sigma_\rho^{-1}\Vert_{\mathrm{op}}.
\]
Let \(\zeta_0\coloneqq (B_0^*)_\#\eta\). There exist a Borel set \(D_0\subseteq\mathbb{R}^d\) with \(\zeta_0(D_0)=1\), a Borel function \(\alpha:D_0\to(0,\infty)\), and a Borel function \(W:D_0\times\mathbb{R}^{d_y}\to\mathbb{R}\) such that, for every \(s\in D_0\), the map \(W_s\coloneqq W(s,\cdot)\) is twice continuously differentiable,
\begin{equation}
    \nabla^2 W_s(y)\succeq\alpha(s)I_{d_y},
    \qquad y\in\mathbb{R}^{d_y},
\label{eq:conditional_target_curvature}
\end{equation}
and
\begin{equation}
    \mathbb{P}_{Y_{1}\mid B^{*}_{0}X}(dy\mid s)=\exp(-W_{s}(y))\,dy.
\label{eq:conditional_target_log_concave_density}
\end{equation}
Here \(W_s\) includes its normalising constant.
\end{Assumption}

Assumption~\ref{ass:euclidean_conditional_transport_regularity} requires the conditional target laws to be slicewise strongly log-concave. The curvature \(\alpha(s)\) need not be bounded away from zero over \(s\), so the resulting Lipschitz constants need not be uniform in the reduced index. Models with multimodal or heavy-tailed conditional response laws fall outside this particular sufficient condition. While this limitation applies to the Euclidean consistency theorem, it does not affect the map and velocity factorisation results detailed in Sections~\ref{sec:restricted_dynamic_cot} and~\ref{sec:sdr_cot}.

\begin{Proposition}[Slicewise regularity of the Euclidean COT velocity]
\label{prop:euclidean_conditional_velocity_regularity}
Suppose that Assumptions~\ref{ass:linear_statistical_target} and~\ref{ass:euclidean_conditional_transport_regularity} hold. For \(s\in D_0\), let \(T_s\) be the quadratic optimal transport map from \(\rho\) to the measure in \eqref{eq:conditional_target_log_concave_density}. Then \(T_s\) has a globally Lipschitz representative\footnote{The Lipschitz constant is defined as \(\operatorname{Lip}(T)\coloneqq\sup_{x\neq y}\Vert T(x)-T(y)\Vert/\Vert x-y\Vert\).} satisfying
\begin{equation}
    \operatorname{Lip}(T_s)
    \leq L_T(s)
    \coloneqq\sqrt{\frac{\beta_\rho}{\alpha(s)}}.
    \label{eq:caffarelli_slicewise_bound}
\end{equation}
For \(0\leq t\leq T_\tau\), set \(T_{t,s}\coloneqq(1-t)\mathrm{id}+tT_s\). The map \(T_{t,s}\) is a bijection of \(\mathbb{R}^{d_y}\), and
\begin{equation}
    u^0_{\mathcal{Y}}(t,s,z)
    \coloneqq \bigl(T_s-\mathrm{id}\bigr)\bigl(T_{t,s}^{-1}(z)\bigr)
    \label{eq:euclidean_exact_current_velocity}
\end{equation}
has a jointly Borel representative such that
\begin{equation}
    \Vert u^0_{\mathcal{Y}}(t,s,z)-u^0_{\mathcal{Y}}(t,s,\widetilde z)\Vert
    \leq \ell_0(s)\Vert z-\widetilde z\Vert,
    \qquad
    \ell_0(s)\coloneqq\frac{1+L_T(s)}{\tau}.
    \label{eq:euclidean_velocity_slicewise_bound}
\end{equation}
Moreover,
\begin{equation}
    V=u^0_{\mathcal{Y}}(t,B_0^*X,Y_t)
    \quad\text{for }dt\otimes\gamma^\star\text{-a.e. }(t,X,Y_0,Y_1).
    \label{eq:caffarelli_true_velocity_identity}
\end{equation}
\end{Proposition}

\begin{proof}
The Gaussian source has density proportional to \(\exp(-V_\rho)\), where \(\nabla^2V_\rho=\Sigma_\rho^{-1}\preceq\beta_\rho I_{d_y}\). For each \(s\in D_0\), Caffarelli's contraction theorem \parencite[Theorem~2.2]{kolesnikov2011mass}, applied with \eqref{eq:conditional_target_curvature}, gives \eqref{eq:caffarelli_slicewise_bound}.

The map \(T_s\) is the gradient of a convex function and is therefore monotone. Hence, for \(y,\widetilde y\in\mathbb{R}^{d_y}\),
\[
    \langle T_{t,s}(y)-T_{t,s}(\widetilde y),y-\widetilde y\rangle
    =(1-t)\Vert y-\widetilde y\Vert^2
      +t\langle T_s(y)-T_s(\widetilde y),y-\widetilde y\rangle
    \geq(1-t)\Vert y-\widetilde y\Vert^2.
\]
Thus \(T_{t,s}\) is injective and \(T_{t,s}^{-1}\) is \((1-t)^{-1}\)-Lipschitz on its range. To see surjectivity, write \(T_s=\nabla\phi_s\) with \(\phi_s\) convex. For each \(z\in\mathbb{R}^{d_y}\), the function
\[
    y\mapsto(1-t)\frac{1}{2}\Vert y\Vert^{2}+t\phi_{s}(y)-\langle z,y\rangle
\]
is \((1-t)\)-strongly convex. In particular, it is bounded below by an affine function plus \((1-t)\Vert y\Vert^2/2\), and is therefore coercive. It has a unique minimiser, whose first-order condition is \(T_{t,s}(y)=z\). Hence \(T_{t,s}\) is bijective. It follows that
\[
    \Vert u^{0}_{\mathcal{Y}}(t,s,z)-u^{0}_{\mathcal{Y}}(t,s,\widetilde z)\Vert
    \leq (1+L_T(s))\Vert T^{-1}_{t,s}(z)-T^{-1}_{t,s}(\widetilde z)\Vert
    \leq \frac{1+L_T(s)}{1-t}\Vert z-\widetilde z\Vert
    \leq \ell_0(s)\Vert z-\widetilde z\Vert.
\]
It remains to choose the slicewise Lipschitz maps in a jointly measurable way. Applying Proposition~\ref{prop:condMonge} to the reduced source and target laws \(\zeta_0\otimes\rho\) and \(\mathbb{P}_{(B_0^*X,Y_1)}\) gives a jointly Borel map
\(\widetilde T:D_0\times\mathbb{R}^{d_y}\to\mathbb{R}^{d_y}\) such that \(\widetilde T(s,\cdot)\) is an optimal map from \(\rho\) to \(\mathbb{P}_{Y_1\mid B_0^*X}(\cdot\mid s)\) for \(\zeta_0\)-a.e.\ \(s\). By replacing \(D_0\) with a Borel subset of full \(\zeta_0\)-measure, we may assume that this assertion holds for every \(s\in D_0\). For each such \(s\), let \(T_s^{\mathrm{Lip}}\) denote the globally Lipschitz representative supplied by Caffarelli's contraction theorem. Uniqueness of the quadratic optimal map gives
\begin{equation}
    \widetilde T(s,y)=T_s^{\mathrm{Lip}}(y)
    \qquad\text{for }\rho\text{-a.e. }y.
    \label{eq:borel_lipschitz_map_agreement}
\end{equation}
The Gaussian measure \(\rho\) has a strictly positive Lebesgue density. Hence the equality in \eqref{eq:borel_lipschitz_map_agreement} also holds for Lebesgue-a.e.\ \(y\).

We now construct a jointly Borel version which agrees everywhere with \(T_s^{\mathrm{Lip}}\) on each slice. Let \(\varphi\in C_c^\infty(\mathbb{R}^{d_y})\) be non-negative with integral one, set \(\varphi_m(v)=m^{d_y}\varphi(mv)\), and define, coordinatewise,
\begin{equation}
    T^{(m)}(s,y)\coloneqq(\widetilde{T}(s,\cdot)*\varphi_{m})(y)=\int_{\mathbb{R}^{d_{y}}}\widetilde{T}(s,y-v)\varphi_{m}(v)\,dv.
\label{eq:slicewise_transport_mollification}
\end{equation}
For every fixed \(s\in D_0\), the map \(\widetilde T(s,\cdot)\) is locally integrable because it agrees Lebesgue-a.e.\ with the Lipschitz map \(T_s^{\mathrm{Lip}}\). Thus \eqref{eq:slicewise_transport_mollification} is finite. The integrand is jointly Borel in \((s,y,v)\), so parameter-dependent integration shows that \(T^{(m)}\) is jointly Borel in \((s,y)\). Moreover, \eqref{eq:borel_lipschitz_map_agreement} implies that
\[
    T^{(m)}(s,y)=(T_s^{\mathrm{Lip}}*\varphi_m)(y).
\]
The Lipschitz bound gives, for every \(s\in D_0\) and \(y\in\mathbb{R}^{d_y}\),
\[
    \Vert T^{(m)}(s,y)-T_s^{\mathrm{Lip}}(y)\Vert
    \leq L_T(s)
    \int_{\mathbb{R}^{d_y}}\Vert v\Vert\varphi_m(v)\,dv
    \longrightarrow0.
\]
Consequently, \(T(s,y)\coloneqq\lim_{m\to\infty}T^{(m)}(s,y)\) is jointly Borel, and \(T(s,\cdot)=T_s^{\mathrm{Lip}}\) everywhere for every \(s\in D_0\). We henceforth take \(T_s=T(s,\cdot)\). This construction imposes no continuity or uniform Lipschitz bound in the index \(s\).

Finally, consider the space--time map
\[
    \mathcal{T}:[0,T_\tau]\times D_0\times\mathbb{R}^{d_y}
    \longrightarrow[0,T_\tau]\times D_0\times\mathbb{R}^{d_y},
    \qquad
    \mathcal{T}(t,s,y)=\bigl(t,s,T_{t,s}(y)\bigr).
\]
The map \(\mathcal{T}\) is Borel because \((s,y)\mapsto T_s(y)\) is Borel. It is bijective because, for every fixed \((t,s)\), the map \(T_{t,s}\) is bijective, as proved above. Since its domain and codomain are standard Borel spaces, the Lusin--Souslin theorem implies that \(\mathcal{T}^{-1}\) is Borel; see \textcite[Corollary~15.2]{kechris2012classical}. Its third coordinate is precisely the jointly Borel map \((t,s,z)\mapsto T_{t,s}^{-1}(z)\). Therefore
\[
    (t,s,z)\mapsto
    (T_s-\mathrm{id})\bigl(T_{t,s}^{-1}(z)\bigr)
\]
is jointly Borel on \([0,T_\tau]\times D_0\times\mathbb{R}^{d_y}\). Defining it to be zero when \(s\notin D_0\) gives the asserted jointly Borel representative of \(u^0_{\mathcal{Y}}\) on the whole domain.

Assumption~\ref{ass:linear_statistical_target} gives \(\mathbb{P}_{Y_1\mid X}(\cdot\mid x)=\mathbb{P}_{Y_1\mid B_0^*X}(\cdot\mid B_0^*x)\) a.e.\ Hence the conditional Monge map at covariate \(X\) is \(T_{B_0^*X}\), and \eqref{eq:euclidean_exact_current_velocity} gives \eqref{eq:caffarelli_true_velocity_identity} along the interpolation.
\end{proof}

Our consistency result requires a correctly specified model, as formalised in the following assumption.

\begin{Assumption}[Euclidean fitted velocity class]
\label{ass:euclidean_velocity_model}
The parameter space \(\Theta\) is compact, the map \((\theta,t,s,y)\mapsto u^\theta_{\mathcal{Y}}(t,s,y)\) is Borel, and \(Q(B,\theta)<\infty\) for every \((B,\theta)\in\mathrm{St}(d_x,d)\times\Theta\). Let \(\zeta_B\coloneqq B^*_\#\eta\). For every \((B,\theta)\), there exist a Borel set \(D_{B,\theta}\subseteq\mathbb{R}^d\), with \(\zeta_B(D_{B,\theta})=1\), and a Borel function \(\ell_{B,\theta}:D_{B,\theta}\to[0,\infty)\), finite everywhere on \(D_{B,\theta}\), such that
\begin{equation}
    \Vert u^\theta_{\mathcal{Y}}(t,s,y)-u^\theta_{\mathcal{Y}}(t,s,\widetilde y)\Vert
    \leq\ell_{B,\theta}(s)\Vert y-\widetilde y\Vert
    \label{eq:slicewise_candidate_lipschitz}
\end{equation}
for every \(s\in D_{B,\theta}\), \(t\in[0,T_\tau]\), and \(y,\widetilde y\in\mathbb{R}^{d_y}\). The population criterion \(Q\) is lower semicontinuous. Finally, the model is correctly specified: there is a \(\theta_0\in\Theta\) such that
\begin{equation}
    V=u^{\theta_0}_{\mathcal{Y}}(t,B_0^*X,Y_t)
    \quad\text{for }dt\otimes\gamma^\star\text{-a.e. }(t,X,Y_0,Y_1).
    \label{eq:euclidean_correct_velocity_specification}
\end{equation}
\end{Assumption}

Proposition~\ref{prop:euclidean_conditional_velocity_regularity} verifies the response-regularity part of Assumption~\ref{ass:euclidean_velocity_model} for the true velocity, with the non-uniform envelope \(\ell_0\). Correct model specification additionally requires the fitted class to contain this representative.

\begin{Lemma}[Euclidean zero-loss identification]
\label{lem:zero_loss_identification}
Under Assumption~\ref{ass:linear_statistical_target}, let \(B\in\mathrm{St}(d_x,d)\), and let \(u_{\mathcal{Y}}:[0,T_\tau]\times\mathbb{R}^d\times\mathbb{R}^{d_y}\to\mathbb{R}^{d_y}\) be Borel. Suppose that there are a Borel set \(D\subseteq\mathbb{R}^d\), with \(\zeta_B(D)=1\), and a finite Borel function \(\ell:D\to[0,\infty)\) for which
\[
    \Vert u_{\mathcal{Y}}(t,s,y)-u_{\mathcal{Y}}(t,s,\widetilde{y})\Vert\leq\ell(s)\Vert y-\widetilde{y}\Vert
\]
for every \(s\in D\) and every \(t,y,\widetilde y\). If
\[
    \int^{T_{\tau}}_{0}\mathbb{E}_{\gamma^{\star}}\bigl[\Vert V-u_{\mathcal{Y}}(t,B^{*}X,Y_{t})\Vert^{2}\bigr]dt=0,
\]
then \(Y_1\indep X\mid B^*X\), and consequently \(\operatorname{ran}(B)=\mathcal{S}_0\).
\end{Lemma}

\begin{proof}
Put \(S=B^*X\). By Fubini's theorem, there is a set \(A\subset(0,T_\tau)\) of full Lebesgue measure such that
\begin{equation}
    V=u_{\mathcal{Y}}(t,S,Y_t)\qquad\gamma^\star\text{-a.s.}
    \label{eq:positive_time_zero_loss}
\end{equation}
for every \(t\in A\). To recover the current state from the source state, we shall invert the backward map \(y\mapsto y-tu_{\mathcal{Y}}(t,s,y)\). Its injectivity is ensured by \(t\ell(s)<1\). Since \(\ell\) need not be uniformly bounded, we first localise the index \(s\) according to the size of its Lipschitz constant. For \(m\geq1\), let
\(D_m=\{s\in D:\ell(s)\leq m\}\). These sets increase to \(D\). Since \(A\) has full Lebesgue measure, we may choose \(t_m\in A\cap(0,T_\tau\wedge(2m)^{-1})\).
For \(s\in D_m\), consider the backward map \(y\mapsto y-t_m u_{\mathcal{Y}}(t_m,s,y)\).
When evaluated at \(y=Y_{t_m}\), equation~\eqref{eq:positive_time_zero_loss} shows that it returns 
\[
   Y_{t_m}-t_mu_{\mathcal{Y}}(t_m,S,Y_{t_m})=Y_{t_m}-t_mV=Y_0.
\]
To construct its inverse jointly in \(s\), define on \(D_m\times\mathbb{R}^{d_y}\)
\[
   F_m(s,y)\coloneqq(s,y-t_mu_{\mathcal{Y}}(t_m,s,y)).
\]
For \(s\in D_m\), the reverse triangle inequality and the slicewise Lipschitz bound give
\begin{equation}
\begin{aligned}
   \Vert(y-t_{m}u_{\mathcal{Y}}(t_{m},s,y))-(\widetilde{y}-t_{m}u_{\mathcal{Y}}(t_{m},s,\widetilde{y}))\Vert&\geq\Vert y-\widetilde{y}\Vert-t_{m}\Vert u_{\mathcal{Y}}(t_{m},s,y)-u_{\mathcal{Y}}(t_{m},s,\widetilde{y})\Vert\\&\geq(1-t_{m}m)\Vert y-\widetilde{y}\Vert\geq\tfrac{1}{2}\Vert y-\widetilde{y}\Vert.
\end{aligned}
\label{eq:slicewise_positive_time_inverse_bound}
\end{equation}
Thus the backward map is injective on every slice, and retaining \(s\) as the first coordinate makes \(F_m\) a Borel injection. By the Lusin--Souslin theorem, \(F_m(D_m\times\mathbb{R}^{d_y})\) is Borel and \(F_m^{-1}\) is Borel on this image. Let \(G_m\) be the second coordinate of this inverse, extended arbitrarily off the image. Then \(G_m:D_m\times\mathbb{R}^{d_y}\to\mathbb{R}^{d_y}\) is Borel and satisfies
\[
    G_m\bigl(s,y-t_m u_{\mathcal{Y}}(t_m,s,y)\bigr)=y,
    \qquad (s,y)\in D_m\times\mathbb{R}^{d_y}.
\]

We next assemble these slicewise inverses. Set \(C_1=D_1\) and \(C_m=D_m\setminus D_{m-1}\) for \(m\geq2\). Because \(\ell\) is finite on \(D\) and \(\zeta_B(D)=1\), the disjoint sets \(C_m\) cover the law of \(S\) up to a null set. Since the collection \(\{t_m:m\geq1\}\) is countable, equation~\eqref{eq:positive_time_zero_loss} holds at all these times on one common \(\gamma^\star\)-full set. On \(\{S\in C_m\}\), we have \(S\in D_m\) and
\[
    F_m(S,Y_{t_m})=(S,Y_0).
\]
Applying the inverse constructed above therefore gives
\[
    Y_{t_m}=G_m(S,Y_0),
    \qquad
    Y_1=G_m(S,Y_0)+(1-t_m)u_{\mathcal{Y}}\bigl(t_m,S,G_m(S,Y_0)\bigr)
\]
almost surely, where the second identity uses \(Y_1=Y_{t_m}+(1-t_m)V\). Defining the right-hand side piecewise over the Borel partition \(\{C_m\}_m\), and arbitrarily on its null complement, yields a Borel map
\[
    H:\mathbb{R}^d\times\mathbb{R}^{d_y}\to\mathbb{R}^{d_y}
    \quad\text{such that}\quad
    Y_1=H(S,Y_0)\quad\gamma^\star\text{-a.s.}
\]

It remains to translate this reconstruction into conditional independence. The \((X,Y_0)\)-marginal of \(\gamma^\star\) is \(\eta\otimes\rho\), so \(Y_0\) is independent of \(X\) and has law \(\rho\). Hence, for every bounded Borel function \(f:\mathbb{R}^{d_y}\to\mathbb{R}\),
\[
    \mathbb{E}_{\gamma^{\star}}[f(Y_{1})\mid X]=\mathbb{E}_{\gamma^{\star}}[f(H(S,Y_{0}))\mid X]=\int_{\mathbb{R}^{d_{y}}}f(H(S,y_{0}))\,\rho(dy_{0}).
\]
The last expression is a Borel function of \(S=B^*X\) alone. Equivalently, the Borel kernel \(K(s,\cdot)\coloneqq H(s,\cdot)_\#\rho\) satisfies  
\[
    \mathbb{P}_{Y_1\mid X}(\cdot\mid X)=K(B^*X,\cdot)
    \qquad\text{a.s.}
\]
Proposition~\ref{prop:sdr_MK} therefore gives \(Y_1\indep X\mid B^*X\). Thus \(\operatorname{ran}(B)\) is a \(d\)-dimensional linear dimension-reduction subspace. Assumption~\ref{ass:linear_statistical_target} states that \(\mathcal{S}_0\) is contained in every such subspace and also has dimension \(d\); consequently, \(\operatorname{ran}(B)=\mathcal{S}_0\).
\end{proof}

For any coupling \(\gamma\) on the endpoint space, let \(Q_\gamma(B,\theta)\) denote the criterion in \eqref{eq:population_linear_risk} with \(\gamma^\star\) replaced by \(\gamma\). Given observations \(\{(X_i,Y_{1,i})\}_{i=1}^n\) and independent source draws \(\{Y_{0,i}\}_{i=1}^n\), let \(\widehat\gamma_n=\widehat\gamma_{n,\varepsilon_n}^\star\) be a relaxed empirical optimal coupling and define
\begin{equation}
    \widehat Q_n(B,\theta)\coloneqq Q_{\widehat\gamma_n}(B,\theta).
\label{eq:empirical_linear_risk}
\end{equation}
The time integral may be replaced by Monte Carlo draws. Note that \(W_2(\widehat\gamma_n,\gamma^\star)\to0\) does not by itself imply uniform convergence of \(\widehat Q_n\). Rather than proving architecture-specific universal-approximation theorems, our aim is to establish the statistical guarantees of the SDR--COT framework given a sufficiently rich parametrisation. Therefore, we impose the necessary empirical-process condition directly.

\begin{Theorem}[Consistency in Euclidean linear SDR]
\label{thm:euclidean_linear_sdr_consistency}
Suppose that Assumptions~\ref{ass:linear_statistical_target} and~\ref{ass:euclidean_velocity_model} hold. Assume that
\begin{equation}
    e_n\coloneqq\sup_{B\in\mathrm{St}(d_x,d),\,\theta\in\Theta}|\widehat Q_n(B,\theta)-Q(B,\theta)|\xrightarrow{p}0.
\label{eq:euclidean_uniform_criterion_convergence}
\end{equation}
Let \((\widehat B_n,\widehat\theta_n)\in\mathrm{St}(d_x,d)\times\Theta\) satisfy
\begin{equation}
    \widehat Q_n(\widehat B_n,\widehat\theta_n)
    \leq
    \inf_{B\in\mathrm{St}(d_x,d),\,\theta\in\Theta}
    \widehat Q_n(B,\theta)+r_n,
    \qquad r_n\geq0,\quad r_n\xrightarrow{p}0.
    \label{eq:approximate_empirical_minimiser}
\end{equation}
Then
\begin{equation}
    d_{\mathrm{pr}}(\widehat B_n,B_0)\xrightarrow{p}0.
    \label{eq:euclidean_projection_consistency}
\end{equation}
\end{Theorem}

\begin{proof}
Assumption~\ref{ass:euclidean_velocity_model} gives \(Q(B_0,\theta_0)=0\). The parameter space \(\mathrm{St}(d_x,d)\times\Theta\) is compact, and \(Q\) is lower semicontinuous.

Fix \(\delta>0\), and let 
\[
    \mathcal{K}_\delta
    \coloneqq \{(B,\theta)\in\mathrm{St}(d_x,d)\times\Theta:
    d_{\mathrm{pr}}(B,B_0)\geq\delta\}.
\]
If \(\mathcal{K}_\delta\) is empty, the required probability is zero and there is nothing to prove. We may therefore assume that \(\mathcal{K}_\delta\) is non-empty. Compactness follows because \(\mathcal{K}_\delta\) is a closed subset of the compact parameter space. Since \(Q\) is lower semicontinuous, it attains its infimum on \(\mathcal{K}_\delta\). If this infimum were zero, there would exist some \((B_\delta,\theta_\delta)\in\mathcal{K}_\delta\) with \(Q(B_\delta,\theta_\delta)=0\). Assumption~\ref{ass:euclidean_velocity_model} ensures that \(u^{\theta_\delta}_{\mathcal{Y}}\) is Borel and satisfies the slicewise Lipschitz hypothesis of Lemma~\ref{lem:zero_loss_identification}, with \(D=D_{B_\delta,\theta_\delta}\) and \(\ell=\ell_{B_\delta,\theta_\delta}\). The equality \(Q(B_\delta,\theta_\delta)=0\) supplies the zero-loss hypothesis of that lemma. Applying it with \(B=B_\delta\) therefore gives 
\[
   \operatorname{ran}(B_\delta)=\mathcal{S}_0=\operatorname{ran}(B_0).
\] 
Consequently, \(d_{\mathrm{pr}}(B_\delta,B_0)=0\), contradicting \((B_\delta,\theta_\delta)\in\mathcal{K}_\delta\) and \(\delta>0\). Hence 
\[
   c_\delta\coloneqq\min_{(B,\theta)\in\mathcal{K}_\delta}Q(B,\theta)>0.
\]

It remains to relate the population risk of the empirical near-minimiser to this separation constant. By the definition of \(e_n\), for every \((B,\theta)\) in the parameter space,
\[
    Q(B,\theta)\leq \widehat Q_n(B,\theta)+e_n,
    \qquad
    \widehat Q_n(B,\theta)\leq Q(B,\theta)+e_n.
\label{eq:two_sided_uniform_risk_bound}
\] 
Using these inequalities and \eqref{eq:approximate_empirical_minimiser}, we obtain
\begin{align*}
    Q(\widehat{B}_{n},\widehat{\theta}_{n})&\leq\widehat{Q}_{n}(\widehat{B}_{n},\widehat{\theta}_{n})+e_{n}\leq\inf_{B\in\mathrm{St}(d_{x},d),\,\theta\in\Theta}\widehat{Q}_{n}(B,\theta)+r_{n}+e_{n}\leq\widehat{Q}_{n}(B_{0},\theta_{0})+r_{n}+e_{n}\\&\leq Q(B_{0},\theta_{0})+2e_{n}+r_{n}=2e_{n}+r_{n}.
\end{align*}
On the event \(\{d_{\mathrm{pr}}(\widehat B_n,B_0)\geq\delta\}\), the pair \((\widehat B_n,\widehat\theta_n)\) belongs to \(\mathcal{K}_\delta\), so \(c_\delta\leq Q(\widehat B_n,\widehat\theta_n)\leq2e_n+r_n\). Consequently,
\[
    \mathbb{P}\bigl(d_{\mathrm{pr}}(\widehat B_n,B_0)\geq\delta\bigr)
    \leq\mathbb{P}\bigl(2e_n+r_n\geq c_\delta\bigr).
\]
The right-hand side tends to zero because \(2e_n+r_n\xrightarrow{p}0\) and \(c_\delta>0\) is fixed. Since \(\delta>0\) is arbitrary, \eqref{eq:euclidean_projection_consistency} follows.
\end{proof}

\subsection{Functional covariates and responses}
\label{sec:functional_linear_consistency}

For Hilbert-valued responses, we use the Gaussian--Sobolev route: the velocity is Cameron--Martin-valued and weakly differentiable along Cameron--Martin directions, without requiring continuity or Lipschitz regularity in the ambient \(\mathcal{Y}\)-norm. This route also requires a Gaussian source and an \(L^r\) compression condition for the conditional interpolation. These are additional statistical assumptions and do not enter the Borel map and velocity results proved earlier.

The identification argument separates existence from uniqueness. The conditional Monge construction already provides, for each covariate value \(x\), an interpolation \((\mu_t^x)_t\) starting from the common source \(\rho\). If the population loss is zero, then, for a fixed reduced value \(s=B^*x\), every interpolation with \(x\) in the same slice is driven by the field \(u(t,s,\cdot)\). Gaussian--Sobolev admissibility supplies the velocity-side conditions for uniqueness of the weak Gaussian continuity equation, while compression places the interpolation densities in the corresponding uniqueness class. Uniqueness then identifies the intermediate laws within each slice, and the common velocity reconstructs a common conditional endpoint law. This yields \(Y_1\indep X\mid B^*X\).

Appendix~\ref{app:gaussian_analysis_background} connects the Gaussian-measure and covariance notation of Section~\ref{sec:preliminaries} to the abstract Wiener and Cameron--Martin structures used here. It also defines the Malliavin derivative, Gaussian Sobolev spaces and Gaussian divergence, and states the continuity-equation uniqueness theorem used below. Throughout this subsection, \(\rho=\mathcal{N}(0,C_\rho)\) is a nondegenerate centred Gaussian measure on \(\mathcal{Y}\); \(\mathcal{H}_\rho\) is its Cameron--Martin space, along which translations yield equivalent measures; \(D_\rho\) denotes the Malliavin derivative, which extends directional derivatives to this stochastic setting; \(W_\rho^{1,q}(\rho;\mathcal{H}_\rho)\) is the Gaussian Sobolev space of functions with integrable Malliavin derivatives; \(\mathcal{S}_{2}(\mathcal{H}_\rho)\) is the space of Hilbert--Schmidt operators; and \(\operatorname{div}_\rho\) is the divergence operator, defined as the adjoint of \(D_\rho\). In particular, Gaussian--Sobolev regularity concerns weak derivatives in \(\mathcal{H}_\rho\), not continuity in the ambient norm \(\Vert\cdot\Vert_{\mathcal{Y}}\).

Fix \(p,q>1\), let \(p'=p/(p-1)\) and \(q'=q/(q-1)\), set \(r=p'\vee q'\), and take \(c>rT_\tau\). For \(B\in\mathrm{St}(\mathcal{X},d)\), let \(\zeta_B=(B^*)_\#\eta\).

\begin{Definition}[Gaussian--Sobolev admissibility]
\label{def:gaussian_sobolev_admissibility}
A Borel field \(u:[0,T_\tau]\times\mathbb{R}^d\times\mathcal{Y}\to\mathcal{Y}\) is Gaussian--Sobolev admissible at \(B\) if, for \(\zeta_B\)-a.e. \(s\), it is \(\mathcal{H}_\rho\)-valued \(dt\otimes\rho\)-a.e., \(u(t,s,\cdot)\in W_\rho^{1,q}(\rho;\mathcal{H}_\rho)\) for a.e. \(t\), and
\begin{align}
    &\int_0^{T_\tau}
    \Vert u(t,s,\cdot)\Vert_{L^p(\rho;\mathcal{H}_\rho)}\,dt<\infty,
    \label{eq:gaussian_sobolev_field_lp}\\
    &\int^{T_{\tau}}_{0}\left(\Vert D_{\rho}u(t,s,\cdot)\Vert_{L^{q}(\rho;\mathcal{S}_{2}(\mathcal{H}_{\rho}))}+\Vert\operatorname{div}_{\rho}u(t,s,\cdot)\Vert_{L^{q}(\rho)}\right)dt<\infty,
    \label{eq:gaussian_sobolev_field_derivative}\\
    &\mathop{\operatorname{ess\,sup}}_{t\in[0,T_{\tau}]}\int_{\mathcal{Y}}\exp\!\left(c[\operatorname{div}_{\rho}u(t,s,y)]^{-}\right)\rho(dy)<\infty,
    \label{eq:gaussian_sobolev_field_divergence}
\end{align}
where \([a]^-\coloneqq\max\{-a,0\}\).
\end{Definition}

Definition~\ref{def:gaussian_sobolev_admissibility} gives the velocity-side assumptions used in the uniqueness argument. The \(\mathcal{H}_\rho\)-valued and \(L^p\) conditions make the flux integrable when paired with an \(L^r\) density, while the weak derivative, divergence and exponential-divergence conditions provide the regularity required by Proposition~\ref{prop:gaussian_continuity_equation_uniqueness}. The choice \(c>rT_\tau\) satisfies the required exponential-divergence threshold on \((0,T_\tau)\). These conditions ensure at most one weak solution in the prescribed density class; existence is supplied separately by the conditional Monge interpolation.

For \(\eta\)-a.e. \(x\), let
\[
    T_x(y)\coloneqq T_{\mathcal{Y}}^\star(x,y),
    \qquad
    \mu_t^x\coloneqq\bigl((1-t)\mathrm{id}+tT_x\bigr)_\#\rho.
\]
To place that interpolation in the density class covered by the uniqueness theorem, we impose the complementary measure-side condition below.
\begin{Assumption}[Gaussian interpolation compression]
\label{ass:gaussian_interpolation_compression}
The source is the nondegenerate
centred Gaussian measure \(\rho=\mathcal{N}(0,C_\rho)\) as above. For \(\eta\)-a.e. \(x\) and a.e. \(t\in(0,T_\tau)\), there is a density \(f_t^x\) such that
\begin{equation}
    \mu_t^x=f_t^x\rho,
    \qquad
    f^x\in L^\infty\bigl((0,T_\tau);L^r(\rho)\bigr).
    \label{eq:gaussian_interpolation_compression}
\end{equation}
\end{Assumption}
Assumption~\ref{ass:gaussian_interpolation_compression} requires absolute continuity with an \(L^r\) bound along the truncated conditional Monge interpolation. It therefore represents each \(\mu_t^x\) by a density relative to the common Gaussian source and places that density in the class in which uniqueness holds. The condition is one-sided: it does not require the reverse absolute continuity and is therefore weaker than quasi-invariance.

Admissibility and compression play complementary roles. Compression alone does not show that an interpolation satisfies a continuity equation. Rather, the zero-loss identity identifies its velocity as \(u(t,s,\cdot)\), after which admissibility and compression permit the uniqueness theorem to be applied. The Wiener-space transport results of \textcite{FeyelUstunel2004} and \textcite{BogachevKolesnikov2013} motivate the Cameron--Martin and Malliavin--Sobolev structure, but they concern the Cameron--Martin transport cost and do not establish \eqref{eq:gaussian_interpolation_compression} or \eqref{eq:gaussian_sobolev_field_lp}--\eqref{eq:gaussian_sobolev_field_divergence} for the ambient quadratic COT map used here. The more general flow theory of \textcite{AmbrosioTrevisan2014} similarly requires a suitable reference measure, weak derivative and divergence control, and bounded compression.

\begin{Example}[Conditional Cameron--Martin translations]
\label{ex:conditional_cameron_martin_translation}
Suppose that, for a Borel map \(\overline h:\mathbb{R}^d\to\mathcal{H}_\rho\),
\[
    \mathbb{P}_{Y_1\mid X}(\cdot\mid x)
    =\bigl(\mathrm{id}+\overline h(B_0^*x)\bigr)_\#\rho
    \qquad\text{for }\eta\text{-a.e. }x.
\]
Then \(T_x=\mathrm{id}+\overline h(B_0^*x)\) is the ambient-quadratic optimal map: Jensen's inequality bounds the transport cost below by the squared distance between the two means, and the translation attains that bound. With the Gaussian linear functional of Definition~\ref{def:gaussian_linear_functional}, Proposition~\ref{prop:cameron_martin_translation_formula} gives
\[
    f^{x}_{t}(z)=\exp\left(t\widehat{h_{x}}(z)-\frac{t^{2}}{2}\Vert h_{x}\Vert^{2}_{\mathcal{H}_{\rho}}\right),\qquad h_{x}=\overline{h}(B^{*}_{0}x),
\]
and hence, for every \(r>1\),
\[
    \sup_{t\in(0,T_{\tau})}\Vert f^{x}_{t}\Vert_{L^{r}(\rho)}\leq\exp\left(\frac{r-1}{2}T^{2}_{\tau}\Vert h_{x}\Vert^{2}_{\mathcal{H}_{\rho}}\right)<\infty.
\]
Thus Assumption~\ref{ass:gaussian_interpolation_compression} holds. Moreover, the exact velocity \(u(t,s,y)=\overline h(s)\) is Gaussian--Sobolev admissible at \(B_0\): its Malliavin derivative vanishes and \(\operatorname{div}_\rho u(t,s,\cdot)=-\widehat{\overline h(s)}\), whose one-sided exponential moments are finite. 

For the concrete source \(\rho=\mathrm{Law}(\sigma W)\) on \(L^2([0,1])\), where \(W\) is a standard Wiener process, Proposition~\ref{prop:wiener_source_cameron_martin_space} identifies \(\mathcal{H}_\rho\) with the absolutely continuous paths that start at zero and have square-integrable derivative. Consequently, the sine response directions in the common-shape Gaussian-noise designs FF1--FF3 in Section~\ref{sec:numerics} belong to \(\mathcal{H}_\rho\), so those designs fall within the translation example. The example does not generally include FF4: Proposition~\ref{prop:feldman_hajek_criterion} shows that a non-trivial scalar covariance rescaling in infinitely many response directions produces a Gaussian measure singular to \(\rho\). Current Gaussian continuity-equation tools therefore do not verify Assumption~\ref{ass:gaussian_interpolation_compression} for FF4. 
\end{Example}

The next lemma combines these two sides of the argument. The Monge construction supplies the interpolation measures, compression represents them by densities in the uniqueness class, and zero population loss makes the densities associated with covariate values in the same fibre solve a common weak Gaussian continuity equation with initial density one. Gaussian--Sobolev admissibility then gives uniqueness within that class. Equality of the intermediate laws, together with the common velocity, yields equality of the conditional endpoint laws and hence identifies the proposed reduction.

\begin{Lemma}[Gaussian--Sobolev zero-loss identification]
\label{lem:gaussian_sobolev_zero_loss}
Suppose that Assumptions~\ref{ass:linear_statistical_target} and~\ref{ass:gaussian_interpolation_compression} hold. Let \(B\in\mathrm{St}(\mathcal{X},d)\), and let \(u\) be Gaussian--Sobolev admissible at \(B\). If
\begin{equation}
    \int_0^{T_\tau}
    \mathbb{E}_{\gamma^\star}\!\left[
    \Vert V-u(t,B^*X,Y_t)\Vert_{\mathcal{Y}}^2
    \right]dt=0,
    \label{eq:gaussian_sobolev_zero_loss}
\end{equation}
then \(Y_1\indep X\mid B^*X\). Consequently, \(\operatorname{ran}(B)=\mathcal{S}_0\).
\end{Lemma}

\begin{proof}
Set \(S=B^*X\), and disintegrate \(\eta\) with respect to \(S\) as \(\eta(dx)=\int_{\mathbb{R}^d}\eta_B(dx\mid s)\,\zeta_B(ds)\).
For \(\eta\)-a.e. \(x\), define \(g_x=T_x-\mathrm{id}\) and \(F_t^x=\mathrm{id}+tg_x\). Conditional on \(X=x\), the endpoint coupling is \((\mathrm{id},T_x)_\#\rho\). Hence \eqref{eq:gaussian_sobolev_zero_loss} and Fubini's theorem give, for \(\zeta_B\)-a.e. \(s\) and \(\eta_B(\cdot\mid s)\)-a.e. \(x\),
\begin{equation}
    g_x(y)=u(t,s,F_t^x(y))
    \qquad\text{for }dt\otimes\rho(dy)\text{-a.e. }(t,y).
    \label{eq:gaussian_sobolev_path_identity}
\end{equation}
Fix such \((s,x)\). There is a time \(t\in(0,T_\tau)\) for which \eqref{eq:gaussian_sobolev_path_identity} holds, \(u(t,s,\cdot)\in L^p(\rho;\mathcal{H}_\rho)\), and \(f_t^x\in L^r(\rho)\). Since \(u(t,s,\cdot)\) is \(\mathcal{H}_\rho\)-valued \(\rho\)-a.e.\ and \(\mu_t^x\ll\rho\), identity~\eqref{eq:gaussian_sobolev_path_identity} first shows that \(g_x\) is \(\mathcal{H}_\rho\)-valued \(\rho\)-a.e. The continuous inclusion \(\mathcal{H}_\rho\hookrightarrow\mathcal{Y}\) is an injective Borel map between Polish spaces. By the Lusin--Souslin theorem, its image is Borel in \(\mathcal{Y}\) and its inverse on that image is Borel. Thus, after modification on a \(\rho\)-null set, \(g_x\) is strongly measurable as an \(\mathcal{H}_\rho\)-valued map. Since \(r\geq p'\), the change of variables \(\mu_t^x=(F_t^x)_\#\rho=f_t^x\rho\) and H\"older's inequality then give
\[
   \int_{\mathcal{Y}}\Vert g_{x}(y)\Vert_{\mathcal{H}_{\rho}}\,\rho(dy)=\int_{\mathcal{Y}}\Vert u(t,s,z)\Vert_{\mathcal{H}_{\rho}}f^{x}_{t}(z)\,\rho(dz)\leq\Vert u(t,s,\cdot)\Vert_{L^{p}(\rho;\mathcal{H}_{\rho})}\Vert f^{x}_{t}\Vert_{L^{p'}(\rho)}<\infty.
\]
Thus \(g_x\) is Bochner integrable as an \(\mathcal{H}_\rho\)-valued map.

Let \(\varphi\in\operatorname{Cyl}_{\rho}^{\infty}(\mathcal{Y})\) be as in Definition~\ref{def:smooth_gaussian_cylindrical_function}. By \eqref{eq:appendix_standardised_gaussian_coordinate}, each coordinate entering \(\varphi\) is a continuous linear functional on \(\mathcal{Y}\). Since the profile of \(\varphi\) has bounded derivatives, \(\varphi\) has a bounded ambient derivative and is Lipschitz in the ambient norm. Moreover, its ambient directional derivative along \(h\in\mathcal{H}_\rho\) is \(D\varphi(z)[h]=\langle D_\rho\varphi(z),h\rangle_{\mathcal{H}_\rho}\). The preceding integrability, the continuous embedding \(\mathcal{H}_\rho\hookrightarrow\mathcal{Y}\), and the chain rule therefore show that \(t\mapsto\int_{\mathcal{Y}}\varphi(F_t^x(y))\,\rho(dy)\) is absolutely continuous. For a.e. \(t\), \eqref{eq:gaussian_sobolev_path_identity} and the pushforward identity yield
\begin{align}
    \frac{d}{dt}\int_{\mathcal{Y}}\varphi(z)f_t^x(z)\,\rho(dz)
    &=\int_{\mathcal{Y}}
    \langle D_\rho\varphi(F_t^x(y)),g_x(y)\rangle_{\mathcal{H}_\rho}\,\rho(dy)\notag\\
    &=\int_{\mathcal{Y}}
    \langle D_\rho\varphi(z),u(t,s,z)\rangle_{\mathcal{H}_\rho}
    f_t^x(z)\,\rho(dz).
    \label{eq:conditional_gaussian_continuity_equation}
\end{align}
The right-hand side is integrable in time by H\"older's inequality, \eqref{eq:gaussian_sobolev_field_lp}, and \eqref{eq:gaussian_interpolation_compression}. The embedding \(\mathcal{H}_\rho\hookrightarrow\mathcal{Y}\) is continuous, so the coupling \((\mathrm{id},F_t^x)_\#\rho\) gives
\[
    W_1(\mu_t^x,\rho)
    \leq t\int_{\mathcal{Y}}\Vert g_x(y)\Vert_{\mathcal{Y}}\,\rho(dy)
    \longrightarrow0
    \qquad\text{as }t\downarrow0.
\]
For these cylindrical tests, the preceding \(W_1\)-convergence gives the initial condition in \eqref{eq:appendix_gaussian_initial_condition}. Thus \(f^x\) is a non-negative \(L^\infty((0,T_\tau);L^r(\rho))\) weak solution, in the sense of Definition~\ref{def:weak_gaussian_continuity_equation}, with initial density one.

For the fixed value of \(s\), conditions~\eqref{eq:gaussian_sobolev_field_lp}--\eqref{eq:gaussian_sobolev_field_divergence} verify the hypotheses of Proposition~\ref{prop:gaussian_continuity_equation_uniqueness} with \(b_t=u(t,s,\cdot)\). Hence any two of the above \(L^r\)-bounded density solutions with initial density one coincide. It follows that, for \(\eta_B(\cdot\mid s)\otimes\eta_B(\cdot\mid s)\)-a.e. \((x,\widetilde x)\),
\begin{equation}
    \mu_t^x=\mu_t^{\widetilde x}
    \qquad\text{for a.e. }t\in(0,T_\tau).
    \label{eq:conditional_interpolation_equality}
\end{equation}
Choose a time at which \eqref{eq:gaussian_sobolev_path_identity} holds for both \(x\) and \(\widetilde x\), and \eqref{eq:conditional_interpolation_equality} holds. Since \(T_x(y)=F_t^x(y)+(1-t)g_x(y)\), we obtain
\[
   \mathbb{P}_{Y_{1}\mid X}(\cdot\mid x)=\bigl(\mathrm{id}+(1-t)u(t,s,\cdot)\bigr)_{\#}\mu^{x}_{t}=\bigl(\mathrm{id}+(1-t)u(t,s,\cdot)\bigr)_{\#}\mu^{\widetilde{x}}_{t}=\mathbb{P}_{Y_{1}\mid X}(\cdot\mid\widetilde{x}).
\]

It remains to pass from slicewise equality to a measurable conditional-law factorisation. Let \(K(x)=\mathbb{P}_{Y_1\mid X}(\cdot\mid x)\), and let \(\{C_m\}_{m\in\mathbb{N}}\) be a countable determining class for \(\mathcal{P}(\mathcal{Y})\). The preceding equality implies that, conditionally on \(S=s\), two independent copies of \(K(X)(C_m)\) are equal a.s. Therefore
\[
    \operatorname{Var}(K(X)(C_{m})\mid S)=0\qquad\text{a.s.},
\]
so each \(K(X)(C_m)\) admits a \(\sigma(S)\)-measurable version. Lemma~\ref{lem:simultaneous_kernel_version} and the Doob--Dynkin lemma give a Borel kernel \(\kappa_B:\mathbb{R}^d\to\mathcal{P}(\mathcal{Y})\) such that
\[
    \mathbb{P}_{Y_1\mid X}(\cdot\mid X)=\kappa_B(B^*X)
    \qquad\text{a.s.}
\]
Proposition~\ref{prop:sdr_MK} yields \(Y_1\indep X\mid B^*X\). The minimality and dimension assumption in Assumption~\ref{ass:linear_statistical_target} then gives \(\operatorname{ran}(B)=\mathcal{S}_0\).
\end{proof}

The lemma identifies the zeros of the population criterion. To turn this population result into consistency of an empirical minimiser, we also need uniform convergence of the fitted criterion. As in the finite-dimensional case, Gaussian--Sobolev regularity alone does not make point evaluation stable under \(W_2\)-perturbations of a coupling, so convergence of the empirical coupling is not sufficient by itself.

For a fixed finite-dimensional subspace \(E_K\subset\mathcal{X}\), let \(\gamma_K^\star\) be the triangular COT coupling between the laws of \((P_KX,Y_0)\) and \((P_KX,Y_1)\). For \(B\in\mathrm{St}(E_K,d)\), define \(Q_K(B,\theta)\) by \eqref{eq:population_linear_risk} with \(X\) replaced by \(P_KX\) and \(\gamma^\star\) replaced by \(\gamma_K^\star\). Let \(\widehat Q_{n,K}\) denote its empirical counterpart based on a relaxed empirical coupling of the projected observations.

\begin{Theorem}[Fixed-projection consistency]
\label{thm:function_on_function_gaussian_sobolev}
Let \(E_K\subset\mathcal{X}\) be fixed and finite-dimensional. Suppose that, with the covariate replaced by \(P_KX\), Assumptions~\ref{ass:linear_statistical_target} and~\ref{ass:gaussian_interpolation_compression} hold with central subspace \(\mathcal{S}_0=\operatorname{ran}(B_0)\subseteq E_K\). Let \(\Theta\) be compact and let \(\{u_{\mathcal{Y}}^\theta:\theta\in\Theta\}\) be Borel fields such that \(u_{\mathcal{Y}}^\theta\) is Gaussian--Sobolev admissible at every \(B\in\mathrm{St}(E_K,d)\). Assume that:
\begin{enumerate}[label=\roman*)]
    \item there is a \(\theta_0\in\Theta\) for which \(Q_K(B_0,\theta_0)=0\), where \(Q_K\) is the population loss based on \(\gamma_K^\star\);
    \item \((B,\theta)\mapsto Q_K(B,\theta)\) is lower semicontinuous on \(\mathrm{St}(E_K,d)\times\Theta\);
    \item the empirical criterion satisfies
    \begin{equation}
        e_{n,K}\coloneqq
        \sup_{B\in\mathrm{St}(E_K,d),\,\theta\in\Theta}
        |\widehat Q_{n,K}(B,\theta)-Q_K(B,\theta)|
        \xrightarrow{p}0.
        \label{eq:gaussian_sobolev_uniform_risk}
    \end{equation}
\end{enumerate}
If \((\widehat B_{n,K},\widehat\theta_{n,K})\) is an \(r_n\)-approximate minimiser of \(\widehat Q_{n,K}\), where \(r_n\geq0\) and \(r_n\xrightarrow{p}0\), then
\[
    d_{\mathrm{pr}}(\widehat B_{n,K},B_0)\xrightarrow{p}0.
\]
\end{Theorem}

\begin{proof}
The space \(\mathrm{St}(E_K,d)\times\Theta\) is compact. Lemma~\ref{lem:gaussian_sobolev_zero_loss} shows that every zero of \(Q_K\) has \(\operatorname{ran}(B)=\mathcal{S}_0\). Consequently, for each \(\delta>0\), lower semicontinuity gives
\[
    c_{\delta,K}\coloneqq
    \inf_{\substack{B\in\mathrm{St}(E_K,d),\,\theta\in\Theta:\
    d_{\mathrm{pr}}(B,B_0)\geq\delta}}
    Q_K(B,\theta)>0
\]
whenever the constraint set is non-empty. Approximate optimality, condition i), and \eqref{eq:gaussian_sobolev_uniform_risk} give
\(Q_K(\widehat B_{n,K},\widehat\theta_{n,K})\leq 2e_{n,K}+r_n\). If \(d_{\mathrm{pr}}(\widehat B_{n,K},B_0)\geq\delta\), the left-hand side is at least \(c_{\delta,K}\). Hence, whenever the constraint set is non-empty,
\[
    \mathbb{P}\bigl(d_{\mathrm{pr}}(\widehat B_{n,K},B_0)\geq\delta\bigr)
    \leq
    \mathbb{P}\bigl(2e_{n,K}+r_n\geq c_{\delta,K}\bigr)
    \longrightarrow0.
\]
If the constraint set is empty, the probability is zero. This proves the claimed consistency.
\end{proof}

The preceding theorem keeps the covariate projection \(E_K\) fixed. For a genuinely infinite-dimensional covariate, consistency for the unprojected central subspace requires the projection dimension to increase, introducing both a sieve approximation problem and a uniform identification problem. Let \(E_K\uparrow\mathcal{X}\), with orthogonal projections satisfying \(P_Kx\to x\) for every \(x\in\mathcal{X}\), and let \(\mathcal{U}_K=\{u_{K,\theta}:\theta\in\Theta_K\}\) be classes of Borel fields for which the population criterion below is finite for every \(B\in\mathrm{St}(E_K,d)\). For \(B\in\mathrm{St}(E_K,d)\) and \(u\in\mathcal{U}_K\), write
\[
    Q(B,u)\coloneqq\int_0^{T_\tau}\mathbb{E}_{\gamma^\star}
    \left[\Vert V-u(t,B^*X,Y_t)\Vert_{\mathcal{Y}}^2\right]dt
\]
and define the sieve approximation error
\begin{equation}
    A_K
    \coloneqq
    \inf_{B\in\mathrm{St}(E_K,d),\,u\in\mathcal{U}_K}Q(B,u).
    \label{eq:functional_sieve_approximation}
\end{equation}
For any coupling \(\gamma\) on the endpoint space, let \(Q_\gamma(B,u)\) denote the same criterion with \(\gamma^\star\) replaced by \(\gamma\). Let \(\widehat\gamma_{n,K}\) be an embedded relaxed empirical coupling based on the projected observations, and write \(\widehat Q_{n,K}(B,u)=Q_{\widehat\gamma_{n,K}}(B,u)\). Since \(B\) has range in \(E_K\), \(B^*P_KX=B^*X\).

When the Gaussian assumptions above hold and the fields are admissible, Lemma~\ref{lem:gaussian_sobolev_zero_loss} rules out an exactly zero population loss at an incorrect subspace for each fixed \(K\). It does not, however, give separation uniformly over growing sieves: a sequence of incorrect subspaces could have positive loss converging to zero. The next theorem therefore takes uniform sieve separation as a distinct statistical identifiability condition. 

\begin{Theorem}[Growing-projection consistency]
\label{thm:function_on_function_growing_projection}
Let \(B_0\in\mathrm{St}(\mathcal{X},d)\) be the target basis. Suppose that \(E_K\uparrow\mathcal{X}\), \(A_K\to0\), and, for every \(\delta>0\),
\begin{equation}
    \liminf_{K\to\infty}
    \inf_{\substack{B\in\mathrm{St}(E_K,d),\,u\in\mathcal{U}_K:\\
    d_{\mathrm{pr}}(B,B_0)\geq\delta}}
    Q(B,u)>0.
    \label{eq:functional_sieve_separation}
\end{equation}
Let \(K_n\uparrow\infty\), and suppose that the corresponding empirical criteria satisfy
\begin{equation}
    e_{n,K_n}\coloneqq
    \sup_{B\in\mathrm{St}(E_{K_n},d),\,u\in\mathcal{U}_{K_n}}
    |\widehat Q_{n,K_n}(B,u)-Q(B,u)|
    \xrightarrow{p}0.
    \label{eq:functional_growing_uniform_risk}
\end{equation}
If \((\widehat B_n,\widehat u_n)\in\mathrm{St}(E_{K_n},d)\times\mathcal{U}_{K_n}\) satisfies
\[
    \widehat Q_{n,K_n}(\widehat B_n,\widehat u_n)
    \leq
    \inf_{B\in\mathrm{St}(E_{K_n},d),\,u\in\mathcal{U}_{K_n}}
    \widehat Q_{n,K_n}(B,u)+r_n,
    \qquad r_n\geq0,\quad r_n\xrightarrow{p}0,
\]
then
\begin{equation}
    d_{\mathrm{pr}}(\widehat B_n,B_0)\xrightarrow{p}0.
    \label{eq:functional_subspace_consistency}
\end{equation}
\end{Theorem}

\begin{proof}
Choose deterministic \((B_n^\circ,u_n^\circ)\) whose population risk is at most \(A_{K_n}+n^{-1}\). Approximate empirical optimality and \eqref{eq:functional_growing_uniform_risk} give
\[
    Q(\widehat B_n,\widehat u_n)
    \leq A_{K_n}+2e_{n,K_n}+r_n+n^{-1}
    \xrightarrow{p}0.
\]
For a fixed \(\delta>0\), condition \eqref{eq:functional_sieve_separation} supplies a constant \(c_\delta>0\) such that, for all sufficiently large \(n\), \(d_{\mathrm{pr}}(B,B_0)\geq\delta\) implies \(Q(B,u)\geq c_\delta\) for \(B\in\mathrm{St}(E_{K_n},d)\) and \(u\in\mathcal{U}_{K_n}\). Therefore
\[
    \mathbb{P}\left(d_{\mathrm{pr}}(\widehat B_n,B_0)\geq\delta\right)
    \leq\mathbb{P}\left(Q(\widehat B_n,\widehat u_n)\geq c_\delta\right)
    \longrightarrow0,
\]
which proves \eqref{eq:functional_subspace_consistency}.
\end{proof}

The condition \(A_K\to0\) is an approximation requirement. It holds, for example, in the translation model of Example~\ref{ex:conditional_cameron_martin_translation} when \(\overline h\) is bounded and continuous, aligned bases \(B_K\in\mathrm{St}(E_K,d)\) satisfy \(B_K\to B_0\) in Hilbert--Schmidt norm, and \(\mathcal{U}_K\) contains \(u_K(t,s,y)=\overline h(s)\). Indeed, \(B^{*}_{K}X\xrightarrow{p}B^{*}_{0}X\), continuity and boundedness of \(\overline h\) then imply
\[
    Q(B_{K},u_{K})=T_{\tau}\,\mathbb{E}[\Vert\overline{h}(B^{*}_{0}X)-\overline{h}(B^{*}_{K}X)\Vert^{2}_{\mathcal{Y}}]\to 0.
\]
This example uses continuity only in the finite-dimensional index; the velocity is constant in the response variable and no ambient response continuity is assumed.

Condition~\eqref{eq:functional_growing_uniform_risk} absorbs projected-coupling, time-sampling, and optimisation-independent approximation errors. Proposition~\ref{prop:relaxed_empirical_coupling_consistency} shows that \(K_n\) and \(\varepsilon_n\) may be chosen so that the embedded input coupling converges in \(W_2\), but Gaussian--Sobolev regularity alone does not turn that coupling convergence into \eqref{eq:functional_growing_uniform_risk}. Verifying uniform criterion convergence requires complexity and envelope control for the fitted operator class, while verifying growing-sieve separation requires additional uniform identifiability. Thus the fixed-projection result combines Gaussian continuity-equation identification with an empirical-process condition, whereas the growing-projection result additionally assumes sieve approximation and uniform separation. Neither theorem claims a rate or inference for the fully trained neural procedure.

%% file: simulation_studies.tex
\subsection{Simulation studies}

We evaluate SDR-COT in four settings: nonlinear Euclidean SDR, linear Euclidean SDR, linear SDR with functional covariates, and linear SDR with both functional covariates and functional responses. In each experiment the structural dimension is treated as known and supplied to all competing methods. Except for the final function-on-function experiment, performance is measured on an independent evaluation sample by the distance correlation \parencite{szekely2007measuring} between the estimated representation and the true sufficient predictor; for linear Euclidean SDR we also report a subspace error. Unless otherwise stated, all reported entries are averages with standard deviations in parentheses over $50$ replications.

\subsubsection*{Nonlinear SDR in the Euclidean case}

For the nonlinear Euclidean experiments, $R_{\theta}$ and $u^{\phi}_{\mathcal{Y}}$ are parameterised by SiLU multilayer perceptrons and trained with Adam. We compare SDR-COT with GenSDR \parencite{xu2025conditional}, a generative nonlinear SDR method based on conditional stochastic interpolation, and GSIR \parencite{Li2018}, implemented through the R package \emph{nsdr}. The GenSDR implementation follows the algorithmic specification in Section~4.2 of \textcite{xu2025conditional}. Each method is trained on $2000$ observations and evaluated on $1000$ independent observations.

The three data-generating mechanisms are designed so that the conditional law of $Y$ depends on $X$ through one, two, and three nonlinear sufficient predictors, respectively. In all models, the components of $X$ are independent of the auxiliary noise variables in $W$.
\begin{enumerate}[label=\roman*)]
    \item Model $A$ is a multiplicative-scale model,
    \[
        Y=h_0(X)W,\qquad h_0(X)=(1+\vert X_2\vert)^{-2}\exp(X_1),
    \]
    where
    \[
        W=\bigl(W_1,\; W_2W_4+(W_3+10)(1-W_4)-\exp(W_1)\bigr)^T,
    \]
    with $W_1\sim\mathrm{Gamma}(3,5)$, $W_2,W_3\sim t(3)$, and $W_4\sim\mathrm{Bernoulli}(0.5)$. The sufficient predictor is one-dimensional.
    \item Model $B$ is a two-dimensional location model,
    \[
        Y=h_0(X)+\gamma W,\qquad W=(W_1,W_2)^T,
    \]
    where $W_1\sim\chi^2_3$, $W_2\sim\chi^2_5$, and
    \[
        h_{0}(X)=\left(X^{2}_{1}\exp(2X^{2}_{4}),\;(1+3X_{2}-X_{3})^{4/3}\mathbbm{1}_{\{\sqrt{X_{1}}>0.8\}}\right)^{T}.
    \]
    \item Model $C$ combines nonlinear location and scale effects. For $W=(W_1,\ldots,W_{d_y})^T$ with $W_j\stackrel{\mathrm{i.i.d.}}{\sim}\mathrm{Gamma}(3,1)$,
    \[
        Y=h_0(X)+\operatorname{diag}\{h_1(X)\}W,
    \]
    where
    \[
        h_0(X)=\bigl(4\cos(\pi X_3),1,\ldots,1\bigr)^T
    \]
    and
    \[
        h_1(X)=\left(5\left(\frac{X_1+\cdots+X_6}{6}\right)^3,\;\log(1+3X_2),1,\ldots,1\right)^T.
    \]
\end{enumerate}
Except in Model $A$, the covariates are generated from $\mathcal{U}([0,1]^{d_x})$. For Model $A$ we consider the three covariate distributions shown in Table~\ref{tab:nonlinear_euclidean}. Model $B$ varies the noise scale $\gamma$, and Model $C$ varies the response dimension $d_y$.

\begin{table}[H]
\centering
\caption{Average distance correlation, with standard deviations in parentheses, over $50$ replications in nonlinear Euclidean SDR.}
\label{tab:nonlinear_euclidean}
\begin{tabular}{llccc}
\toprule
Model & Configuration & SDR-COT & GenSDR & GSIR \\
\midrule
A 
& $\mathcal{U}([0,1]^{d_x})$ 
& 0.951 (0.009) & 0.955 (0.008) & \textbf{0.957 (0.004)} \\

A
& $\mathcal{N}(\mathbf{1}_{d_x}, I_{d_x})$ 
& 0.779 (0.116) & \textbf{0.828 (0.013)} & 0.825 (0.011) \\

A
& $\mathcal{N}(\mathbf{1}_{d_x}, HH^\top)$ 
& 0.779 (0.129) & \textbf{0.831 (0.014)} & 0.824 (0.011) \\
\midrule
B
& $\gamma=0.1$ 
& 0.829 (0.011) & \textbf{0.861 (0.017)} & 0.795 (0.008) \\

B
& $\gamma=0.3$ 
& \textbf{0.840 (0.016)} & 0.836 (0.013) & 0.792 (0.009) \\

B
& $\gamma=0.5$ 
& \textbf{0.815 (0.022)} & 0.773 (0.029) & 0.783 (0.011) \\
\midrule
C
& $d_y=5$ 
& \textbf{0.896 (0.015)} & 0.862 (0.019) & 0.808 (0.010) \\

C
& $d_y=10$ 
& \textbf{0.877 (0.016)} & 0.814 (0.026) & 0.799 (0.010) \\

C
& $d_y=20$ 
& \textbf{0.860 (0.021)} & 0.670 (0.038) & 0.787 (0.012) \\
\bottomrule
\end{tabular}
\end{table}

Table~\ref{tab:nonlinear_euclidean} shows that SDR-COT is competitive with the nonlinear generative benchmark and improves over GSIR in most configurations. Its advantage is most pronounced in Model $C$, where information about the response is carried jointly by location and scale and the response dimension increases. In Model $A$ under non-uniform covariate distributions, SDR-COT is less stable than GenSDR and GSIR, indicating that the relaxed empirical coupling can be sensitive to the geometry of the covariate distribution.

\subsubsection*{Linear SDR in the Euclidean case}

We next constrain the reduction to be linear, $R_B(X)=B^TX$ with $B\in\mathrm{St}(d_x,d)$. The conditional velocity is represented by a low-order polynomial in time,
\[
    u_{\phi}(t,r,y)=\sum_{\ell=0}^{L-1}t^{\ell}c_{\phi,\ell}(r,y),\qquad r=B^TX,\quad y=Y_t,
\]
where the coefficient functions are produced by a shared neural network. We use $L=3$ and optimise $B$ on the Stiefel manifold. The competitors are SAVE \parencite{Cook2000} and MAVE, implemented via the R package \texttt{MAVE} \parencite{R-MAVE}.

The covariate dimension is $d_x=10$, the structural dimension is $d=2$, and $X_j\stackrel{\mathrm{i.i.d.}}{\sim}\mathcal{N}(2,1)$. The response is
\[
    Y=2\sin\left(\frac{\pi(X_1+X_2)}{20}\right)+3\sin^2\left(\frac{\pi(X_3+3X_4)}{30}\right)E,
\]
where $E\sim\mathcal{N}(0,1)$ in Model I and $E\sim t_3/\sqrt{3}$ in Model II\@. The central subspace is spanned by
\[
    b_1=(1,1,0,\ldots,0)^T,
    \qquad
    b_2=(0,0,1,3,0,\ldots,0)^T.
\]
Let \(B_0=(b_1/\|b_1\|,b_2/\|b_2\|)\), whose columns are orthonormal.
Each replication uses $2000$ training observations and $2000$ independent evaluation observations. In addition to distance correlation, we report the projection error
\[
   \Delta_{\mathrm{proj}}=\frac{\Vert\widehat{B}\widehat{B}^{T}-B_0B_0^{T}\Vert_{F}}{\sqrt{2d}},
\]
which lies in $[0,1]$ and is invariant to the choice of basis for the estimated subspace.

\begin{table}[H]
\centering
\caption{Average distance correlation and projection error, with standard deviations in parentheses, over $50$ replications in linear Euclidean SDR.}
\label{tab:linear_euclidean}
\begin{tabular}{llcc}
\toprule
Model & Method & Distance correlation & $\Delta_{\mathrm{proj}}$ \\
\midrule
I  & SDR-COT & 0.889 (0.082) & 0.398 (0.150) \\
I  & SAVE    & \textbf{0.913 (0.052)} & \textbf{0.363 (0.101)} \\
I  & MAVE    & 0.703 (0.071) & 0.678 (0.080) \\
%\addlinespace
\midrule
%\cmidrule(lr){2-4}
II & SDR-COT & \textbf{0.962 (0.019)} & \textbf{0.238 (0.057)} \\
II & SAVE    & 0.953 (0.017) & 0.269 (0.050) \\
II & MAVE    & 0.706 (0.072) & 0.678 (0.087) \\
\bottomrule
\end{tabular}
\end{table}

The linear results in Table~\ref{tab:linear_euclidean} are consistent with the exhaustiveness-oriented nature of the objective. SDR-COT does not optimise the projection error directly, but it still recovers a competitive estimate of the central subspace. SAVE is slightly better in the Gaussian-noise setting, whereas SDR-COT is strongest under the heavier-tailed $t_3$ noise.

\subsubsection*{Linear SDR for functional covariates}

The next experiment applies Algorithm~\ref{alg:functional_sdr_cot_projection} to functional covariates with scalar responses. Each centred curve is projected to a coefficient vector $C(X)\in\mathbb{R}^K$, and SDR-COT fits a linear representation $B^TC(X)$ with $B\in\mathrm{St}(K,d)$. Four models have a one-dimensional sufficient predictor after projection, and Model F5 has a two-dimensional sufficient predictor. We compare against the RKHS inverse-regression estimator of \textcite{hsing2009rkhs}, using the same truncation rank $K$.

The basis dimensions are $K=8$ for F1, F3, and F5, $K=20$ for F2, and $K=24$ for F4. Accordingly, SDR-COT is fitted on $\mathrm{St}(8,1)$ for F1 and F3, $\mathrm{St}(20,1)$ for F2, $\mathrm{St}(24,1)$ for F4, and $\mathrm{St}(8,2)$ for F5. The curves are observed on $t_j=j/(J+1)$, $j=1,\ldots,J$. For Model F2, $X$ is fractional Brownian motion with Hurst parameter $H=0.75$ observed on $120$ grid points; for the other models, $X$ is a Wiener process observed on $100$ grid points. The additive noise, when present, is $\varepsilon\sim\mathcal{N}(0,0.3^2)$.

The five data-generating mechanisms are
\begin{enumerate}[label=\roman*)]
    \item Model F1:
    \[
        Y=\exp\left(\int_0^1\beta(t)X(t)\,dt\right)+\varepsilon,\qquad
        \beta(t)=\sin\left(\frac{3\pi t}{2}\right).
    \]
    \item Model F2:
    \[
        Y=\arctan\left(\sum_{i\in\{30,31,32,90,91,92\}}X(t_i)\right)+\varepsilon.
    \]
    \item Model F3:
    \[
        Y=e^\eta+\varepsilon,\qquad
        \eta=\int_0^1\beta(t)X(t)\,dt,\qquad
        \beta(t)=\sin(2\pi t)+\cos(4\pi t).
    \]
    \item Model F4:
    \[
        Y=(0.2+4\eta^2)Z,\qquad
        \eta=\int_0^1\beta(t)X(t)\,dt,\qquad
        \beta(t)=\sin\left(\frac{3\pi t}{2}\right),\quad Z\sim\mathcal{N}(0,1).
    \]
    \item Model F5:
    \[
        Y=\eta_1+0.75\eta_2^2+\varepsilon,\qquad
        \eta_k=\int_0^1\beta_k(t)X(t)\,dt,
    \]
    where $\beta_1(t)=\sin(2\pi t)$ and $\beta_2(t)=\cos(2\pi t)$.
\end{enumerate}
The evaluation metric is the distance correlation between the estimated finite-dimensional representation and the true latent index on an independent evaluation sample. For F1--F4 this is a scalar--scalar distance correlation; for F5 it is the multivariate distance correlation between $(\eta_1,\eta_2)$ and the two-dimensional estimated representation.

\begin{table}[H]
\centering
\caption{Average distance correlation, with standard deviations in parentheses, for functional linear SDR with $n_{\mathrm{train}}=n_{\mathrm{eval}}=2000$ over $50$ replications.}
\label{tab:functional_linear_sdr_cot}
\begin{tabular}{lrr}
\toprule
Model & SDR-COT & Hsing--Ren \\
\midrule
F1 & 0.980 (0.012) & \textbf{0.988 (0.006)} \\
F2 & 0.994 (0.002) & \textbf{0.998 (0.001)} \\
F3 & 0.989 (0.005) & \textbf{0.994 (0.003)} \\
F4 & \textbf{0.965 (0.010)} & 0.195 (0.124) \\
F5 & \textbf{0.948 (0.010)} & 0.808 (0.027) \\
\bottomrule
\end{tabular}
\end{table}

Table~\ref{tab:functional_linear_sdr_cot} shows that both methods recover the one-index mean-regression structures in F1--F3. The difference appears in F4 and F5. In F4 the conditional mean is zero and the information about the index is carried by the conditional scale, while F5 requires a two-dimensional sufficient predictor. SDR-COT remains accurate in both cases, which is consistent with its distributional objective.

\subsubsection*{SDR for a function-on-function model}

The final experiment applies Algorithm~\ref{alg:functional_functional_sdr_cot} to the function-on-function models of \textcite{li2022weakfunctional}. Let $X$ and $W$ be independent standard Wiener processes on $[0,1]$, let $\sigma=0.1$, and take
\[
    \beta_1(t)=v_1(t)=\sin\left(\frac{3\pi t}{2}\right),\qquad
    \beta_2(t)=v_2(t)=\sin\left(\frac{5\pi t}{2}\right).
\]
The four response models are
\begin{enumerate}[label=\roman*)]
    \item Model FF1:
    \[
        Y(t)=0.2\exp(1+\langle\beta_{1},X\rangle)v_{1}(t)+\sigma W(t).
    \]
    \item Model FF2:
    \[
        Y(t)=(3\langle\beta_{1},X\rangle)^{2}v_{1}(t)+\sigma W(t).
    \]
    \item Model FF3:
    \[
        Y(t)=3\langle\beta_{1},X\rangle v_{1}(t)+(9\langle\beta_{2},X\rangle)^{2}v_{2}(t)+\sigma W(t).
    \]
    \item Model FF4:
    \[
        Y(t)=5\langle \beta_1,X\rangle v_1(t)+8\langle \beta_2,X\rangle\sigma W(t).
    \]
\end{enumerate}
The sufficient predictor is one-dimensional for FF1--FF2 and two-dimensional for FF3--FF4. We match the simulation design of \textcite{li2022weakfunctional}: the functional covariate is observed at $11$ equally spaced points in $[0,1]$, $n=100$ and $n=200$ denote equal-sized training and testing samples, and the experiment is repeated $100$ times. In SDR-COT, the covariate is represented by the first $K=8$ Brownian Karhunen--Loève coefficients, the source response law is chosen as the law of $\sigma W$, and the functional response velocity is represented by a one-dimensional FNO.

Because the target sufficient predictor is finite-dimensional, we use the multiple correlation adopted by \textcite{li2022weakfunctional}. For two random vectors $U,V\in\mathbb{R}^d$ with sample covariance blocks $C_{UU}$, $C_{UV}$ and $C_{VV}$, the multiple correlation is defined as
\[
    \operatorname{mcorr}(U,V)\coloneqq
    \operatorname{tr}\left(C_{VV}^{-1/2}C_{VU}C_{UU}^{-1}C_{UV}C_{VV}^{-1/2}\right).
\]
This is the sum of squared sample canonical correlations, lies between $0$ and $d$, and is invariant to nonsingular linear transformations of the estimated sufficient predictors. We compute it on the test set between the true sufficient predictor and the estimated representation.

\begin{table}[H]
\centering
\caption{Average multiple correlation, with standard deviations in parentheses, for functional covariate and functional response SDR over $100$ replications. The final column gives the best entry reported by \textcite{li2022weakfunctional} among f-GSIR, WIRE, WAVE, and WDR.}
\label{tab:functional_functional_sdr_cot}
\begin{tabular}{llcc}
\toprule
Model & $n$ & SDR-COT & Best Li--Song method \\
\midrule
FF1 & 100 & 0.977 (0.015) & 0.982 (WIRE) \\
FF1 & 200 & \textbf{0.986 (0.011)} & 0.986 (WIRE) \\
\midrule
FF2 & 100 & 0.959 (0.040) & \textbf{0.964 (WDR)} \\
FF2 & 200 & 0.974 (0.020) & \textbf{0.974 (WAVE/WDR)} \\
\midrule
FF3 & 100 & \textbf{1.945 (0.036)} & 1.772 (WAVE) \\
FF3 & 200 & \textbf{1.967 (0.023)} & 1.807 (WAVE) \\
\midrule
FF4 & 100 & \textbf{1.945 (0.045)} & 1.681 (WIRE) \\
FF4 & 200 & \textbf{1.966 (0.024)} & 1.749 (WDR) \\
\bottomrule
\end{tabular}
\end{table}

Table~\ref{tab:functional_functional_sdr_cot} shows that SDR-COT is essentially tied with the best weak conditional-moment methods in the one-index models and is substantially stronger in the two-index models. The sparse $11$-point design increases finite-sample variability relative to denser grids, especially for the quadratic model FF2, but the mean multiple correlations remain close to the ideal upper bounds of $1$ for FF1--FF2 and $2$ for FF3--FF4. This supports the Hilbert-valued SDR-COT construction in a setting where both the covariate and the response are random functions. As explained in Example~\ref{ex:conditional_cameron_martin_translation}, the present Gaussian--Sobolev consistency theory covers the translation structures in FF1--FF3 but does not verify the interpolation-compression condition for the covariance rescaling in FF4. The FF4 result is therefore empirical evidence beyond the scope of the current functional theorem.

%% file: real_data_analysis.tex
\subsection{Bikeshare data analysis}
\label{sec:bikeshare_analysis}

We next compare SDR-COT with the functional dimension-reduction methods of \textcite{li2022weakfunctional} on the Capital Bikeshare data. Each observational unit is one day. The functional covariate $X_i(t)$ is the temperature over the 24 hours of the day, and the functional response is
\[
    Y_i(t)=\log(\operatorname{count}_i(t)+1),
\]
where $\operatorname{count}_i(t)$ is the total number of casual and registered rentals in hour $t$. Following the analysis in \parencite[Section 11]{li2022weakfunctional}, we retain only non-working days, including weekends and holidays, and require all 24 hourly observations to be present. This gives 100 complete days from 2011, used for training, and 112 complete days from 2012, used only for testing.

The temperature and transformed rental-count observations are reconstructed on a common hourly grid using cubic B-splines. The temperature smoother has 15 basis functions, with its second-difference penalty selected by pooled GCV on the training curves. The response is represented by a 24-basis interpolating cubic spline, so that the MSE remains on the observed log-count scale. All functional inner products used by the competing estimators are evaluated by the trapezium rule on this grid.

We consider dimensions $d=1,\ldots,7$. The competitors are functional principal component analysis (FPCA), the weak inverse regression estimator (WIRE), the weak average variance estimator (WAVE), and weak directional regression (WDR). Their sample operators follow the coordinate algorithms in the supplement to \textcite{li2022weakfunctional}. For WIRE, WAVE and WDR, the predictor and response kernels are Gaussian radial basis kernels, with scales $\epsilon_X=0.0251$ and $\epsilon_Y=0.0126$ chosen by the pairwise-distance rule in that paper. 

For SDR-COT, the smoothed temperature curves are represented by their first eight weighted FPCA scores, which explain more than 99.98\% of the training variation. We use the function-on-function procedure in Algorithm~\ref{alg:functional_functional_sdr_cot}, with the functional velocity represented by a one-dimensional FNO\@. During COT fitting, the response curves are centred by their training mean curve and divided by the training root-mean-square scale.\footnote{Specifically, for $n_{\mathrm{tr}}=100$ training curves observed at $m=24$ grid points, let $\overline{Y}_{\mathrm{tr}}(t_j)=n_{\mathrm{tr}}^{-1}\sum_{i=1}^{n_{\mathrm{tr}}}Y_i(t_j)$. The scale is
\[
    s_Y=\left\{\frac{1}{n_{\mathrm{tr}}m}
    \sum_{i=1}^{n_{\mathrm{tr}}}\sum_{j=1}^{m}
    \left[Y_i(t_j)-\overline{Y}_{\mathrm{tr}}(t_j)\right]^2\right\}^{1/2}.
\]} The Gaussian reference covariance uses a 10\% isotropic shrinkage of the empirical response covariance. The FNO has width 20, 10 Fourier modes and three spectral layers, and is trained for 800 Adam steps with batch size 256 and learning rate 0.002.  

After dimension reduction, every method is evaluated through the same functional-response kernel regression. Each coordinate of the estimated sufficient predictor is standardised using the training sample, and the Gaussian Nadaraya--Watson estimator is
\[
    \widehat Y(u)(t)
    =\frac{\sum_{i=1}^{100}\exp\{-\lVert U_i-u\rVert^2/(2h^2)\}Y_i(t)}
    {\sum_{i=1}^{100}\exp\{-\lVert U_i-u\rVert^2/(2h^2)\}}.
\]
The bandwidth $h$ is selected separately for each method and dimension by leave-one-out MSE on the 2011 sample. As in \textcite{li2022weakfunctional}, the reported error averages over subjects and the 24 hourly evaluations,
\[
    \operatorname{MSE}
    =\frac{1}{24n}\sum_{i=1}^{n}\sum_{j=1}^{24}
    \{Y_i(t_j)-\widehat Y(U_i)(t_j)\}^{2}.
\]

\begin{table}[H]
\centering
\small
\setlength{\tabcolsep}{4pt}
\caption{Training and test MSEs for the Bikeshare functional-response analysis. Smaller values are better; boldface marks the smallest unrounded entry within each data set and dimension. All tuning is based on the 2011 training sample.}
\label{tab:bikeshare_mse}
\begin{tabular}{llccccccc}
\toprule
Data & Method & $d=1$ & $d=2$ & $d=3$ & $d=4$ & $d=5$ & $d=6$ & $d=7$ \\
\midrule
Training & SDR-COT & 0.185 & \textbf{0.169} & \textbf{0.157} & \textbf{0.128} & \textbf{0.113} & 0.106 & \textbf{0.089} \\
         & FPCA    & 0.191 & 0.188 & 0.168 & 0.155 & 0.152 & 0.115 & 0.092 \\
         & WIRE    & 0.186 & 0.184 & 0.169 & 0.154 & 0.123 & \textbf{0.096} & 0.096 \\
         & WAVE    & 0.186 & 0.182 & 0.175 & 0.145 & 0.120 & 0.103 & 0.103 \\
         & WDR     & \textbf{0.183} & 0.185 & 0.167 & 0.146 & 0.121 & 0.106 & 0.109 \\
\midrule
Test     & SDR-COT & \textbf{0.320} & 0.311 & \textbf{0.299} & \textbf{0.306} & \textbf{0.303} & 0.314 & 0.314 \\
         & FPCA    & 0.323 & \textbf{0.311} & 0.325 & 0.326 & 0.325 & 0.318 & 0.315 \\
         & WIRE    & 0.322 & 0.312 & 0.328 & 0.322 & 0.313 & \textbf{0.313} & 0.308 \\
         & WAVE    & 0.322 & 0.316 & 0.305 & 0.316 & 0.312 & 0.314 & \textbf{0.305} \\
         & WDR     & 0.322 & 0.312 & 0.318 & 0.314 & 0.314 & 0.317 & 0.312 \\
\bottomrule
\end{tabular}
\end{table}

Table~\ref{tab:bikeshare_mse} shows that SDR-COT is competitive throughout the range of dimensions. The WIRE, WAVE and WDR test errors are on the same scale as, and generally slightly below, the corresponding entries in Table~3 of \textcite{li2022weakfunctional}. An exact numerical reproduction is not possible from the published specification because the B-spline basis and penalty and the kernel-regression bandwidth convention are not reported.

%% file: gaussian_analysis_background.tex
\section{Gaussian measures and Malliavin calculus on Hilbert spaces}
\label{app:gaussian_analysis_background}

This appendix develops the Gaussian-analysis background used in the consistency theory for functional covariates and responses in Section~\ref{sec:functional_linear_consistency}. The key identification step there is to regard the conditional Monge interpolations as solutions of a common Gaussian continuity equation and use uniqueness to show that their conditional response laws agree. This avoids imposing continuity of the velocity in the ambient response norm, but requires a centred Gaussian source, Cameron--Martin weak differentiability and \(L^r\)-controlled densities along the interpolation. Gaussian measures and covariance operators were introduced in Section~\ref{sec:preliminaries}; here we develop only the additional structure needed for that argument. Throughout, \(\mathcal{Y}\) is a separable Hilbert space.

Subsections~\ref{app:gaussian_geometry}--\ref{app:malliavin_calculus} introduce the Cameron--Martin geometry, Gaussian translations, Malliavin derivatives and Gaussian divergence. Subsection~\ref{app:gaussian_continuity} then combines these ingredients into the compression and uniqueness statements applied in Lemma~\ref{lem:gaussian_sobolev_zero_loss}.

\subsection{Gaussian measures and Cameron–Martin geometry}
\label{app:gaussian_geometry}

\begin{Definition}[Centred Gaussian measure]
\label{def:centred_gaussian_measure}
Let \(\rho\) be a Gaussian measure on \(\mathcal{Y}\) in the sense of Definition~\ref{def:gaussian_null}, and let \(Z_\rho\) be the canonical random element \(Z_\rho(y)=y\) on \((\mathcal{Y},\mathcal{B}(\mathcal{Y}),\rho)\). By Fernique's theorem, as recalled after Definition~\ref{def:gaussian_null}, \(Z_\rho\in L^2(\rho;\mathcal{Y})\). Using the Bochner mean and covariance operator from Subsection~\ref{sec:notation}, write
\[
    m_\rho\coloneqq\mathbb{E}_\rho[Z_\rho],
    \qquad
    C_\rho\coloneqq\Gamma_{Z_\rho}.
\]
The measure \(\rho\) is \emph{centred} if \(m_\rho=0\), and we then write \(\rho=\mathcal{N}(0,C_\rho)\). It is nondegenerate in the sense of Definition~\ref{def:gaussian_null} if and only if \(\ker(C_\rho)=\{0\}\).\footnote{The kernel is defined as $\ker(C_{\rho})\coloneqq\{y\in\mathcal{Y}:C_{\rho}y=0\}$.}
\end{Definition}

\begin{Definition}[Cameron--Martin space]
\label{def:cameron_martin_space}
Let \(\rho=\mathcal{N}(0,C_\rho)\) be nondegenerate. Its Cameron--Martin space is
\[
    \mathcal{H}_\rho\coloneqq\operatorname{ran}(C_\rho^{1/2}),
\]
equipped with the inner product
\begin{equation}
    \langle h,k\rangle_{\mathcal{H}_\rho}
    \coloneqq
    \langle C_\rho^{-1/2}h,C_\rho^{-1/2}k\rangle_{\mathcal{Y}},
    \qquad h,k\in\mathcal{H}_\rho.
    \label{eq:appendix_cameron_martin_inner_product}
\end{equation}
Here \(C_\rho^{-1/2}\) denotes the inverse of \(C_\rho^{1/2}\) on its range.
\end{Definition}

\begin{Proposition}
\label{prop:cameron_martin_geometry}
Let \(\rho=\mathcal{N}(0,C_\rho)\) be nondegenerate. The space \(\mathcal{H}_\rho\), endowed with \eqref{eq:appendix_cameron_martin_inner_product}, is a separable Hilbert space that is densely, continuously and compactly embedded in \(\mathcal{Y}\). If \(\mathcal{Y}\) is infinite-dimensional, then \(\rho(\mathcal{H}_\rho)=0\).
\end{Proposition}

\begin{proof}
The covariance operator \(C_\rho\) is positive, self-adjoint and trace class. Nondegeneracy gives \(\ker(C_\rho)=\{0\}\). The map $C^{1/2}_{\rho}:\mathcal{Y}\to\mathcal{H}_{\rho}$ is therefore a surjective isometry when its codomain carries the Cameron--Martin norm. It follows that \(\mathcal{H}_\rho\) is a separable Hilbert space. Moreover,
\[
    \|h\|_{\mathcal{Y}}
    \leq\|C_\rho^{1/2}\|_{\mathrm{op}}
    \|h\|_{\mathcal{H}_\rho},
    \qquad h\in\mathcal{H}_\rho,
\]
so the embedding into \(\mathcal{Y}\) is continuous. Under the preceding isometric identification, this embedding is represented by \(C_\rho^{1/2}\), which is compact because \(C_\rho\) is trace class. Its range is dense since
\[
    \overline{\operatorname{ran}(C_\rho^{1/2})}
    =\ker(C_\rho^{1/2})^\perp=\mathcal{Y}.
\]

Suppose now that \(\mathcal{Y}\) is infinite-dimensional. By the spectral theorem, there is an orthonormal basis \(\{e_j\}_{j\geq1}\) such that \(C_\rho e_j=\lambda_j e_j\), with \(\lambda_j>0\). Under \(\rho\), the variables
\[
    \xi_j(y)\coloneqq\lambda_j^{-1/2}
    \langle y,e_j\rangle_{\mathcal{Y}}
\]
are independent standard Gaussian random variables. The spectral description of the range gives
\[
    y\in\mathcal{H}_\rho
    \quad\Longleftrightarrow\quad
    \sum_{j=1}^{\infty}\frac{|\langle y,e_j\rangle_{\mathcal{Y}}|^2}{\lambda_j}
    =\sum_{j=1}^{\infty}\xi_j(y)^2<\infty.
\]
The strong law of large numbers yields
\(m^{-1}\sum_{j=1}^m\xi_j^2\to1\) almost surely, so the series diverges almost surely. Hence \(\rho(\mathcal{H}_\rho)=0\).
\end{proof}

\begin{Definition}[Abstract Wiener space]
\label{def:abstract_wiener_space}
An abstract Wiener space is a triple \((\mathcal{E},\mathcal{H},\mu)\) in which \(\mathcal{E}\) is a separable Banach space, \(\mu\) is a centred nondegenerate Gaussian measure on \(\mathcal{E}\), and \(\mathcal{H}\) is its Cameron--Martin space, continuously and densely embedded in \(\mathcal{E}\). Thus \((\mathcal{Y},\mathcal{H}_\rho,\rho)\) is an abstract Wiener space. In the present Hilbert setting the embedding is also compact by Proposition~\ref{prop:cameron_martin_geometry}.
\end{Definition}

\begin{Example}[Finite-dimensional comparison]
\label{ex:finite_dimensional_cameron_martin}
If \(\mathcal{Y}=\mathbb{R}^m\) and \(C_\rho\) is positive definite, then \(\mathcal{H}_\rho=\mathbb{R}^m\) as a set, and their inner products are related by \(\langle h,k\rangle_{\mathcal{H}_\rho}=\langle C_\rho^{-1}h,k\rangle_{\mathbb{R}^m}\). Thus, Cameron--Martin differentiability is simply ordinary differentiability expressed in a covariance-weighted geometry. 

In infinite dimensions, by contrast, the trace-class requirement forces the eigenvalues of \(C_\rho\) to decay to zero, making \(C_\rho^{-1/2}\) an unbounded operator. Consequently, \(\mathcal{H}_\rho\) is a strictly smaller, \(\rho\)-null subspace of \(\mathcal{Y}\). Cameron--Martin differentiability only controls directions in this subspace, rather than every ambient direction, and therefore need not imply continuity in the ambient topology.
\end{Example}

To operate within this geometry, we require a mechanism to evaluate the inner product between a fixed Cameron–Martin direction $h\in\mathcal{H}_{\rho}$ and a random element $y$ drawn from $\rho$. Because $y$ almost surely falls outside $\mathcal{H}_{\rho}$, a direct covariance-weighted dot product diverges. We therefore require the Gaussian linear functional (often called the Paley–Wiener integral), which acts as a ``stochastic dot product'' defined via an $L^2(\rho)$-limit. Let \(C_\rho e_j=\lambda_j e_j\), where \(\{e_j\}_{j\geq 1}\) is an orthonormal basis of \(\mathcal{Y}\) and \(\lambda_j>0\). For \(h\in\mathcal{H}_\rho\), write \(h_j\coloneqq \langle h,e_j\rangle_{\mathcal{Y}}\). Then \(\sum_jh_j^2/\lambda_j<\infty\).

\begin{Definition}[Gaussian linear functional]
\label{def:gaussian_linear_functional}
For \(h\in\mathcal{H}_\rho\) and \(m\in\mathbb{N}\), define the finite-dimensional random variable on the canonical probability space \((\mathcal{Y},\mathcal{B}(\mathcal{Y}),\rho)\) by
\begin{equation}
    S_m^h(y)\coloneqq
    \sum_{j=1}^m\frac{h_j}{\lambda_j}
    \langle y,e_j\rangle_{\mathcal{Y}}.
    \label{eq:gaussian_linear_functional}
\end{equation}
For \(n>m\), the covariance identity gives \footnote{In particular,
\[
   \Vert S^{h}_{n}-S^{h}_{m}\Vert^{2}_{L^{2}(\rho)}=\sum^{n}_{i=m+1}\sum^{n}_{j=m+1}\frac{h_{i}h_{j}}{\lambda_{i}\lambda_{j}}\mathbb{E}_{\rho}\bigl[\langle y,e_{j}\rangle_{\mathcal{Y}}\langle y,e_{i}\rangle_{\mathcal{Y}}\big]=\sum^{n}_{i=m+1}\sum^{n}_{j=m+1}\frac{h_{i}h_{j}}{\lambda_{i}\lambda_{j}}\langle C_{\rho}e_{i},e_{j}\rangle_{\mathcal{Y}}=\sum^{n}_{j=m+1}\frac{h^{2}_{j}}{\lambda_{j}}.
\]
}
\[
    \|S_n^h-S_m^h\|_{L^2(\rho)}^2
    =\sum_{j=m+1}^n\frac{h_j^2}{\lambda_j}.
\]
Since \(h\in\mathcal{H}_\rho\), the series is finite, and hence \((S_m^h)_{m\geq1}\) is Cauchy in \(L^2(\rho)\). The \emph{Gaussian linear functional} associated with \(h\) is the unique element \(\widehat h\in L^2(\rho)\) satisfying
\begin{equation}
    \lim_{m\to\infty}\mathbb{E}_{\rho}\left[|S^{h}_{m}(y)-\widehat{h}(y)|^{2}\right]=0.
\label{eq:gaussian_linear_functional_mean_square}
\end{equation}

Equivalently, this functional can be defined by continuous extension without using an eigenbasis. On the dense subspace \(\operatorname{ran}(C_\rho)\subset\mathcal{H}_\rho\), any element takes the form \(h=C_\rho a\) for some \(a\in\mathcal{Y}\). For such elements, the functional is the \(L^2(\rho)\)-class of the continuous linear functional
\[
    \widehat h(y)=\langle a,y\rangle_{\mathcal{Y}}
    =\langle C_\rho^{-1}h,y\rangle_{\mathcal{Y}}.
\]
The map \(h\mapsto\widehat h\) is a linear isometry from \(\operatorname{ran}(C_\rho)\) into \(L^2(\rho)\). Because \(\operatorname{ran}(C_\rho)\) is dense in \(\mathcal{H}_\rho\), it extends uniquely to an isometry on all of \(\mathcal{H}_\rho\). This characterisation is independent of the chosen covariance eigenbasis and gives
\[
    \mathbb{E}_{\rho}\bigl[\widehat{h}(y)\widehat{k}(y)\bigr]=\langle h,k\rangle_{\mathcal{H}_{\rho}},\qquad h,k\in\mathcal{H}_{\rho}.
\]
In particular, as an \(L^2\)-limit of centred Gaussian random variables, \(\widehat h\) is centred Gaussian with variance \(\|h\|_{\mathcal{H}_\rho}^2\).

\end{Definition}

\subsection{Translation and equivalence of Gaussian measures}
\label{app:gaussian_translation}

The structures defined above completely characterise how a Gaussian measure transforms under shifts and linear transformations. Specifically, the Cameron–Martin formula provides the explicit Radon–Nikodym derivative between a Gaussian measure and its shifted counterpart. The result is standard; see, for example, \textcite[Proposition 2.26]{da2014stochastic}.

\begin{Proposition}[Cameron--Martin formula]
\label{prop:cameron_martin_translation_formula}
For \(h\in\mathcal{Y}\), let \(\tau_h(y)\coloneqq y+h\) be the translation. The measures \((\tau_h)_\#\rho\) and \(\rho\) are equivalent if and only if \(h\in\mathcal{H}_\rho\); they are mutually singular otherwise. If \(h\in\mathcal{H}_\rho\), then
\begin{equation}
    \frac{d(\tau_{h})_{\#}\rho}{d\rho}(y)=\exp\left(\widehat{h}(y)-\frac{1}{2}\|h\|^{2}_{\mathcal{H}_{\rho}}\right)\qquad\rho\text{-a.e.}
\label{eq:appendix_cameron_martin_formula}
\end{equation}
Consequently, for \(a>1\) and \(t\in\mathbb{R}\),
\begin{equation}
    \left\Vert \frac{d(\tau_{th})_{\#}\rho}{d\rho}\right\Vert _{L^{a}(\rho)}=\exp\left(\frac{a-1}{2}t^{2}\|h\|^{2}_{\mathcal{H}_{\rho}}\right).
\label{eq:appendix_translation_density_norm}
\end{equation}
\end{Proposition}

The following Feldman--H\'ajek theorem provides necessary and sufficient conditions under which two Gaussian measures are equivalent; see \textcite[Theorem~2.25]{da2014stochastic}. It explains the limitation of covariance-rescaling examples discussed in Section~\ref{sec:functional_linear_consistency}.

\begin{Proposition}[Feldman--H\'ajek theorem]
\label{prop:feldman_hajek_criterion}
Let \(\rho_i=\mathcal{N}(m_i,C_i)\), \(i=1,2\), be nondegenerate Gaussian measures on \(\mathcal{Y}\), and set \(\mathcal{H}_i=\operatorname{ran}(C_i^{1/2})\). The measures \(\rho_1\) and \(\rho_2\) are either equivalent or mutually singular. They are equivalent if and only if the following conditions hold:
\begin{enumerate}[label=\roman*)] 
    \item $\mathcal{H}_1=\mathcal{H}_2$.
    \item $m_1-m_2\in\mathcal{H}_1$.
    \item Let \(L\coloneqq C_1^{-1/2}C_2^{1/2}\) be  understood as its bounded extension under condition~i), the operator \(LL^*-I\) belongs to \(\mathcal{S}_{2}(\mathcal{Y})\).\footnote{The space $\mathcal{S}_{2}(\mathcal{E})$ of Hilbert--Schmidt operators on \(\mathcal{E}\) is a Hilbert space with norm $\Vert A\Vert^{2}_{\mathcal{S}_{2}(\mathcal{E})}=\sum^{\infty}_{j=1}\Vert Ae_{j}\Vert^{2}_{\mathcal{E}}$. Moreover, $\mathcal{S}_{1}(\mathcal{E})\subset\mathcal{S}_{2}(\mathcal{E})$.}
\end{enumerate}
\end{Proposition}

The Feldman--H\'ajek theorem implies that, if \(\mathcal{Y}\) is infinite-dimensional, \(a>0\), and \(a\neq1\), then \(\mathcal{N}(0,a^2C_1)\) and \(\mathcal{N}(0,C_1)\) are mutually singular. Specifically, $LL^{*}-I=(a^{2}-1)I$, and 
\[
   \Vert(a^{2}-1)I\Vert^{2}_{\mathcal{S}_{2}(\mathcal{Y})}=\sum^{\infty}_{j=1}\Vert(a^{2}-1)Ie_{j}\Vert^{2}_{\mathcal{Y}}=\sum^{\infty}_{j=1}(a^{2}-1)^{2},
\]
which diverges.

\begin{Proposition}[Cameron--Martin space of the Wiener source]
\label{prop:wiener_source_cameron_martin_space}
Let \(W\) be a standard Wiener process, let \(\sigma>0\). Denote the law of $\sigma W$ by $\rho$, and regard it as a Gaussian measure on \(L^2([0,1])\). Then
\(\mathcal{H}_\rho\) consists precisely of the elements of \(L^2([0,1])\) having an absolutely continuous representative, still denoted by \(h\), such that \(h(0)=0\) and \(h'\in L^2([0,1])\). For every such \(h\),
\begin{equation}
    \|h\|_{\mathcal{H}_\rho}^2
    =\sigma^{-2}\int_0^1|h'(t)|^2\,dt.
    \label{eq:appendix_wiener_cameron_martin_norm}
\end{equation}
\end{Proposition}

See \textcite[Chapter~2]{da2014stochastic}. Although the measure is here viewed on \(L^2([0,1])\), the Cameron--Martin space consists of the same absolutely continuous directions as for Wiener measure on the path space.

\subsection{Malliavin calculus, Gaussian Sobolev spaces, and divergence}
\label{app:malliavin_calculus}

Because infinite-dimensional spaces lack a standard translation-invariant Lebesgue measure, differential operators cannot be defined globally in the ambient topology. Instead, derivatives are defined on a dense core of ``cylindrical'' functions that depend only on finitely many standardised Gaussian coordinates, and then close the derivatives in the appropriate Sobolev topology.

Let $q_j\coloneqq\lambda_j^{1/2}e_j$, $j\geq 1$, then \(\{q_{j}\}_{j\geq1}\) is an orthonormal basis of \(\mathcal{H}_\rho\), and the Gaussian linear functional associated with \(q_j\) satisfies
\begin{equation}
    \widehat{q_j}(y)=\lambda_j^{-1/2}
    \langle y,e_j\rangle_{\mathcal{Y}}.
    \label{eq:appendix_standardised_gaussian_coordinate}
\end{equation}

Definition~\ref{def:smooth_cylindrical_test_functions} introduced compactly supported, time-dependent cylindrical test functions for the ambient continuity equation. The analysis of Gaussian measures relies on this same finite-coordinate principle but requires a different core class of functions. Specifically, the test functions defined below are time-independent, constructed from bounded, smooth base functions, and expressed in the standardised coordinates determined by \(\rho\). We therefore distinguish the two classes notationally.

\begin{Definition}[Gaussian smooth cylindrical test function]
\label{def:smooth_gaussian_cylindrical_function}
A function \(\varphi:\mathcal{Y}\to\mathbb{R}\) is a Gaussian smooth cylindrical test function if there are indices \(j_1,\ldots,j_m\) and a function \(\psi\in C_b^\infty(\mathbb{R}^m)\) such that
\begin{equation}
    \varphi(y)=\psi\bigl(\widehat{q_{j_1}}(y),\ldots,
    \widehat{q_{j_m}}(y)\bigr).
    \label{eq:appendix_cylindrical_function}
\end{equation}
The collection of such functions is denoted by \(\operatorname{Cyl}_{\rho}^{\infty}(\mathcal{Y})\).
\end{Definition}

With this smooth, finite-dimensional core established, we can define a directional derivative. Because the cylindrical function behaves locally like a Euclidean function, its derivative is defined directly via the classical chain rule along the admissible directions of the Cameron–Martin space.

\begin{Definition}[Malliavin derivative]
\label{def:malliavin_derivative}
For \(\varphi\) as in \eqref{eq:appendix_cylindrical_function}, its Malliavin derivative (or Cameron--Martin gradient) is
\begin{equation}
    D_\rho\varphi(y)
    =\sum_{k=1}^m
    \partial_k\psi\bigl(\widehat{q_{j_1}}(y),\ldots,
    \widehat{q_{j_m}}(y)\bigr)q_{j_k}.
    \label{eq:appendix_malliavin_derivative}
\end{equation}
It is characterised by
\[
    \partial_h\varphi(y)
    =\langle D_\rho\varphi(y),h\rangle_{\mathcal{H}_\rho},
    \qquad h\in\mathcal{H}_\rho.
\]
A cylindrical \(\mathcal{H}_\rho\)-valued field has the form
\[
    b(y)=\sum_{k=1}^m\varphi_k(y)g_k,
    \qquad
    \varphi_k\in\operatorname{Cyl}_{\rho}^{\infty}(\mathcal{Y}),\quad
    g_k\in\mathcal{H}_\rho,
\]
and its Malliavin derivative is the finite-rank operator
\begin{equation}
    D_\rho b(y)[h]
    =\sum_{k=1}^m
    \langle D_\rho\varphi_k(y),h\rangle_{\mathcal{H}_\rho}g_k,
    \qquad h\in\mathcal{H}_\rho.
    \label{eq:appendix_vector_malliavin_derivative}
\end{equation}
\end{Definition}

Cylindrical functions form only a dense subset of the function spaces we typically study. To extend this derivative to a broader, complete class of functions, the differential operator must behave well under limits, a technical property known as closability. This closability is underpinned by a Gaussian integration-by-parts formula.

\begin{Proposition}[Closability and Gaussian integration by parts]
\label{prop:gaussian_derivative_closability}
For every \(1<q<\infty\), the operator \(D_\rho\) on Gaussian smooth cylindrical test functions and cylindrical \(\mathcal{H}_\rho\)-valued fields is closable in the corresponding \(L^q(\rho)\) spaces. Moreover, for \(h\in\mathcal{H}_\rho\) and \(\varphi\in\operatorname{Cyl}_{\rho}^{\infty}(\mathcal{Y})\),
\begin{equation}
    \int_{\mathcal{Y}}
    \langle D_\rho\varphi(y),h\rangle_{\mathcal{H}_\rho}\,\rho(dy)
    =\int_{\mathcal{Y}}\varphi(y)\widehat h(y)\,\rho(dy).
    \label{eq:appendix_gaussian_integration_by_parts}
\end{equation}
\end{Proposition}

See \textcite[Section~2]{AmbrosioFigalli2009}. Closability makes the following definition independent of the chosen cylindrical approximation.

\begin{Definition}[Gaussian Sobolev space]
\label{def:gaussian_sobolev_space}
For \(1<q<\infty\), the Gaussian Sobolev space
\(W_\rho^{1,q}(\rho;\mathcal{H}_\rho)\) is the domain of the closure of \(D_\rho\) on \(L^q(\rho;\mathcal{H}_\rho)\), equipped with the norm
\begin{equation}
    \|b\|_{W_\rho^{1,q}(\rho;\mathcal{H}_\rho)}
    =\left(
    \|b\|_{L^q(\rho;\mathcal{H}_\rho)}^q
    +\|D_\rho b\|_{L^q(\rho;\mathcal{S}_{2}(\mathcal{H}_\rho))}^q
    \right)^{1/q}.
    \label{eq:appendix_gaussian_sobolev_norm}
\end{equation}
\end{Definition}

The derivative \(D_\rho b\) records weak variation along Cameron--Martin directions. Since $\rho(\mathcal{H}_{\rho})=0$, membership in this Gaussian Sobolev space does not require or imply continuity in the ambient \(\mathcal{Y}\)-topology.

\begin{Definition}[Gaussian divergence]
\label{def:gaussian_divergence}
An \(\mathcal{H}_\rho\)-valued field \(b\in L^1(\rho;\mathcal{H}_\rho)\) belongs to the domain of the Gaussian divergence if there is a function \(g\in L^1(\rho)\) such that
\begin{equation}
    \int_{\mathcal{Y}}
    \langle D_\rho\varphi(y),b(y)\rangle_{\mathcal{H}_\rho}\,\rho(dy)
    =-\int_{\mathcal{Y}}\varphi(y)g(y)\,\rho(dy)
    \label{eq:appendix_gaussian_divergence}
\end{equation}
for every \(\varphi\in\operatorname{Cyl}_{\rho}^{\infty}(\mathcal{Y})\). The function \(g\), which is unique in \(L^1(\rho)\), is denoted by \(\operatorname{div}_\rho b\). With this negative-adjoint convention, \(\operatorname{div}_\rho b\) is the negative of the usual Skorokhod divergence \(\delta(b)\), also called the Skorokhod integral.
\end{Definition}

As in classical vector calculus, the divergence is the negative adjoint of the gradient. Combining Definitions~\ref{def:gaussian_linear_functional} and~\ref{def:gaussian_divergence} with \eqref{eq:appendix_gaussian_integration_by_parts} shows that a constant Cameron--Martin field satisfies
\begin{equation}
    \operatorname{div}_\rho h=-\widehat h,
    \qquad h\in\mathcal{H}_\rho.
\label{eq:appendix_constant_gaussian_divergence}
\end{equation}

\subsection{Compression and the Gaussian continuity equation}
\label{app:gaussian_continuity}

The functional identification argument must compare, for covariate values having the same value of a proposed reduction, several conditional Monge interpolations driven by the same candidate velocity. Since there is no infinite-dimensional analogue of Lebesgue measure, these evolving measures are compared with the fixed Gaussian source \(\rho\). Compression provides the required control: it represents each interpolation measure by a density relative to \(\rho\) and bounds that density in the space in which uniqueness of the continuity equation holds. The bound also supplies the H\"older integrability needed to pair a density with a Sobolev velocity.

\begin{Definition}[\(L^r\) compression relative to \(\rho\)]
\label{def:gaussian_compression}
Let \((\mu_t)_{0<t<T}\) be probability measures on \(\mathcal{Y}\), and let \(1<r<\infty\). The family of measures has \emph{\(L^r\) compression relative to \(\rho\)} on \((0,T)\) if
\begin{equation}
    \mu_t=f_t\rho\quad\text{for a.e. }t,
    \qquad
    f\in L^\infty\bigl((0,T);L^r(\rho)\bigr).
    \label{eq:appendix_gaussian_compression}
\end{equation}
\end{Definition}
This is a one-sided condition: it prevents excessive concentration relative to \(\rho\), but does not require \(\rho\ll\mu_t\) and therefore does not assert quasi-invariance.\footnote{A flow is quasi-invariant if the pushed measure $\mu_t$ and the original measure $\rho$ are mutually absolutely continuous.} In Section~\ref{sec:functional_linear_consistency}, Assumption~\ref{ass:gaussian_interpolation_compression} imposes precisely this condition on each truncated conditional Monge interpolation.

Once \(\mu_t=f_t\rho\), its evolution under a Cameron--Martin-valued velocity can be expressed as an equation for \(f_t\) relative to the fixed reference measure. The Malliavin derivative supplies the appropriate directional gradient, while \(\operatorname{div}_\rho\) incorporates the Gaussian reference measure into the divergence. The following weak formulation requires neither pointwise differentiability of \(f_t\) nor ambient continuity of the velocity.

\begin{Definition}[Weak Gaussian continuity equation]
\label{def:weak_gaussian_continuity_equation}
Let \(b:(0,T)\times\mathcal{Y}\to\mathcal{H}_\rho\) be Borel. A measurable density \(f:(0,T)\times\mathcal{Y}\to\mathbb{R}\) is a \emph{weak solution} of
\begin{equation}
    \partial_t f_t+\operatorname{div}_\rho(b_tf_t)=0
    \label{eq:appendix_gaussian_continuity_equation}
\end{equation}
if \(f\), \(f\|b\|_{\mathcal{H}_\rho}\in L^1((0,T)\times\mathcal{Y};dt\otimes\rho)\) and, for every \(\varphi\in\operatorname{Cyl}_{\rho}^{\infty}(\mathcal{Y})\), the map
\[
    t\mapsto\int_{\mathcal{Y}}\varphi(y)f_t(y)\,\rho(dy)
\]
has an absolutely continuous representative satisfying
\begin{equation}
    \frac{d}{dt}\int_{\mathcal{Y}}\varphi f_t\,d\rho
    =\int_{\mathcal{Y}}
    \langle D_\rho\varphi,b_t\rangle_{\mathcal{H}_\rho}f_t\,d\rho
    \quad\text{for a.e. }t.
\label{eq:appendix_weak_gaussian_continuity_equation}
\end{equation}
It has initial density \(f_0\in L^1(\rho)\) if
\begin{equation}
    \lim_{t\downarrow0}\int_{\mathcal{Y}}\varphi f_t\,d\rho
    =\int_{\mathcal{Y}}\varphi f_0\,d\rho
    \qquad
    \text{for every }\varphi\in\operatorname{Cyl}_{\rho}^{\infty}(\mathcal{Y}).
    \label{eq:appendix_gaussian_initial_condition}
\end{equation}
\end{Definition}

This formulation isolates the role of uniqueness. If two compressed interpolation families start from \(\rho\) and solve the same equation with the same velocity, uniqueness forces their densities, and hence their measures, to agree. This is the step that converts a common fitted velocity into equality of the conditional response laws in Lemma~\ref{lem:gaussian_sobolev_zero_loss}.

The next proposition is the uniqueness part of \textcite[Theorem~3.1]{AmbrosioFigalli2009}, restated in the notation used here and specialised from an abstract Wiener space to the Hilbert realisation \((\mathcal{Y},\mathcal{H}_\rho,\rho)\). Our assumptions are slightly stronger because they control the full Malliavin derivative \(D_\rho b_t\), whereas Ambrosio and Figalli require only its symmetric part. We record the specialised statement so that the exponents and regularity conditions used in Section~\ref{sec:functional_linear_consistency} can be checked directly.

\begin{Proposition}[Ambrosio--Figalli uniqueness criterion]
\label{prop:gaussian_continuity_equation_uniqueness}
Let \(p,q>1\), let \(p'=p/(p-1)\) and \(q'=q/(q-1)\), and set \(r=p'\vee q'\). Suppose that the Borel field \(b:(0,T)\times\mathcal{Y}\to\mathcal{H}_\rho\) satisfies
\(b_t\in W_\rho^{1,q}(\rho;\mathcal{H}_\rho)\) for a.e. \(t\) and
\begin{align}
    &\int^{T}_{0}\Vert b_{t}\Vert_{L^{p}(\rho;\mathcal{H}_{\rho})}\,dt<\infty,
    \label{eq:appendix_af_velocity_integrability}\\
    &\int^{T}_{0}\left(\Vert D_{\rho}b_{t}\Vert_{L^{q}(\rho;\mathcal{S}_{2}(\mathcal{H}_{\rho}))}+\Vert\operatorname{div}_{\rho}b_{t}\Vert_{L^{q}(\rho)}\right)dt<\infty,
    \label{eq:appendix_af_derivative_integrability}\\
    &\mathop{\operatorname{ess\,sup}}_{t\in(0,T)}\int_{\mathcal{Y}}\exp\left(c[\operatorname{div}_{\rho}b_{t}(y)]^{-}\right)\rho(dy)<\infty
    \label{eq:appendix_af_exponential_divergence}
\end{align}
for some \(c> rT\), where \([a]^-\coloneqq\max\{-a,0\}\). Then \eqref{eq:appendix_gaussian_continuity_equation} has at most one weak solution in
\[
    L^\infty\bigl((0,T);L^r(\rho)\bigr)
\]
for each prescribed initial density \(f_0\in L^r(\rho)\).
\end{Proposition}

\begin{proof}
The triple \((\mathcal{Y},\mathcal{H}_\rho,\rho)\) is an abstract Wiener space by Definition~\ref{def:abstract_wiener_space}, and Definition~\ref{def:weak_gaussian_continuity_equation} is the corresponding weak formulation of the continuity equation. For a.e. \(t\), membership of \(b_t\) in \(W_\rho^{1,q}(\rho;\mathcal{H}_\rho)\) implies that its symmetric weak Malliavin derivative\footnote{The symmetric part of the Malliavin derivative is defined as $(D_{\rho}b)^{\mathrm{sym}}\coloneqq\frac{1}{2}((D_{\rho}b)+(D_{\rho}b)^{*})$.} is Hilbert--Schmidt and satisfies
\[
    \|(D_\rho b_t)^{\mathrm{sym}}\|_{\mathcal{S}_{2}(\mathcal{H}_\rho)}
    \leq\|D_\rho b_t\|_{\mathcal{S}_{2}(\mathcal{H}_\rho)}.
\]
Consequently, \eqref{eq:appendix_af_velocity_integrability}--\eqref{eq:appendix_af_exponential_divergence} imply all the hypotheses of the uniqueness part of \textcite[Theorem~3.1]{AmbrosioFigalli2009}; the threshold there is exactly \(c\geq rT\). Applying that theorem gives the stated uniqueness in \(L^\infty((0,T);L^r(\rho))\).
\end{proof}

To see how the pieces enter the main proof, fix a proposed reduced value \(s\). The zero-loss identity makes every relevant conditional interpolation solve the weak equation above with the same field \(b_t=u(t,s,\cdot)\) and initial density one. Compression places all of their densities in the uniqueness class, so the proposition implies that the interpolation measures are identical for covariate values in the same fibre of \(B^*X\). The terminal conditional laws can then be reconstructed from a common intermediate measure and the common velocity. Thus the proposition supplies the analytic identification step; existence of the interpolations themselves comes from the conditional Monge construction developed earlier in the paper.